\documentclass[12pt, leqno]{amsart}

\usepackage[nopar]{lipsum}
\usepackage{stmaryrd,colonequals}
\usepackage{amsmath, amscd,amsxtra,amssymb,mathrsfs, bbm}

\usepackage[showframe=false,marginparsep=0in,
top=0.9in, 
text={5.85in,9.3in},
left=1.325in, 
bindingoffset=0in,twoside=true]{geometry} %

\usepackage{tikz}
\usepackage{wrapfig}
\usepackage[inline]{enumitem}
\usepackage{graphicx}
\usepackage{float}
\usetikzlibrary{arrows,decorations.pathmorphing,backgrounds,positioning,fit,petri,graphs,automata, intersections,shapes}
\usepackage{upgreek}

\usepackage{caption}

\usepackage{kbordermatrix}

\usepackage{accents}
\newcommand\thickbar[1]{\accentset{\rule{.5em}{1.0pt}}{#1}}

\definecolor{light-gray}{gray}{0.75} 

\DeclareRobustCommand\longtwoheadrightarrow
{\relbar\joinrel\twoheadrightarrow}
\newcommand{\hooklongrightarrow}{\lhook\joinrel\longrightarrow}

\usepackage{enumitem}
\setlist[itemize]{leftmargin=*}
\setlist[enumerate]{leftmargin=*}

\usepackage{etoolbox}
\makeatletter
\patchcmd{\chapter}
{\if@openright\cleardoublepage\else\clearpage\fi}{}{}{}
\makeatother

\makeatletter
\newcommand{\altqedhere}{%
	\ifmeasuring@\else\sbox0{\popQED}\fi
	\tag*{\qedsymbol}%
}
\makeatother

\newenvironment{td}[0]
{\begin{tikzcd}[ampersand replacement=\&, cells={outer sep=2pt, inner sep=2pt}, font=\normalsize]}		
{\end{tikzcd}}

\newcommand{\Tau}{\mathcal{T}}

\newcommand{\CK}{ {K}^{\bt}}
\newcommand{\CB}{ {B}^{\bt}}
\newcommand{\CL}{ {L}^{\bt}}
\newcommand{\CFb}{ {F}^{\bt}}

\newcommand{\CLb}{ {L}^{\bt}}
\newcommand{\CMb}{ {M}^{\bt}}
\newcommand{\CNb}{ {N}^{\bt}}
\newcommand{\CPb}{ {P}^{\bt}}

\newcommand{\CC}{ {C}^{\bt}}
\newcommand{\CF}{ {F}^{\bt}}

\newcommand{\CZ}{ {Z}^{\bt}}
\newcommand{\PCL}{ \ol{\CL} }

\newcommand{\CM}{ {M}^{\bt}}
\newcommand{\CN}{ {N}^{\bt}}	
\newcommand{\PCM}{ \ol{\CM}}

\newcommand{\CP}{ {P}^{\bt}}
\newcommand{\CT}{ {T}^{\bt}}
\newcommand{\PCT}{ \ol{\CT}}
\newcommand{\CQ}{ {Q}^{\bt}}
\newcommand{\CX}{ {X}^{\bt}}
\newcommand{\CY}{ {Y}^{\bt}}

\newcommand{\PCX}{ \ol{\CX}}

\DeclareMathOperator{\teq}{ \;\underline{\triangleright}\, }

\newcommand{\kk}{k}

\newcommand{\mx}{\mathfrak m}
\newcommand{\px}{\mathfrak p}
\newcommand{\Rx}{R}

\newcommand{\Sx}{S}
\newcommand{\Sxx}{\Rx/\ida}
\newcommand{\Tx}{\Rx/\idb}

\newcommand{\A}{\Lambda} %

\newcommand{\B}{\RI} 
\newcommand{\sfN}{N}
\newcommand{\sfM}{A}
\newcommand{\sfL}{B}

\newcommand{\rA}{\A} 
\newcommand{\rAI}{\AI}

\newcommand{\AI}{\ol{\A}} 
\newcommand{\AIx}{{\AI}^{\ida}}
\newcommand{\Gammax}{{\ol{\Gamma}}^{\ida}}
\newcommand{\Abar}{\overline{\A}} 
\newcommand{\AoB}{\A_i} 
\newcommand{\RI}{{\thickbar \Rx}} 
 
\newcommand{\BI}{{\AI}^{\idb}}

\newcommand{\Eta}{\mathrm{E}^{\bt}}
\newcommand{\trEta}{\Eta_{0}}
\newcommand{\trM}{N^{\bt}}
\newcommand{\Cstr}{\mathrm{C}_{\mathsf{str}} (\trEta)}

\newcommand{\trFm}{\Fm_{0}}

\DeclareMathOperator{\cone}{cone}
\newcommand{\Ho}{\mathrm{H}}

\DeclareMathOperator{\D}{D}
\newcommand{\Dfm}{\D^-_{\mathcal{A}}}
\newcommand{\Dbf}{\D^{\mathrm{b}}_{\mathcal{A}}}
\newcommand{\Df}{\D_{\mathcal{A}}}
\newcommand{\Db}{\D^{\mathrm{b}}}

\DeclareMathOperator{\Kom}{C}
\newcommand{\Hotb}{\Hot^{\mathrm{b}}}

\DeclareMathOperator{\im}{im}

\DeclareMathOperator{\End}{End}

\DeclareMathOperator{\Hom}{Hom}
\DeclareMathOperator{\Ext}{Ext}
\DeclareMathOperator{\Tor}{Tor}

\DeclareMathOperator{\id}{id}

\DeclareMathOperator{\proj}{proj}
\DeclareMathOperator{\Proj}{Proj}

\DeclareMathOperator{\Add}{Add}
\DeclareMathOperator{\modd}{\mathsf{mod}}
\DeclareMathOperator{\tilt}{\mathsf{tilt}}

\DeclareMathOperator{\silt}{\mathsf{silt}}
\DeclareMathOperator{\presilt}{\mathsf{presilt}}
\DeclareMathOperator{\pretilt}{\mathsf{pretilt}}

\DeclareMathOperator{\Modd}{Mod}
\DeclareMathOperator{\Bimod}{Bimod}

\newcommand{\md}{
\modd}

\newcommand{\Md}{\Modd}

\newcommand{\dash}{\mathstrut_{\text{\textemdash}}}
\newcommand{\ida}{\mathfrak{a}}
\newcommand{\idb}{\mathfrak{b}}

\usepackage{graphicx}

\makeatletter
\newcommand*\bigcdot{\mathpalette\bigcdot@{.75}}
\newcommand*\bigcdot@[2]{\mathbin{\vcenter{\hbox{\scalebox{#2}{$\m@th#1\bullet$}}}}}
\makeatother

\newcommand*{\setumath}[2]{%
\csname luatexUmath#1\endcsname\displaystyle=#2\relax
\csname luatexUmath#1\endcsname\luatexcrampeddisplaystyle=#2\relax
\csname luatexUmath#1\endcsname\textstyle=#2\relax
\csname luatexUmath#1\endcsname\luatexcrampedtextstyle=#2\relax
\csname luatexUmath#1\endcsname\scriptstyle=#2\relax
\csname luatexUmath#1\endcsname\luatexcrampedscriptstyle=#2\relax
\csname luatexUmath#1\endcsname\scriptscriptstyle=#2\relax
\csname luatexUmath#1\endcsname\luatexcrampedscriptscriptstyle=#2\relax
}

\usepackage{mathtools}
\newcommand*{\myov}[1]{\overbracket[1.25pt][-1.25pt]{#1}}

\makeatletter
\let\save@mathaccent\mathaccent
\newcommand*\if@single[3]{%
\setbox0\hbox{${\mathaccent"0362{#1}}^H$}%
\setbox2\hbox{${\mathaccent"0362{\kern0pt#1}}^H$}%
\ifdim\ht0=\ht2 #3\else #2\fi
}
\newcommand*\rel@kern[1]{\kern#1\dimexpr\macc@kerna}
\newcommand*\wideaccent[2]{\@ifnextchar^{{\wide@accent{#1}{#2}{0}}}{\wide@accent{#1}{#2}{1}}}
\newcommand*\wide@accent[3]{\if@single{#2}{\wide@accent@{#1}{#2}{#3}{1}}{\wide@accent@{#1}{#2}{#3}{2}}}
\newcommand*\wide@accent@[4]{%
\begingroup
\def\mathaccent##1##2{%
	\let\mathaccent\save@mathaccent
	\if#42 \let\macc@nucleus\first@char \fi
	\setbox\z@\hbox{$\macc@style{\macc@nucleus}_{}$}%
	\setbox\tw@\hbox{$\macc@style{\macc@nucleus}{}_{}$}%
	\dimen@\wd\tw@
	\advance\dimen@-\wd\z@
	\divide\dimen@ 3
	\@tempdima\wd\tw@
	\advance\@tempdima-\scriptspace
	\divide\@tempdima 10
	\advance\dimen@-\@tempdima
	\ifdim\dimen@>\z@ \dimen@0pt\fi
	\rel@kern{0.6}\kern-\dimen@
	\if#41
	#1{\rel@kern{-0.6}\kern\dimen@\macc@nucleus\rel@kern{0.4}\kern\dimen@}%
	\advance\dimen@0.4\dimexpr\macc@kerna
	\let\final@kern#3%
	\ifdim\dimen@<\z@ \let\final@kern1\fi
	\if\final@kern1 \kern-\dimen@\fi
	\else
	#1{\rel@kern{-0.6}\kern\dimen@#2}%
	\fi
}%
\macc@depth\@ne
\let\math@bgroup\@empty \let\math@egroup\macc@set@skewchar
\mathsurround\z@ \frozen@everymath{\mathgroup\macc@group\relax}%
\macc@set@skewchar\relax
\let\mathaccentV\macc@nested@a
\if#41
\macc@nested@a\relax111{#2}%
\else
\def\gobble@till@marker##1\endmarker{}%
\futurelet\first@char\gobble@till@marker#2\endmarker
\ifcat\noexpand\first@char A\else
\def\first@char{}%
\fi
\macc@nested@a\relax111{\first@char}%
\fi
\endgroup
}
\makeatother

\newcommand{\bt}{\bigcdot}

\newcommand{\ol}{\myov}

\newcommand{\Acat}{\mathcal{A}}

\newcommand{\Ccat}{\mathcal{C}}

\newcommand{\Lcat}{\mathcal{L}}
\newcommand{\Mcat}{\mathcal{M}}

\newcommand{\Pcat}{\mathcal{P}}

\newcommand{\N}{\mathbb{N}}
\newcommand{\Z}{\mathbb{Z}}

\numberwithin{equation}{section}

\newtheorem{thm}[equation]{Theorem}

\newtheorem{cor}[equation]{Corollary}
\newtheorem{prp}[equation]{Proposition}
\newtheorem{lem}[equation]{Lemma}

\newtheorem{dfn}[equation]{Definition}

\newtheorem{dfn-sp}[equation]{Definition*}
\newtheorem{ex}[equation]{Example}
\newtheorem{ex-sp}[equation]{Example*}
\newtheorem{rmk}[equation]{Remark}
\newtheorem{setup}[equation]{Setup}

\newtheorem{notation}[equation]{Notation}
\newtheorem*{sketch-proof}{Sketch of proof}
\newtheorem*{comment-proof}{Comment on the proof}
\newcounter{ABC}

\newtheorem{intro-thm}[ABC]{Theorem}

\DeclareRobustCommand\longtwoheadrightarrow
{\relbar\joinrel\twoheadrightarrow}

\DeclareMathOperator{\Loc}{\mathsf{Loc}}
\DeclareMathOperator{\per}{\mathsf{per}}

\DeclareMathOperator{\rad}{\mathsf{rad}}

\DeclareMathOperator{\Hot}{\mathsf{K}}

\DeclareMathOperator{\add}{\mathsf{add}}

\DeclareMathOperator{\Spec}{\mathsf{Spec}}

\newcommand{\FF}{\mathbf{F}}
\newcommand{\GG}{\mathbf{G}}
\newcommand{\HH}{\mathbf{H}}
\newcommand{\Fm}{\mathbbm{F}}

\newcommand{\Gm}{\mathbbm{G}}

\usepackage{kbordermatrix}
\usepackage{multirow}
\usepackage[pagebackref=true]{hyperref}
\hypersetup{
	pdfauthor={...},
	pdftitle={...},
	pdfsubject={...},
	urlcolor=blue!75!black,
	colorlinks=true,%
	linkcolor=blue!75!black,%
	citecolor=blue!75!black,%
	filecolor=blue!75!black,%
	menucolor=blue!75!black,%
}

\usepackage[initials, alphabetic, msc-links, nobysame,short-publishers]{amsrefs}
\usepackage{enumitem}
\usepackage{longtable}
\usepackage{tikz}
\usepackage{tikz-cd}
\usepackage{float}
\usepackage{color}
\usepackage{prerex}
\usepackage{helvet}
\usepackage[utf8]{inputenc}
\usetikzlibrary{calc, positioning, shapes, fit, matrix, decorations}
\usetikzlibrary{decorations.shapes}

\usepackage[all,cmtip]{xy}
\xyoption{arc}

\usepackage{ifthen, xkeyval, xcolor, calc, graphicx}
\usepackage[backgroundcolor=pink, colorinlistoftodos, disable]{todonotes}

\makeatletter
\providecommand\@dotsep{5}
\def\listtodoname{List of Todos}
\def\listoftodos{\@starttoc{tdo}\listtodoname}
\makeatother

\author{Wassilij Gnedin}
\title{Silting theory 
	 under change of rings}
 \address{
Faculty of Mathematics,
Ruhr University Bochum, 
Universitätsstraße 150, 
\newline $\mathstrut$\quad\,44780 Bochum, Germany
}
\email{wassilij.gnedin@rub.de}

\tikzset{labl/.style={anchor=south, rotate=90, inner sep=.5mm}}

\begin{document}

\begin{abstract}
	The main goal of this paper is to compare the silting theory of an $R$-algebra $\Lambda$ over a Noetherian ring $R$ with that of its tensor product $\Lambda \otimes \Gamma$ with another $R$-algebra $\Gamma$.
	
	In the case that the $R$-algebra $\Lambda$ is Noetherian, $R$ a complete local ring and $\mathfrak{a}$ a certain ideal of the ring $R$, we obtain an isomorphism between the silting poset of $A$ and that of its quotient $\Lambda \otimes \Rx/\mathfrak{a}$. 
	Furthermore, we study the restrictions of such a bijection to tilting complexes and deduce silting embedding and descent results for the algebra $\Lambda$ and a certain family of algebras $(\Lambda \otimes \Gamma_i)_{i \in I}$.
\end{abstract}

\maketitle

One of the intensively studied problems in representation theory of algebras is to clarify which rings have equivalent derived categories.
This paper is a contribution to this problem guided by the following question.
\begin{quote}
	\emph{How does the derived Morita theory of an $\Rx$-algebra $\A$  differ from that of its tensor product $\A \otimes \Gamma$ with another $R$-algebra $\Gamma$?}
\end{quote}
To obtain stronger answers, we assume that $\A$ is a Noetherian $R$-algebra, that is, $\A$ is finitely generated as an $R$-module, the ring $R$ is complete local, and $\Gamma$ is set to be the quotient $\Rx/\ida$ by a proper ideal $\ida$ of $\Rx$. 
This allows to identify the ring $\A \otimes \Gamma$ with the quotient $\A/\ida \A$.

The following results from Rickard's derived Morita theory were the starting point of this paper.
Any ring which is derived equivalent to $\A$ is an $\Rx$-algebra given by the endomorphism ring of a tilting complex of $\A$ \cite{Rickard89a}. 
If the algebra $\A$ is free as an $\Rx$-module and the ideal $\ida$ is the maximal ideal $\mx$ of the local ring $\Rx$, there is a natural bijection
\[
\begin{td} \tilt^*_{\Rx} \A \ar{r}{\sim} \& \tilt \A/\mx \A \end{td}
\]
from the set of isomorphism classes of basic tilting complexes of $\A$ with $R$-free endomorphism rings to an analogous set given by basic tilting complexes of the finite-dimensional algebra $\A/\mx \A$ \cites{Rickard91a, Rickard91b}.
In a prototype of this setup,
$\A$ is given by the algebra $RG$ of a finite group $G$, and its quotient $\A/\mx \A$
by the algebra $kG$ over the residue field $k$ of the local ring $\Rx$.

The notion of `derived Morita theory' in the question above might be understood in the broader sense that one should consider all \emph{differential graded algebras} which are derived equivalent to $\A$. However, the generators of the perfect derived category of $\A$ are not related to the perfect generators over the quotient $\A/\ida \A$ in a coherent way in general.

In this paper, we will focus on a special sort of perfect generators encompassing all tilting complexes, namely
\emph{silting complexes}, which have been introduced by Keller and Vossieck \cite{Keller-Vossieck}.
Silting theory has attracted much interest in recent years, after
Aihara and Iyama  \cite{Aihara-Iyama} discovered that, in contrast to tilting complexes, any silting complex can be mutated in order to produce new ones, and that
silting complexes give rise to a partially ordered set $(\silt \A, \geq)$.
Work by Koenig and Yang showed that
silting complexes of a finite-dimensional algebra had already appeared in various homological disguises \cite{Koenig-Yang}.
We refer to \cite{Angeleri-Huegel} for a comprehensive survey of recent developments in silting theory by Angeleri H{\"u}gel.

The main result of this paper extends Rickard's tilting bijection in several ways.
\begin{intro-thm}[Corollary~\ref{cor/setups}]\label{thm/main}
	Let $\ida$ be a proper ideal of a complete local ring $\Rx$ and $\A$ be a Noetherian $\Rx$-algebra satisfying any of the following conditions.
\begin{enumerate}[label={\rm (S1\alph*)}, align=left]
	\item \label{setup/orders0} The $\Rx$-algebra $\A$ is free as an $\Rx$-module and the ring $\Rx$ is regular.
	\item \label{setup/groups0} The $\Rx$-algebra $\A$ is free as an $\Rx$-module and the ideal $\ida$ is maximal.	
	\end{enumerate}
\begin{enumerate}[label={\rm (S2)}, align=left]
	\item \label{setup/CI0} The ideal $\ida$ is generated by an $\Rx$- and $\A$-regular sequence.
\end{enumerate}
	Then the left-derived functor 
	$$
	\begin{td}\FF\colon \D^-(\md \A) \ar{r} \& \D^-(\md \A/\ida \A) \& \CL \ar[mapsto]{r} \& \PCL \colonequals \CL \overset{\mathbb{L}}{\underset{\A}{\otimes}} \A/\ida \A 
	\end{td}
	$$
	gives rise to three bijections of sets of isomorphism classes
	\begin{align*}
	\begin{tikzcd}[ampersand replacement=\&, cells={outer sep=2pt, inner sep=1pt}, column sep=0.75cm, row sep=0.5cm]
	{\mathstrut} \&		
	\tilt^*_{\Rx} \A  
	\ar["\sim" labl]{dd}{f^*_t} \ar[hookrightarrow, color=light-gray]{rr} 
	\&  \&
	\tilt^{\Rx/\ida} \A
	\ar["\sim" labl]{dd}{f_t} 
	\ar[hookrightarrow, color=light-gray]{rr}
	\& \&
	\silt \A \ar["\sim" labl,swap]{dd}[pos=0.33]{}[swap]{f_s}
	\\[0.75cm]
	\\
	{\mathstrut} \&
	\tilt^*_{\Rx/\ida} \A/\ida\A 
	\ar[hookrightarrow, color=light-gray]{rr} \& \& 
	\tilt \A/\ida\A
	\ar[hookrightarrow, color=light-gray]{rr}
	\&   \&
	\silt \A/\ida\A
	\end{tikzcd}
	\end{align*}
	where  the map $f_s$ is an isomorphism of posets and the set
	$\tilt^{\Rx/\ida} \A$ 
	can be defined as the set of isomorphism classes of basic tilting complexes $\CT$ of the ring $\A$ such that $\Hom_{\D(\A/\ida\A)}(\PCT,\PCT[-1] ) = 0$.
\end{intro-thm}
Any order over a complete local equicharacteristic Cohen--Macaulay ring gives rise to setup~\ref{setup/orders0}. 
The conditions in~\ref{setup/groups0} are motivated by the representation theory of finite groups.
At last, any non-commutative complete intersection such as the quotient of a Noetherian algebra by a central regular element 
leads to situation~\ref{setup/CI0}.
The conclusions of Theorem~\ref{thm/main} are also true under more general conditions which unify the three setups above (Theorem~\ref{thm/embeddings2}).

The main purpose of Theorem~\ref{thm/main} is to reduce questions about the silting theory of a family of algebras to one particular case with easier representation theory.
In case~\ref{setup/orders0}, we may vary
the proper ideal $\ida$. This 
results in a family of quotients with the same silting theory although any two quotients
are usually not derived equivalent. 
In particular, if  $R$ is the ring $\kk \llbracket x \rrbracket$ of formal power series, varying the ideal $\ida$ 
yields an 
isomorphism of posets
\begin{align*}
\begin{td}
\silt \A/x^m \A \ar{r}{\sim} \& \silt \A/x \A
\end{td}
\end{align*}
for any power $m \in \N$. 
Such a phenomenon of multiplicity independence was previously observed in the silting theory of Brauer graph algebras \cite{Adachi/Aihara/Chan}, \cite{Eisele}.

The absence of a bijection between the sets $\tilt \A$ and $\tilt \A/\ida\A$, as well as the proofs of the main results, suggest that
silting complexes are well-behaved generators which can be more convenient to work with than tilting complexes.

The methods to prove Theorem~\ref{thm/main} 
and their applications yield the following results which might be interesting on their own.
\begin{enumerate}
	\item The functor $\FF$ preserves and reflects certain natural relations $\geq$ and $\perp$ of perfect complexes (see Proposition~\ref{prp/keyA}).
		\item 
Any perfect complex $\CP$ of the quotient $\A/\ida \A$ which does not have second self-extensions lifts to a perfect complex of $\A$
in case $\Rx$ is \emph{normally flat along the ideal $\ida$} (Proposition~\ref{prp/2-rigid-lift}). 
	\item 
In each of the three setups 
the functor $\FF$ reflects the tilting property of a perfect complex. This can be 
explained invoking \emph{Tor-rigidity} 
	of $R/\ida$ as an $R$-module in each case (see Proposition~\ref{prp/descent}).
		\item  
		The question whether the endomorphism ring of 
	a tilting complex $\CT$ of an $R$-free Noetherian algebra $\A$ is also free over the ring $R$  
	can be reduced to the computation of a finite-dimensional morphism space
	(Proposition~\ref{prp/endo-free}).
	 \item If $\A$ is free as an $\Rx$-module, each of its silting complexes $\CL$ may be viewed as the completion obtained from a family of silting complexes  $(\CL_n)_{n\in\N}$ over certain finite-dimensional quotients $\A_n$ (see Corollary~\ref{cor/silt-completion}).
	\item Parts of Theorem~\ref{thm/main} 
	extend to several setups in which $R$ is neither complete nor local and the $R$-algebra $\Gamma$ is not a quotient of $R$ (see Subsection~\ref{subsec:app2}).
	These extensions yield a close relationship 
	between silting subcategories
	of a Noetherian $\Rx$-projective algebra $\A$ 
to those of 
a family of certain finite-dimensional algebras $(\A_{\px,n})_{\px \in \Spec \Rx}$, which is 
stated rigorously in Theorem~\ref{thm/global-to-local}.
\end{enumerate}
The proof of
 Theorem~\ref{thm/main} 
develops some ideas
 by Rickard from \cites{Rickard91a,Rickard91b}.
The proof that $\FF$ reflects silting complexes
 is based on dg-categorical arguments due to Keller \cite{Keller:private-communication}.
 Recent work by Nasseh, Ono and Yoshino \cite{Nasseh-Ono-Yoshino}
inspired one of the main ideas to lift complexes in a setup including~\ref{setup/orders0}.

The results above are also related to several, more recent investigations of the behavior of silting complexes under base-change
by 
 Eisele \cite{Eisele}, and Iyama and Kimura \cite{Iyama-Kimura}.
In the context of commutative dg-algebras, lifting problems have been studied by Nasseh and Sather-Wagstaff \cite{Nasseh/Sather-Wagstaff},  Ono and Yoshino \cite{Ono-Yoshino} and Nasseh, Ono and Yoshino \cite{Nasseh-Ono-Yoshino}.
The results above recover special cases of bijections of two-term silting complexes, which were established by Eisele, Janssens and Raedschelders \cite{EJR} and Kimura \cite{Kimura}.
Further connections to these works are discussed in the course of the paper. 

In forthcoming work, we will focus on compatibility of silting embeddings with mutation, silting bijections for complexes of fixed length and their variations for quotients of Noetherian algebras by a normal regular element.

In the beginning of this paper (Subsections~\ref{sec:main}--\ref{sec:recall}), we fix the general setup and recall all notions which are relevant to state the main results. The latter are gathered in Section~\ref{sec/main}. In Subsection~\ref{subsec/structure} we sketch the proof of the silting bijection and describe the other parts of the paper.

\subsubsection*{Acknowledgment}
This work was started in conversations with Frederik Marks during a visit to Verona.
The author would like to thank Lidia Angeleri Hügel for her hospitality and Frederik Marks for productive discussions.
The author benefited from constructive feedback
 on the contents of this paper
by Steffen Koenig, Sebastian Opper, Markus Reineke and Alexandra Zvonareva.

The author is deeply grateful to Bernhard Keller for 
explaining to him the essential steps to prove Proposition~\ref{prp/Keller}.

Part of the research was carried out during the Junior Trimester Program
``New Trends in Representation Theory'' organized by the Hausdorff Institute of Mathematics in Bonn. The author thanks the HIM for financial support and a stimulating research environment.

\setcounter{tocdepth}{1}
\tableofcontents

\subsubsection*{Conventions}

Any ring is assumed to have a unit.
Unless otherwise stated, by a module we mean a \emph{right} module.
The category of finitely generated $\A$-modules is denoted by $\md \A$.
The derived category of all $\A$-modules of a ring $\A$ is abbreviated with  $\D(\A)$, the homotopy category of complexes of $\A$-modules with $\Hot(\A)$. The full subcategory of perfect complexes in $\D(\A)$, that is, complexes quasi-isomorphic to bounded complexes of finitely generated projective $\A$-modules, is denoted by $\per \A$.

By a complex $\CM$ of modules over a ring we mean a \emph{cochain} complex.  The complex $\CM[1]$ is given by shifting $\CM$ to the \emph{left} and changing the signs of its differentials.

\section{Main setup and recall of silting theory}
\label{sec/first}
The goal of the first part of this section is to fix the central assumptions and to recall all relevant notions to state the main results of this paper.

\subsection{Main setup}\label{sec:main}

Until Section~\ref{sec/main}
we assume the following setup.
\begin{setup}\label{setup/main}
Let $R, \A$, $\RI$ and $\Sx$ be rings satisfying the following conditions:
	\begin{itemize}
		\item 	The ring $\A$ is a \emph{Noetherian} $\Rx$-algebra, that is, the ring $\Rx$ is a commutative Noetherian ring and $\A$ is finitely generated as $\Rx$-module.
		\item Moreover, we assume that the ring $\Rx$ is local and complete with respect to its unique maximal ideal $\mx$. 
		\item Let $\RI \colonequals \Rx/\ida$ be the quotient ring by a proper ideal $\ida$  of $\Rx$
		such that
		the $\Rx$-modules
$\A$ and $\RI$ are \emph{Tor-independent}  which means that
		\begin{align}\label{eq/tor-ind0} 
		\tag{$\star$}
		\Tor_+^{\Rx}(\A,\RI) = 0, \text{ that is, } \Tor_{n}^{\Rx}(\A,\RI) = 0 \text{ for any integer }n>0.
		\end{align}
	\item For later technical reasons, we assume that $\RI$ is also an $\Sx$-algebra over another commutative ring $\Sx$.
	\end{itemize}
		In the following we abbreviate the tensor product $\otimes_{\Rx}$ with $\otimes$ and set 
				 $\AI \colonequals \A/\ida \A$.
The above data gives rise to a commutative diagram of rings and a  functor
$$
\begin{td}
	\&
\Rx \ar{r}{\varrho} \ar[twoheadrightarrow]{d} \& \A \ar[twoheadrightarrow]{d}  \&[-1cm] \& \D^-(\md\A) \ar{d}[swap,font=\normalsize]{\FF} \& \CM \ar[mapsto]{d}  \&[-1cm] \\
\Sx \ar{r} \&\RI \ar{r} \& \AI  \& 
	\cong \A \underset{\Rx}{\otimes} \RI
\& \D^-(\md \AI) \& \PCM \& \colonequals \CM
\underset{\A}{\overset{\mathbb{L}}{\otimes}} \AI\,. 
\end{td}
$$
\end{setup}
By the definition of an \emph{$\Rx$-algebra} the image of the map $\varrho$ lies in the center of the ring $\A$. 
In particular, $\AI$ is the quotient of $\A$ by a two-sided ideal which is generated by \emph{central} elements.
\begin{rmk}\label{rmk/sub}
	Condition~\eqref{eq/tor-ind0} is satisfied in the following cases.
	\begin{enumerate}[label={\rm (S\arabic*)}, align=left]
		\item \label{rmk/sub1} The Noetherian $\Rx$-algebra $\A$ is free as an $\Rx$-module.
		\item \label{rmk/sub2} The ideal $\ida$ is generated by a sequence $\mathbf{x}$ which is $\Rx$- and $\A$-regular.
	\end{enumerate}
	More precisely, if the ideal $\ida$ is generated by an $\Rx$-regular sequence $\mathbf{x}$, then the Koszul homology $\Ho_n(\mathbf{x},\A)$ can be identified with
	$\Tor_n^{\Rx}(\A,\RI)$ for any integer $n > 0$, and the  equivalences 
	\begin{center}
		\begin{tabular}{ccccc}
			$\mathbf{x}$ is $\A$-regular  &
			\quad $\Leftrightarrow$ \quad
			&
			$\Tor_1^{\Rx}(\A,\RI) =0$ &
			\quad	$\Leftrightarrow$ \quad
			&
			$\Tor_+^{\Rx}(\A,\RI) = 0$
		\end{tabular}
	\end{center}
are true.
	We refer to \cite{Bruns-Herzog}*{Corollaries~1.6.14~and~1.6.19} for more details.
\end{rmk}
In setup~\ref{rmk/sub2} we will call $\AI$ a \emph{complete intersection}. This is motivated by existing terminology in case the $\Rx$-algebra $\A$ is commutative.

\begin{ex}\label{ex/preproj1}
		Let $\kk$ be a field
		and $(Q,I)$ the quiver of the preprojective algebra of affine type $\widetilde{\mathbb{A}}_2$ which is shown on the left.
	$$
	\begin{array}{cccc}
		(Q,I) && \qquad \, {\ol{(Q,I)}}^{x_1}&\\
		
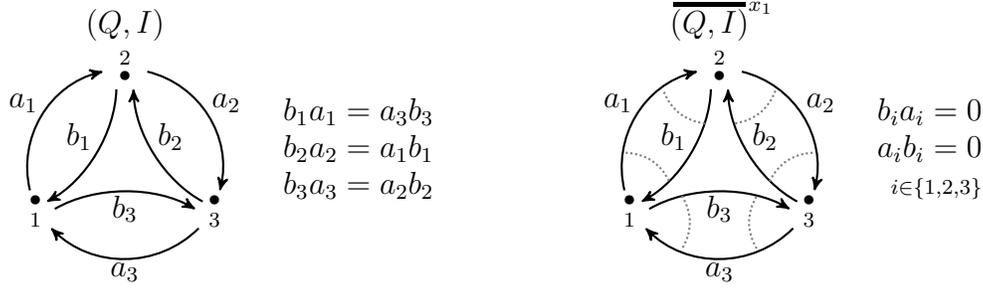
{
			\begin{tikzcd}[,nodes={
					baseline=0pt,
					inner sep=1pt}, 
				every label/.append style={inner sep=1pt, font=\normalsize, color=black},
				column sep=0.6cm, cells={shape=circle},
				row sep=1.35cm,
				arrows={->,thick,black, >=stealth',yshift=0mm
				},
				ampersand replacement=\&
				]
				\& 
				\overset{2}{\bt} 
				\arrow[rd,bend left,looseness=1, "{a_2}",
				""{name=e1, near start,swap, inner sep =-1pt},
				""{name=e2, near end, swap, inner sep =-1pt}, out=45, in = 135
				]
				\ar[ld, "b_1",swap, bend left,looseness=0.85,xshift=-3pt,
				""{name=c1, near start, swap, inner sep =-1pt},
				""{name=c2, near end, swap, inner sep =-1pt}]
				\&
				\\
				\underset{1}{\bt} \ar[ru, bend left,looseness=1, "{a_1}", 
				""{name=d1, near start, swap, inner sep =-1pt},
				""{name=d2, near end, swap, inner sep =-1pt}, out=45, in = 135
				]  
				\ar[rr, "b_3",swap, bend left,looseness=0.85,yshift=-3pt,
				""{name=b1, near start, swap, inner sep =-1pt},
				""{name=b2, near end, swap, inner sep =-1pt}] 
				\& \&
				\underset{3}{\bt} \ar[ll, bend left, looseness=1, "{a_3}",
				""{name=f1, near start, swap, inner sep =-1pt},
				""{name=f2, near end, swap, inner sep =-1pt} , out=45, in = 135
				]
				\ar[lu, "b_2",swap, bend left,looseness=0.85,xshift=3pt,
				""{name=a1, near start, swap, inner sep =-1pt},
				""{name=a2, near end, swap, inner sep =-1pt} 
				]
		\end{tikzcd} } 
		&
		\begin{array}{l} \mathstrut \\
			b_1 a_1 = a_3 b_3  \\
			b_2 a_2 = a_1 b_1 \\
			b_3 a_3 = a_2 b_2  
		\end{array}
		\qquad
		&
		\qquad
		{
			\begin{tikzcd}[,nodes={
					baseline=0pt,
					inner sep=1pt}, 
				every label/.append style={inner sep=1pt, font=\normalsize, color=black},
				column sep=0.6cm, cells={shape=circle},
				row sep=1.35cm,
				arrows={
					->,
					thick,black
					, >=stealth',yshift=0mm},
				ampersand replacement=\&
				]
				\& 
				\overset{2}{\bt}
				\arrow[rd,bend left,looseness=1, "{a_2}",
				""{name=e1, near start,swap, outer sep=0pt, inner sep =-1pt},
				""{name=e2, near end, swap, outer sep=0pt, inner sep =-1pt}, out=45, in = 135
				]
				\ar[ld, "b_1", swap, bend left,looseness=0.85,xshift=-3pt,
				""{name=c1, near start, swap, outer sep=0pt, inner sep =-1pt},
				""{name=c2, near end, swap, outer sep=0pt, inner sep =-1pt}]
				\&
				\\
				\underset{1}{\bt} \ar[ru, bend left,looseness=1, "{a_1}", 
				""{name=d1, near start, swap, outer sep=0pt, inner sep =-1pt},
				""{name=d2, near end, swap, outer sep=0pt, inner sep =-1pt}, out=45, in = 135
				]  
				\ar[rr, "b_3",swap, bend left,looseness=0.85,yshift=-3pt,
				""{name=b1, near start, swap, outer sep=0pt, inner sep =-1pt},
				""{name=b2, near end, swap, outer sep=0pt, inner sep =-1pt}] 
				\& \&
				\underset{3}{\bt} \ar[ll, bend left, looseness=1, "{a_3}",
				""{name=f1, near start, outer sep=0pt, swap, inner sep =-1pt},
				""{name=f2, near end, swap,outer sep=0pt, inner sep =-1pt} , out=45, in = 135
				]
				\ar[lu, "b_2", swap, bend left,looseness=0.85,xshift=3pt,
				""{name=a1, near start, swap, outer sep=0pt, inner sep =-1pt},
				""{name=a2, near end, swap,outer sep=0pt,  inner sep =-1pt} 
				]
				\arrow[from=e1, to=a2, densely dotted, black!50, bend left, -]
				\arrow[from=a1, to=e2, densely dotted, black!50, bend left, -]
				\arrow[from=f1, to=b2, densely dotted, black!50, bend left, -]
				\arrow[from=b1, to=f2, densely dotted, black!50, bend left, -]
				\arrow[from=d1, to=c2, densely dotted, black!50, bend left, -]
				\arrow[from=c1, to=d2, densely dotted, black!50, bend left, -]
		\end{tikzcd} } 
		&
		\begin{array}{r} \mathstrut \\
			b_i a_i = 0\\
			a_i b_i = 0  \\
			\ \scriptstyle i \in \{1,2,3\}
		\end{array}
	\end{array}
	$$
 Let $\Rx$ denote the ring of formal power series $\kk\llbracket x_1,x_2 \rrbracket$ and $\A$ the completion of the path algebra $\kk Q/I$ with respect to its arrow ideal.
	For any integers $m_a, m_b > 0$ it can be shown that the morphism
	$\Rx \longrightarrow \A$ given by $x_1 \longmapsto 	\sum_{i=1}^3 b_i a_i$
	and
	$$
	x_2 \longmapsto 
	(a_3 a_2 a_1)^{m_a} -
	(b_1 b_2 b_3)^{m_b} +
	(a_1 a_3 a_2)^{m_a} -
	(b_2 b_3 b_1)^{m_b} +
	(a_2 a_1 a_3)^{m_a} - 
	(b_3 b_1 b_2)^{m_b}
	$$
	endows the ring $\A$ with the structure of a Noetherian $\Rx$-free algebra.
	
	The quotient algebra ${\AI}^{x_1} \colonequals \A/x_1 \A$ is   
	isomorphic to the arrow ideal completion of the path algebra of the 
	gentle quiver ${\ol{(Q,I)}}^{x_1}$ shown on the right.

	The quiver of the finite-dimensional $\kk$-algebra  ${\AI}^{(m_a,m_b)} \colonequals 
	\A/ (x_1,x_2)\A $
 is given by the latter together with the relations
	$$
	(a_3 a_2 a_1)^{m_a} = 
	(b_1 b_2 b_3)^{m_b} \quad
	(a_1 a_3 a_2)^{m_a} =
	(b_2 b_3 b_1)^{m_b} \quad
	(a_2 a_1 a_3)^{m_a} =
	(b_3 b_1 b_2)^{m_b}\,.
	$$
	In different terms, $\AI^{x_1}$ is a \emph{ribbon graph order} and ${\AI}^{(m_a,m_b)}$ is the \emph{Brauer graph algebra} associated to the Brauer graph
	$$
	\begin{td}
		\overset{m_a}{\circ} \ar[bend left, looseness=1,-,thick, gray]{rr} \ar[-,thick,gray]{rr} \ar[bend right, looseness=1,-, thick, gray]{rr} \& \& \overset{m_b}{\circ}
		\end{td}
	$$
	We will apply our main result to 
	the
algebra $\A$ and its quotients 
	in Subsection~\ref{subsec/app-silt-bij}.
\end{ex}

\subsection{Recall on Krull-Remak-Schmidt categories}

We recall a few notions on Krull-Remak-Schmidt categories, which can be found for example in \cite{Krause}*{Section~4}.

A ring $A$ is \emph{local} if the sum of any two non-units from $A$ is not a unit.
The ring $A$ is \emph{semiperfect} if there is a decomposition $1_{A} = e_1+ \ldots + e_n$ into mutually orthogonal idempotents of $A$ such that each ring $e_i A e_i$ is local.
By \cite{Krause}*{Corollary~4.4}, a \emph{Krull-Remak-Schmidt category} can be defined as an additive category $\mathcal{A}$ such that the endomorphism ring of any object from $\mathcal{A}$ is semiperfect and $\mathcal{A}$ is \emph{idempotent-complete}, that is, any idempotent endomorphism in $\mathcal{A}$ has a kernel.

Any object  of a Krull-Remak-Schmidt category $\mathcal{A}$ is isomorphic to  a direct sum of finitely many indecomposable objects.
Such a decomposition is unique up to permutation and isomorphism of the indecomposable summands. 
An object in $\mathcal{A}$ is \emph{basic} if it does not have isomorphic indecomposable summands.

The main motivation to assume that $\A$ is a Noetherian $\Rx$-algebra over a complete local ring $\Rx$ is the following.
\begin{prp} \label{KRS} 
	The category $\per \A$ has the Krull-Remak-Schmidt property and any of its morphisms spaces is a finitely generated $\Rx$-module.
\end{prp}
\begin{proof}
	For any objects $\CL, \CM \in \Tau \colonequals  \per \A$ we may choose quasi-isomorphic objects $\CP, \CQ \in \Hotb(\proj \A)$ to argue that $\Hom_{\Tau}(\CL,\CM)$ is isomorphic to a subquotient $\Hom_{\Hotb(\proj \A)}(\CP,\CQ)$ of the finitely generated $\Rx$-module $\bigoplus_{i\in\Z} \Hom_{\Rx}(P^i,Q^i)$.
	
	In particular, the $\Rx$-algebra $\End_{\Tau}(\CL)$ is Noetherian, and thus a semiperfect ring by \cite{Lam}*{(23.3)}.
	The perfect derived category of any ring is known to be idempotent-complete by \cite{Boekstedt-Neeman}*{Proposition~3.4}.
\end{proof}

Similarly, the category $\per \AI$ of the Noetherian $\RI$-algebra $\AI$ has the Krull-Remak-Schmidt property and morphisms spaces are finitely generated $\RI$-modules.

\subsection{Recall of silting and tilting complexes}\label{sec:recall}
Next, we recall a result on silting complexes by Aihara and Iyama \cite{Aihara-Iyama} and a brief version of Rickard's derived Morita theorem \cite{Rickard89a}.
For this, we introduce  three relations on pairs of complexes.
\begin{notation}\label{not/order}
	Throughout this paper, for any complexes $\CL, \CM \in \D(\A)$ we set 
	\begin{align*}
	\begin{array}{cl}
	\CL \geq \CM  &\text{if }  \Hom_{\D(\A)}(\CL,\CM[i]) = 0\text{ for any integer }i > 0,\\
	\CL \teq \CM  &\text{if }  \Hom_{\D(\A)}(\CL,\CM[i]) = 0\text{ for any integer }i \neq 0,\\
	\CL \perp \CM &\text{if }  \Hom_{\D(\A)}(\CL,\CM[i]) = 0\text{ for any integer }i. 
	\end{array}
	\end{align*}
\end{notation}

Although this notation might suggest otherwise, usually none of these relations defines a partial order on objects in $\D(\A)$ or their isomorphism classes.

Let $\CL$ be a perfect complex of $\A$.
We denote by $\langle \CL \rangle$ the smallest strictly full subcategory of $\D(\A)$ which contains $\CL$ and which is closed under cones, shifts
and direct summands.
We call the complex $\CL$ a \emph{perfect generator} if 
$\langle \CL \rangle = \per \A$, or, equivalently, $\A \in \langle \CL \rangle$.
The complex $\CL$ of $\A$ is called \emph{presilting} if $\CL \geq \CL$, \emph{pretilting}\footnote{The term `pretilting' does not belong to standard terminology in the literature but will be convenient for abbreviation in this paper.}
if even $\CL \teq \CL$, \emph{silting} if it is presilting and generates $\per \A$, and \emph{tilting} if it is a pretilting generator of $\per \A$.
\begin{notation}
	We denote by $\silt \A$ the set of isomorphism classes of \emph{basic} silting complexes of the Krull-Remak-Schmidt category $\per \A$.
		For any complex $\CL \in \per \A$ we denote by $|\CL|$ the number of isomorphism classes of its indecomposable summands.
\end{notation}

Silting complexes are distinguished by the following results.
\begin{thm}[Aihara and Iyama \cite{Aihara-Iyama}*{Theorems 2.11, 2.35}] \label{thm/AI}
	The following statements hold.
	\begin{enumerate}
		\item \label{thm/AI1} The relation $\geq$ defines a partial order on the set $\silt \A$.
		\item \label{thm/AI2} For any complex $\CL \in \silt \A$  it holds that $|\CL| = |\A|$.
		\item \label{thm/AI3} For any non-zero direct summand  $\CX$ of a complex $\CL \in \silt \A$ there is a complex
		$\mu^+_{\CX}(\CL) \in \silt \A$, the 
\emph{right mutation} of $\CL$ at $\CX$, such that $
\CL \not\cong \mu^+_{\CX}(\CL)$ and $ \CL \geq \mu^+_{\CX}(\CL) \geq\CL[1]$.
	\end{enumerate}
\end{thm}
There is dual notion of \emph{left mutation}.
Silting complexes might be viewed as an indirect means 
to study tilting complexes.
The study of the latter can be motivated by a well-known result due to Rickard.

\begin{thm}[Rickard \cite{Rickard89a}]
	The ring $\A$ is derived equivalent to another ring $A$ if and only if there is a tilting complex $\CT$ of $\A$ such that $\End_{\D(\A)}(\CT) \cong A$.
\end{thm}
Since $\A$ is a Noetherian $\Rx$-algebra, so is its derived equivalent ring $A$.
Moreover, any derived equivalence induces an isomorphism of posets 
$$
\begin{td}  \phantom{ f_s\colon}\silt \A \ar{r}{\sim} \& \silt A \,.\end{td}
$$
The notions and statements of this subsection apply to the quotient $\AI$ as well.

One of the main goals of this paper is to establish a well-defined isomorphism of posets
\begin{align}\label{eq/silt-bij-goal}
\begin{td} f_s\colon \silt \A \ar{r}{\sim} \& \silt \AI \end{td}
\end{align}
in both setups of Remark~\ref{rmk/sub}, that is, if $\A$ is free over $\Rx$ or if $\AI$ a complete intersection.
At this point, the reader might jump to the main results of this paper and their applications in Section~\ref{sec/main}.

\subsection{Basic approach and structure of the paper} \label{subsec/structure}
In this subsection, we sketch the arguments to establish the silting bijection~\eqref{eq/silt-bij-goal}.
The main problem is that the functor 
$$\FF\colon \begin{td} \D^-(\md \A) \ar{r} \& \D^-(\md \AI), \& \CM \ar[mapsto]{r} \&  \CM
\underset{\A}{\overset{\mathbb{L}}{\otimes}} \AI \end{td}
$$ 
does not usually induce an injective or surjective map on isomorphism classes of objects.
We address the issue of injectivity for silting complexes in the first three steps.
\begin{enumerate}
	\item \label{step01}
	Let $\CL \in \per \A$ and $\CM \in \D^-(\md \A)$ be given by complexes of finitely generated projective $\A$-modules.
	The three relations between $\CL$ and $\CM$ from Notation~\ref{not/order} translate directly into cohomological vanishing conditions of the complex $\CK \colonequals \Hom^{\bt}_{\A}(\CL,\CM)$  of $\Rx$-modules.
	The same turns out to be true for the pairs of complexes $\PCL$ and $\PCM$ and the complex $\CK \otimes \RI$. 
	\item 
	Applying the \emph{Künneth trick} together with Nakayama's Lemma to the complex $\CK$ of $\Rx$-modules yields the implications
	\begin{align*}
	\Ho^+(\CK) = 0
	\quad &\Leftrightarrow \quad 	\Ho^+(\CK \otimes \RI) = 0	 \quad \Rightarrow \quad \begin{td}
	\Ho^0(\CK)\ar[twoheadrightarrow]{r} \& \Ho^0(\CK\otimes \RI)  \end{td}\\
	\Ho^*(\CK) = 0 \quad &\Leftrightarrow \quad \Ho^*(\CK \otimes \RI) = 0\,.
	\end{align*}
	\item Using the translations in~\eqref{step01} it follows that
	\begin{align*}
	\CL \geq \CM \quad & \Leftrightarrow \quad \PCL \geq \PCM \quad \Rightarrow \quad
	\begin{td}
	{\Hom_{\D(\A)}(\CL,\CM)} \ar[twoheadrightarrow]{r}{\FF} \& 
	\Hom_{\D(\Abar)}(\PCL,\PCM) \end{td}
	\\
	\CL \perp \CM \quad &\Leftrightarrow \quad \PCL \perp \PCM\,.
	\end{align*}
	These ``key implications'' for the functor $\FF$ allow to deduce that the map $f_s$ is a well-defined injective embedding of posets by purely categorical arguments.
	\end{enumerate}
The approach above is based on \cite{Rickard91a}*{Proof of Theorem~2.1}.
It remains to show surjectivity of the map $f_s$.
\begin{enumerate}\setcounter{enumi}{3}
	\item \label{step1} Next, we need to lift a given silting complex $\CP$ of $\AI$ to a perfect complex $\CL$ of $\A$ under $\FF$. 
	For this, we extend techniques developed by Eisenbud, Rickard and Yoshino in order to show that the complex $\CP$ has such a lift if 
	$$
	\Hom_{\D(\Abar)}(\CP, \alpha_n(\CP)[2]) = 0 \text{ for any integer }n > 0
	$$ 
	where $\alpha_n(\CP)$ might be viewed as a twist of $\CP$
	with the $\RI$-module $\ida^n/\ida^{n+1}$.
	If each of the latter $\RI$-modules is free, these twists are trivial and $\CP$ lifts immediately.
	Under a more general condition satisfied in each of the setups of Remark~\ref{rmk/sub}, we show that $\CP \geq \CP[1]$ implies  $\CP \geq \alpha_n(\CP)[1]$.
	In this way, we obtain that any silting complex $\CP$ has a presilting lift $\CL$.
	\item  Using  arguments by Keller on differential graded categories and an approach by Rickard it can be shown that the presilting complex $\CL$ is a perfect generator if and only if $\CL \not\perp \CM$ for any non-zero object $\CM \in \D^-(\md \A)$.
	\item The last condition follows from the fact that $\PCL \not\perp\CN$ for any non-zero object $\CN \in \D^-(\md \AI)$.
	Therefore, the lift $\CL$ of the silting complex $\CP$ is silting and the map $f_s$ is surjective.
\end{enumerate}
	In summary, this shows that the map $f_s$ is an isomorphism of posets.
	
	We will say that the silting property is \emph{ascent} if the functor $\FF$ preserves silting complexes, and that it is \emph{descent}
	if any perfect complex $\CL$ of $\A$ such that $\PCL$ is silting is itself silting.
	
The first step is carried out in the next subsection.
The remaining five steps correspond to the next five sections of this paper.
The first three steps can be carried out in a more general situation than Setup~\ref{setup/main} which leads to 
silting embeddings and descent results in Section~\ref{sec/embeddings}.
The sections of this paper depend on each other essentially as follows.
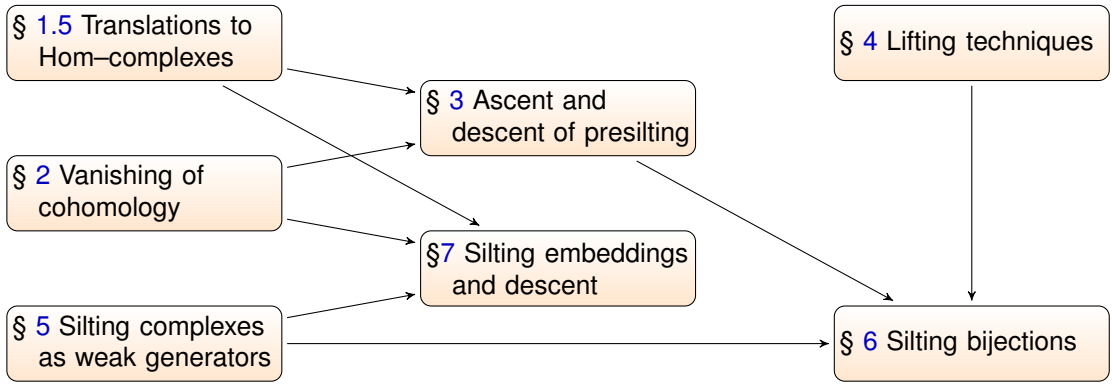
\begin{center}
	\sffamily 
	\footnotesize
	\begin{tikzpicture}[block/.style = {shape=rectangle, rounded corners,
		draw, align=center, text ragged, minimum height=10mm, text width=10em,
		top color=white, bottom color=orange!20}, inner sep=2pt, outer sep=2pt, >=stealth']
	\node (sec1) [block] at (0,0) {
		\S~\ref{subsec/hom-complexes} Translations to \mbox{\quad  Hom--complexes}};
	\node (sec2) [block] at (0,-2) {\S~\ref{sec/refined-Kuenneth} Vanishing of \mbox{\quad cohomology}};
	\node (sec5) [block] at (0,-4) {\S~\ref{sec/characterization} Silting complexes \mbox{\quad as weak generators}}; 
	\node (sec3) [block] at (5.5,-1) {\S~\ref{sec/ascent-descent} Ascent and \mbox{\quad descent of presilting}};
			\node [block] at (5.5,-3)  (sec7)  {\S \ref{sec/embeddings} Silting embeddings \mbox{\quad and descent}};
	\node [block] (sec4) at (11,0) {\S~\ref{sec/lifting} Lifting techniques}; 				
	\node [block](sec6) at (11, -4) {\S~\ref{sec/main} Silting bijections};
	\draw[->] (sec1) -- (sec3);
	\draw[->] (sec2) -- (sec3);
	\draw[->] (sec1) -- (sec7);
	\draw[->] (sec2) -- (sec7);
	\draw[->] (sec5) -- (sec7);
	\draw[->, bend right] (sec5) -- (sec6);
	\draw[->] (sec3) -- (sec6);
	\draw[->] (sec4) -- (sec6);
	\end{tikzpicture} 
\end{center}

\subsection{Translations to Hom--complexes}
\label{subsec/hom-complexes}

Let $\CL, \CM$ be complexes of modules over the ring $\A$.
We recall the
definition of 
the \emph{$\Hom$-complex} $\CK \colonequals 
\Hom^{\bt}_{\A}(\CL,\CM)$. Namely, at each degree $i \in \Z$ it is given by
$$
\begin{array}{rcl}
d^i_{\CK}\colon
K^i  \colonequals \displaystyle \prod_{j \in \Z} \Hom_{\A}(L^j,M^{i+j})
&\begin{td}
\mathstrut
\ar{r} \& \mathstrut
\end{td}&
K^{i+1}  \colonequals \displaystyle \prod_{j \in \Z} \Hom_{\A}(L^j,M^{i+j+1}) \\
(\phi^{ij})_{j \in \Z} &
\begin{td}
\mathstrut
\ar[mapsto]{r} \& \mathstrut
\end{td}
&
(\phi^{i,j+1} \, d^j_{\CL} - 
d^j_{\CM[i]}
\, \phi^{ij})_{j \in \Z}
\end{array}
$$
The main feature of the complex $\CK$ is that for any integer $i \in \Z$ there is an equality
\begin{align}\Ho^i(\CK) = \Hom_{\Hot(\A)}(\CL,\CM[i]). \label{eq/hom-complex-cohomology}\end{align}
In the following for a set $I$ we denote by $\A^I$ the coproduct
$\oplus_{i \in I} \A$.

\begin{rmk}\label{rmk/right-bounded}
	Let  $P$ and $Q$ be projective $\A$-modules. 
	Then there is a retraction $\pi\colon  \A^{I} \twoheadrightarrow P$ and a section $\iota\colon Q \hookrightarrow \A^J$ of $\A$-modules for some index sets $I$ and $J$. 
	These maps give rise to a section of $\Rx$-modules
	\begin{align} \label{eq/morphism-summand}
	\begin{td}
	\Hom_{\A}(P,Q) \ar[hookrightarrow]{r}{\alpha} \& \Hom_{\A}(\A^I, \A^J) \cong
	\underset{i \in I}{\prod} \ \A^J, \&
	\phi \ar[mapsto]{r}{\alpha} \& \iota \phi \pi.
	\end{td}
	\end{align}
	In particular, for any complexes $\CP \in \Hotb(\proj \A)$ and $\CQ \in \Hot^-(\proj \A)$ it follows that $\Hom^{\bt}_{\A}(\CP,\CQ) \in \Hot^-(\add \A_{\Rx})$, where $\A_{\Rx}$ denotes $\A$ viewed as an $\Rx$-module.
\end{rmk}

The next lemma shows that taking Hom--complexes
commutes with the tensor product with the $\Rx$-algebra $\B$.
It is a variation 
	of an observation due to Rickard \cite{Rickard91a}*{Proof of Theorem~2.1}. A closely related statement was obtained also by Iyama and Kimura \cite{Iyama-Kimura}*{Lemma 2.15}.

\begin{lem}\label{lem/hom-complex}
	Let $\CP, \CQ \in \Hot(\Proj \A)$. 
	Then there is an isomorphism of complexes of $\Sx$-modules
	\begin{align} \nonumber 
	\begin{td}
	\xi\colon
	\Hom^{\bt}_{\A}(\CP,\CQ) \otimes \B \ar{r}{\sim} \& \Hom^{\bt}_{\Abar}(\CP \otimes \B, \CQ \otimes \B),  
	\end{td} 
	\\ 
	\label{eq/hom-complex-pushdown}
	\qquad\qquad
	\begin{td}
	\phi^{ij} \otimes a 
	\ar[mapsto]{r} \& 
	\phi^{ij} \otimes
	(\lambda_a\colon	 
	x \mapsto ax)
	\end{td} \qquad \mathstrut
	\end{align}
\end{lem}
\begin{proof}
	It can be verified that the morphism $\xi$ commutes with the differentials of complexes in~\eqref{eq/hom-complex-pushdown}.
	Therefore, we need only to show the corresponding claim 	
	for  projective $\A$-modules $P$ and $Q$. 
	The $\Rx$-linear section $\alpha$ from
	\eqref{eq/morphism-summand}  appears in the commutative diagram
	of $\Sx$-modules
	\begin{align*}
	\begin{td}
	\Hom_{\A}(P,Q) \otimes \B \ar{d}{\xi_1} \ar[hookrightarrow]{r}{\alpha \otimes \id} \& \Hom_{\A}(\A^I, \A^J) \otimes \B \ar{r}{\sim} \ar{d}{\xi_2} \& (\prod_{i \in I}  \A^J) \otimes \B  \ar{d}{\psi} \\
	\Hom_{\Abar}(P \otimes \B,Q\otimes \B)\ar[hookrightarrow]{r}{
		\beta} \& \Hom_{\Abar}( \A^I\otimes \B, \A^J \otimes \B) \ar{r}{\sim} \& \prod_{i \in I}  (\A^J \otimes \B)\,,
	\end{td} 
	\end{align*}
	where the map $\beta$ is a section given by $\varphi \mapsto (\iota \otimes \id) \varphi (\pi \otimes \id)$, the maps $\xi_1$ and $\xi_2$ by the same rule as in~\eqref{eq/hom-complex-pushdown}
	and the remaining maps are certain natural choices.
	
	Since the $\Rx$-module $\B$ is finitely presented, the map $\psi$, and thus the maps $\xi_2$ and $\xi_1$ are isomorphisms. This implies the claim.
\end{proof}

\begin{notation}\label{not/complexes}
	Throughout this paper, for any complex $\CK$ we denote
	\begin{align*}
	\begin{array}{cl}
	\Ho^+(\CK)=0  &\text{if }  \Ho^{i}(\CK) = 0\text{ for any integer }i > 0,\\
	\Ho^{\pm}(\CK)=0  &\text{if }  \Ho^i(\CK) = 0\text{ for any integer }i \neq 0,\\
	\Ho^{*}(\CK)=0  &\text{if }  \Ho^i(\CK) = 0\text{ for any integer }i.
	\end{array}
	\end{align*}
\end{notation}

The next lemma translates the three relations from Notation~\ref{not/order} of a pair of complexes into 
cohomological properties of 
their Hom--complex.
\begin{lem}\label{lem/translation}
	Let $\CL \in \per \A$ and $\CM\in \D^-(\md \A)$.
	Let $\CP \in \Hotb(\proj \A)$ and $\CQ \in \Hot^-(\proj \A)$ be complexes 
	quasi-isomorphic to $\CL$ and $\CM$, respectively.
	Set $\CK \colonequals \Hom^{\bt}_\A(\CP,\CQ)$.
	Then the following statements hold.
	\begin{enumerate}
		\item 	\label{lem/translations1}
		In Notations~\ref{not/order}~and~\ref{not/complexes}, the following equivalences hold.
		\begin{align*} 
		\begin{array}{cclcccl}
		\CL \geq \CM &\Leftrightarrow&  \Ho^+(\CK) = 0\,,  & &
		\PCL  \geq \PCM &\Leftrightarrow& \Ho^{+}(\CK \otimes \B) = 0\,,  \\
		\CL \teq \CM & \Leftrightarrow& \Ho^{\pm}(\CK) = 0\,, &&
		\PCL  \teq \PCM &\Leftrightarrow&
		\Ho^{\pm}(\CK \otimes \B) = 0\,,\\
		\CL \perp \CM &\Leftrightarrow&
		\Ho^*(\CK) = 0\,, && 
		\PCL \perp \PCM &\Leftrightarrow&
		\Ho^*(\CK \otimes \B) = 0\,.
		\end{array}
		\end{align*}
		\item \label{lem/translations2}
		There is a commutative diagram of $\Rx$-modules
		\begin{align}\label{eq/morphisms}
		\begin{td}
		\&
		\Ho^0(\CK) \otimes \B  \ar{rd}{\kappa}  \ar[dashed, "\sim" labl, pos=0.33]{dd} \& \\
		\Ho^0(\CK) \ar{rr} \ar[twoheadrightarrow]{ru}{\eta}  \ar["\sim" labl, pos=0.33]{dd}
		\&[-1.0cm]  \&[-1.0cm] \Ho^{0}(\CK\otimes \B) \ar["\sim" labl, pos=0.33, "{\xi'}" pos=0.33]{dd}, 
		\\
		\&
		\Hom_{\D(\A)}(\CL,\CM) \otimes \B \ar[dashed]{rd}{\gamma} \\
		\Hom_{\D(\A)}(\CL,\CM) \ar[dashed,twoheadrightarrow]{ru} \ar{rr}{\FF} \&  \&
		\Hom_{\D(\Abar)}(\PCL, \PCM) 
		\end{td}
		\end{align}
		where $\eta$
		denotes the unit map and 
		$\kappa$
		the 
		\emph{Künneth map} given by
		$\ \ol{x} \otimes y \mapsto  \ol{x \otimes y}$.
		\item \label{algebra-morphism} If $\CL = \CM$, the map $\gamma$ defines a morphism of $\Sx$-algebras 
		$$
		\begin{td} 
		\End_{\D(\A)}(\CL)\otimes \B \ar{r} \&  \End_{\D(\Abar)}(\PCL)\,. 
		\end{td}
		$$
		\item \label{lem/translations4}
		The complex $\CK$ is right-bounded and satisfies
		$\Tor_{+}^{\Rx}(\CK,\B) = 0$, that is,
		$$
		\Tor_n^{\Rx}(K^i,\B) = 0\text{ for any integers }n > 0, i \in \Z \,.
		$$
		Moreover, $\Ho^i(\CK)$ is a finitely generated $\Rx$-module for any integer $i \in \Z$.
		
	\end{enumerate}
\end{lem}
\begin{proof}
	\begin{enumerate}
		\item
		For any 
		integer $i \in \Z$
		\eqref{eq/hom-complex-cohomology}
		and
		\eqref{eq/hom-complex-pushdown} 
		yield the isomorphisms
		$$
		\Ho^i(\CK) 
		\cong \Hom_{\D( \A)}(\CL,\CM[i]) \quad \text{and} \quad
		\Ho^i(\CK \otimes \B) \cong 
		\Hom_{\D(\Abar)}(\PCL,\PCM[i]).$$
		This yields the equivalences in  
		\eqref{lem/translations1}. 
		\item 
		There is an equivalence
		$
		\mathbf{E}_{\A}\colon \Hot^-(\Proj \A){\overset{\sim}{\longrightarrow}}\D^-(\Md\A)
		$
		of categories
		which acts as identity on objects,
		and a similar equivalence $\mathbf{E}_{\Abar}$ for the ring $\AI$.

		Let us denote
		$\mathbf{T} \colonequals \dash \otimes \B \colon \Hot^-(\proj \A){\overset{\sim}{\longrightarrow}} \Hot^-(\proj \AI)$.
		We are going to construct a commutative diagram of $\Rx$-modules
		\begin{align*}
		\begin{td}
		\&[-1cm]	\Hom_{\Hot(\A)}(\CP,\CQ) \otimes \B
		\ar[dashed, "\sim" labl, pos=0.2]{dd}[pos=0.2]{\alpha \otimes \id}
		\ar{r}{\kappa} 
		\ar[dashed]{rd}
		\& 
		\Ho^0(\CK\otimes \B) \ar["\sim" labl, pos=0.5]{d}{\Ho^0(\xi)}
		\\
		\Hom_{\Hot(\A)}(\CP,\CQ) \ar[twoheadrightarrow]{ru}{\eta}
		\ar["\sim" labl, pos=0.33]{dd}[pos=0.33]{\alpha}  \ar{rr}[swap]{\mathbf{T}}
		\&  \& \Hom_{\Hot(\Abar)}(\CP \otimes \B,\CQ\otimes\B) 	\ar["\sim" labl, pos=0.33, dashed]{dd}[pos=0.33]{\beta} \\	
		\&
		\Hom_{\D(\A)}(\CL,\CM) \otimes \B \ar[dashed]{rd}{\gamma} 
		\& 
		\\
		\Hom_{\D(\A)}(\CL,\CM) \ar{rr}[swap]{\FF} \ar[dashed,twoheadrightarrow]{ru}
		\&
		\&  \Hom_{\D(\Abar)}(\PCL,\PCM)   \,.
		\end{td}
		\end{align*}
		It is straightforward to check that the composition $\Ho^0(\xi) \cdot \kappa \cdot \eta$ maps a homotopy class 
		$[\phi]\colon\CP \longrightarrow \CQ$
		to the homotopy class
		$\mathbf{T}([\phi])$.
		The isomorphism $\alpha$ on the left
		is given by a certain conjugation with the isomorphisms
		$\mathbf{E}_{\A}(\CP)\cong \CL$ and 
		$\mathbf{E}_{\A}(\CQ)\cong \CM$ in $\D^-(\Md \A)$.
		
		By definition of the derived functor $\FF$
		there is a natural isomorphism 
		$\mathbf{E}_{\Abar} \circ \mathbf{T} \cong \FF \circ \mathbf{E}_{\A}$
		which
		implies the existence of an $\Sx$-linear isomorphism $\beta$ 
		making the front square commutative.
		Set $\gamma \colonequals \beta \cdot \Ho^0(\xi) \cdot \kappa \cdot (\alpha \otimes \id)^{-1}$.
		Then all morphisms in the diagram above commute and we obtain diagram~\eqref{eq/morphisms} setting $\xi' \colonequals \beta \cdot \Ho^0(\xi)$.
		\item		In case $\CL = \CM$, it can be verified that the $\Sx$-linear maps $\beta$, $\Ho^0(\xi)\cdot\kappa$ and $\alpha \otimes \id$ are ring morphisms.
		Thus, the map $\gamma$ is an $\Sx$-algebra morphism. 
		\item
		In the notations above, it holds that $\CK \in \Hot^-(\add \A_{\Rx})$
		by Remark~\ref{rmk/right-bounded}.
		Since  $\Tor_+^{\Rx}(\A,\B)=0$ and $\A_{\Rx} \in \md \Rx$
	are assumptions of Setup~\ref{setup/main}, it follows that $\Tor_+^{\Rx}(\CK,\B) = 0$ 
		and
		$\Ho^i(\CK)\in \md \Rx$ for any $i \in \Z$.
		\qedhere	
	\end{enumerate}
\end{proof}

\section{Vanishing of cohomology under change of rings}

\label{sec/refined-Kuenneth}

In this section we use a few basic facts on spectral sequences in order to
show that a certain complex $\CK$ of $\Rx$-modules 
like the one in Lemma~\ref{lem/translation}
satisfies
the implications
\begin{align*}
\Ho^+(\CK) = 0 \quad \Leftrightarrow \quad
\Ho^+(\CK \otimes \B) = 0\,, \qquad
\Ho^{\pm}(\CK) = 0 \quad \Rightarrow \quad
\Ho^{\pm}(\CK \otimes \B) = 0\,.
\end{align*}
More precisely, the first equivalence is a consequence of the so-called \emph{Künneth trick} which is recovered
using a variant of the Künneth spectral sequence $E^2_{pq}$ in Subsection~\ref{subsec/kss}.
In \cite{Rickard91a} Rickard deduced the implication on the right from the fact that the vanishing $E^2_{pq} = 0$ at all lattice points $p,q \in \Z$ with $(p,q) \neq (0,0)$ implies vanishing
of the limit term $E_i = 0 $ for any integer $i \neq 0$.
Subsection~\ref{subsec/spectral} provides a converse of this statement for certain spectral sequences, which makes the implication on the right into an equivalence. In the last subsection, we describe a refinement of  the latter equivalence.

\subsection{Reverse vanishing for spectral sequences}\label{subsec/spectral}

Throughout this subsection, we consider a convergent homological 
spectral sequence $$E^2_{pq} \quad {\Rightarrow}\quad E_{p+q}$$ of objects in an abelian category.
We assume that the spectral sequence $E^2_{pq}$ lies in the \emph{first quadrant}, that is,
$E^2_{pq}= 0$	for any $p,q\in \Z$ such that $p < 0$ or $q < 0$.

We will need only a few basic facts about spectral sequences in the following, which are collected in the next remark. 
\begin{rmk}
	For each integer $r \geq 2$ the spectral sequence has an $r$-th page
	which is given by objects $(E^r_{pq})_{p,q\in \Z}$ with certain  differentials 
	$d^{r}_{pq}\colon E^r_{pq} \longrightarrow E^r_{p-r,q+r-1}$, whose explicit form will not be relevant.
	Taking homology at each lattice point $(p,q)\in \Z^2$ on page $r$ yields the entry 
	$E^{r+1}_{pq}$ of the next page.
	
	Since the spectral sequence converges,
	each lattice point $(p,q) \in \Z^2$ admits a number $ r \colonequals r(p,q) \geq 2$ such that
	$d^s_{pq} = 0$ and $E^s_{pq}$ is isomorphic to a \emph{subquotient} $E^{\infty}_{pq}$ of the limit object $E_{p+q}$ for any number $s \geq r$.
	
	Vice versa, for each integer $i \in \Z$ the limit object $E_i$ has a countably indexed filtration whose subquotients can be identified with the objects $(E^{\infty}_{p,i-p})_{p \in \Z}$.
	
	Since $E_{pq}^2$ lies in the first quadrant, it follows that $E_i = 0$ for any integer $i < 0$.
\end{rmk}

\begin{notation}\label{not/spectral}
	We will write
	\begin{itemize}
		\item
		$E^2_{p+} = 0$  if
		$E^2_{pq} = 0$ for any 
		integer $q > 0$, that is,
		$E^2_{pq}$ collapses on the $p$-axis;
		\item
		$E^2_{+0}=0$  if 
		$E^2_{p0} = 0$ for any integer $p > 0$, so
		$E^2_{pq}$ vanishes on the positive $p$-axis;
		\item	$E_+ = 0$  if $E_n = 0$ for any integer $n > 0$, thus
		only $E_0$ may not vanish.
	\end{itemize}
\end{notation}

The second statement below is our starting point 
for a reverse vanishing result.
\begin{lem}\label{lem/spectral1}
	In the setup above, the following statements hold.
	\begin{enumerate}
		\item \label{lem/spectral1b} Assume that $E^2_{p+} = 0$. Then there is an isomorphism
		$
		E^2_{p0} \cong E_p
		$ for any integer $p \in \Z$, and thus
		$E^2_{+0} = 0$ is equivalent to $E_{+} = 0$.
		\item \label{lem/spectral1a} 
		If $E_{1} = 0$, then $E_{10}^2 = 0$.
	\end{enumerate}
\end{lem}
\begin{proof}
	\begin{enumerate}
		\item If $E^2_{p+}=0$, then each differential $d^r_{pq}$ on any page $r \geq 2$ is zero.
		\item
		Since $E^2_{pq}$ is zero outside the first quadrant, on any page $r \geq 2$  we obtain 
		$$
		\begin{td} E^r_{1+r,1-r}=0 \ar{rr}{d^r_{1+r,1-r}}
		\& \& E^r_{10} \ar{rr}{d^r_{10}} \&\& E^r_{1-r,r-1} = 0. \end{td}
		$$
		This shows that $E^{2}_{10}$ is isomorphic to a subquotient 	of $E_{1} = 0$.
		\qedhere
	\end{enumerate}
\end{proof}
To obtain stronger results, we impose conditions on the spectral sequence $E^2_{pq}$ of the type that \emph{vanishing at a particular lattice point $(p,q)$ implies vanishing of all entries right from $(p,q)$ in the same row}.
\begin{prp} \label{prp/spectral2}
	In the setup above, the following statements hold.
	\begin{enumerate}
		\item \label{prp/spectral2a}
		Assume that the spectral sequence $E^2_{pq}$ has the property
		\begin{enumerate}[label={$\mathsf{(R1)}$}, ref={$\mathsf{(R1)}$}]
			\item \label{spectral/reflect} If $E^2_{0q} = 0$ for an integer $q > 0$, then 
			$E^2_{pq} = 0$ for any integer $p > 0$.
		\end{enumerate}
		Then 
		any two of 
		the conditions 
		$E^2_{p+} = 0$,  $E^2_{+0} = 0$, $E_+ = 0$ imply the third one.
		
		\item \label{prp/spectral2b}
		Assume that $E^2_{pq}$ satisfies~\ref{spectral/reflect} and the property
		\begin{enumerate}[label={$\mathsf{(R2)}$}, ref={$\mathsf{(R2)}$}]
			\setcounter{enumi}{1}
			\item \label{spectral/rigid} If $E^2_{10} = 0$, then $E^2_{+0} = 0$.
		\end{enumerate}
		Then  $E^2_{p+} = E^2_{+0} = 0$  is equivalent to $E_{+} = 0$.
	\end{enumerate}
\end{prp}
\begin{proof}
	Assume that $E^2_{pq}$ satisfies~\ref{spectral/reflect}.
	Because of Lemma~\ref{lem/spectral1}~\eqref{lem/spectral1b} it is sufficient to show the implication
	``$E^2_{+0} = E_{+} = 0 \Rightarrow E^2_{p+} = 0$''
	for the first claim,
	and ``$E_+ = 0 \Rightarrow E^2_{+0} = 0$''
	assuming~\ref{spectral/rigid} for the second.
	\begin{enumerate}
		\item
		Assume that
		$E^2_{+0} = E_{+} = 0$.
		We show that $E^2_{pq} = 0$ for any integers $p \geq 0$ and $m \geq q \geq 0$ with $(p,q) \neq (0,0)$ by  induction on $m \geq 0$. 
		\begin{itemize}
			\item $m = 0$. It holds that $E^2_{p0} = 0$ for any $p \neq 0$ by the assumptions.
			\item $m \rightarrow m+1$. Assume that
			the claim is true
			for an integer $m \geq 0$.
			The homology at lattice point $(0,m+1)$ on any page  
$r \geq 2$ is computed via the differentials
			$$
			\begin{td} E^r_{r,m-r+2}=0 \ar{rr}{d^r_{r,m-r+2}}
			\& \& E^r_{0,m+1} \ar{rr}{d^r_{0,m+1}} \&\& E^r_{-r,m+r} = 0. \end{td} 
			$$	
			This implies that
			$E^2_{0,m+1} \cong E^\infty_{0,m+1} $, which is zero since $E_{m+1} = 0$. Using property~\ref{spectral/reflect} it follows that $E^2_{p,m+1} = 0$ for any $p > 0$. 
		\end{itemize}
		This shows that 
		$E^2_{pq} = 0$ at any lattice point $(p,q) \neq (0,0)$. Thus,
		$E^2_{p+} = 0$.
		\item 	Assume that $E^2_{pq}$ has property~\ref{spectral/rigid} as well, and that $E_{+} = 0$.
		Then $E_{1} = 0$, and Lemma~\ref{lem/spectral1}~\eqref{lem/spectral1a} yields
		$E_{10}^2 = 0$. Property 	\ref{spectral/rigid}
		implies that $E^2_{+0} = 0$.	\qedhere
	\end{enumerate}
\end{proof}
The last statement allows to deduce vanishing properties of a certain spectral sequence from vanishing of its limit objects.

\subsection{The Künneth spectral sequence}
\label{subsec/kss}
In this subsection, we fix a complex $\CK$ of the following form.
\begin{setup}\label{setup/complexes}
	Let $\CK$ be any right-bounded complex  of $\Rx$-modules such that $$
	\Tor_+^{\Rx}(K^i,\B) = 0 \quad \text{and} \quad 
	\Ho^i(\CK) \in \md \Rx \quad \text{ for any integer $i \in \Z$}.
	$$
\end{setup}
These assumptions are motivated by Lemma~\ref{lem/translation}~\eqref{lem/translations4}.
In particular, the complex $\CK$ is adapted to the 
derived functor 
$$\dash \underset{\Rx}{\overset{\mathbb{L}}{\otimes}} \B\colon 
\begin{td}
\D^{-}(
\Md \Rx
) \ar{r} \& \D^-(
\Md \B
).
\end{td}
$$

The main tool of this subsection is the following
variant of the Künneth spectral sequence \cite{MacLane}*{Theorem~12.1}.
\begin{thm}\label{thm/Kuenneth}
	There is a convergent spectral sequence
	\begin{align}\label{eq/kuenneth-spectral}
	E_{pq}^2 \colonequals \Tor_{p}^{\Rx}(\Ho^{-q}(\CK), \B) \quad \Rightarrow \quad E_{p+q} \colonequals \Ho^{-p-q}(\CK \otimes \B)	
	\end{align}
\end{thm}

\subsubsection*{The Künneth trick}
The next statement is called the \emph{Künneth trick} in \cite{Yekutieli99}*{Lemma 2.1}.

\begin{lem} \label{lem/keyA1a}
	Let $m \in \Z$ be an integer such that $\Ho^i(\CK) = 0$ for any integer $i > m$.
 Then $\Ho^i(\CK \otimes \B) = 0$ for any integer $i > m$
	and 
	there is an isomorphism of $\Sx$-modules
	\begin{align}\label{eq:kuenneth0}
	\begin{td}
		\kappa^m \colon
	{\Ho}^{m}(\CK) \otimes \B \ar{r}{\sim} \& \Ho^{m}(\CK \otimes \B)\,.
	\end{td}
	\end{align}
\end{lem}
\begin{proof}
	Let $m \in \Z$ such that $\Ho^{i}(\CK)= 0$ for any $i > m$, so $\Ho^{-q}(\CK) = 0$ for any $q < -m$.
	Let	$E_{pq}^2$ denote the spectral sequence~\eqref{eq/kuenneth-spectral}.
	Then 
	$E_{pq}^2 = 0$ for any $p,q \in \Z$ with $p + q < -m$ or $p+q = -m$ and $p \neq 0$. 
	This implies that $E_{n} = \Ho^{-n}(\CK) = 0$ for any $n < -m$
	and $E_{0,-m}^2 \cong E_{-m}$. 
\end{proof}
\begin{rmk}\label{rmk/kuenneth}
	It can be shown that the isomorphism in
		\eqref{eq:kuenneth0}
	is the \emph{K\"{u}nneth map} which is given by
	\[
			\begin{td}  
	(x + \im d^{m-1}) \otimes y
	\ar[mapsto]{r}{\kappa^m}
	\& x \otimes y + \im(d^{m-1} \otimes \id)
	\end{td}
	\]
for any $x \in \ker d^m$ and $y \in \RI$.
\end{rmk}

\begin{prp}
	\label{prp/keyA1b}
	The following implications hold.
	\begin{align}\label{eq/keyA1b}
	\Ho^+(\CK) = 0
	\quad &\Leftrightarrow \quad 	\Ho^+(\CK \otimes \B) = 0	 \quad \Rightarrow \quad 
	\text{$\kappa^0$ is bijective,}
\\
	\label{eq/keyA1b2}
	\Ho^*(\CK) = 0 \quad &\Leftrightarrow \quad \Ho^*(\CK \otimes \B) = 0\,.
	\end{align}
\end{prp}
\begin{proof}
	Assume
	that 
	$\Ho^{+}(\CK \otimes \B) = 0$.
	The number
	$m \colonequals \max\{ i \in \Z \ | \ \Ho^i(\CK) \neq 0 \}$ is well-defined, because $\CK$ is right-bounded.
	Assume that  $m > 0$. Then  $\Ho^m(\CK) \otimes \B \cong \Ho^m(\CK \otimes \B) = 0$ by Lemma~\ref{lem/keyA1a}.
	Since $\ida \subseteq \mx$, Nakayama's Lemma 
	 yields the contradiction	
	$\Ho^m(\CK) = 0$. So $m\leq0$, that is, $\Ho^+(\CK) = 0$.
	
	The remaining implications in~\eqref{eq/keyA1b} follow from Lemma~\ref{lem/keyA1a} and Remark~\ref{rmk/kuenneth}. Finally, 
	the equivalence in~\eqref{eq/keyA1b2} follows from applications of 
	\eqref{eq/keyA1b} to shifts of $\CK$.
\end{proof}
\subsubsection*{Tor-rigid pairs and complexes with a single non-vanishing cohomology}

Next, we apply the observations of Subsection~\ref{subsec/spectral}.
\begin{lem}\label{lem/cohom1}
	Set $\sfM \colonequals \Ho^0(\CK)$.
	The following statements hold.
	\begin{enumerate}
		\item \label{lem/cohom1b} Assume that $\Ho^{\pm}(\CK)=0$. Then
		there is an $\B$-linear isomorphism 
		\begin{align*} 
		\Tor_i^{\Rx}(\sfM,\B) \cong \Ho^{-i}(\CK \otimes \B)
		\end{align*}
		for any integer $i \in \Z$, and thus 
		$\Tor_{+}^{\Rx}(\sfM,\B) = 0$ is equivalent to $\Ho^{\pm}(\CK \otimes \B) = 0$.
		\item \label{cor/cohom1a} 
		Assume that $\Ho^{\pm}(\CK \otimes \B) = 0$. Then $\Tor_1^{\Rx}(\sfM,\B) = 0$.\\
		If, moreover, $\Tor_+^{\Rx}(\sfM, \B) = 0$, then $\Ho^{\pm}(\CK) = 0$.
	\end{enumerate}
\end{lem}
\begin{proof}
	Let $E^2_{pq}$ be the Künneth spectral sequence~\eqref{eq/kuenneth-spectral}.
	\begin{enumerate}
		\item If $\Ho^+(\CK) = 0$, then $E^2_{pq}$ lies in the first quadrant, and the claims follow from Lemma~\ref{lem/spectral1}~\eqref{lem/spectral1b}.
		\item Assume that $\Ho^{\pm}(\CK\otimes \B)=0$.
		Proposition~\ref{prp/keyA1b}
		ensures that $\Ho^+(\CK)=0$. So $E^2_{pq}$ is a  first-quadrant sequence with $E_+ = 0$.
		Lemma~\ref{lem/spectral1}~\eqref{lem/spectral1a} states that $E_{10}^2 = 0$. 
		
		Assume, moreover, that $E^2_{+0}=0$.
		If $E^2_{0q} = 0$ for  $q \in \Z$, Nakayama's Lemma implies that  $\Ho^q(\CK)=0$, and thus 
		$E^2_{pq} = 0$ for any $p > 0$.
		In other words,
		the spectral sequence $E^2_{pq}$ has property~\ref{spectral/reflect}.
		Proposition~\ref{prp/spectral2}~\eqref{prp/spectral2a}
		implies that $E^2_{p+} = 0$. It follows that 
		$E^2_{0q} = \Ho^{-q}(\CK) = 0$ for any $q > 0$.
		\qedhere
	\end{enumerate}
\end{proof}
For any $\Rx$-module $\sfM$, statement~\eqref{lem/cohom1b}
recovers the fact  
that each torsion $\Tor_i^{\Rx}(\sfM,\B)$ can  be computed
using any flat resolution $\CK$ of $\sfM$.
Statement~\eqref{cor/cohom1a} 
motivates to
consider the following notion from commutative algebra.

\begin{dfn}\label{dfn/tor-rigid}
	An $\Rx$-module $\sfL$ is \emph{Tor-rigid} if 
	for any finitely generated $\Rx$-module $\sfM$ 
	\emph{the pair $(\sfM,\sfL)$ is Tor-rigid}, that is, for any integer $n > 0$ such that $\Tor_n^{\Rx}(\sfM,\sfL) = 0$ it follows that
	$\Tor_{m}^{\Rx}(\sfM,\sfL)=0$
	for any larger integer $m \geq n$. 
\end{dfn}

We will recover 
some well-known examples 
of Tor-rigidity 
in Corollary~\ref{cor/recover-tor-rigid}.
For now, we note that
Lemma~\ref{lem/cohom1} yields the following consequences for $\Ho^0$-complexes under change of rings.
\begin{cor}\label{cor/cohom1c}
	Set $\sfM \colonequals \Ho^0(\CK)$. The following statements hold.
	\begin{enumerate}
		\item \label{cor/cohom1ca}
	If $\Ho^{\pm}(\CK)= \Tor_+^{\Rx}(\sfM,\B) = 0$, then $\Ho^{\pm}(\CK\otimes\B) =  0$.	
\item	The converse of the previous statement is true if $\RI$ is Tor-rigid as an $\Rx$-module. \qed
\end{enumerate}
\end{cor}

\begin{rmk}\label{rmk/generalize}
	All statements of Subsection~\ref{subsec/kss}
	remain valid for any right-bounded complex $\CK$ and the quotient $\RI$ replaced by any $\Rx$-module $\Gamma$ such that 
	$\Tor_+^{\Rx}(K^i, \Gamma) = 0$
	and 
	$\Ho^i(\CK) \otimes \Gamma = 0$ for an integer $i \in \Z$ implies that $\Ho^i(\CK) = 0$.
\end{rmk}

Next, we consider
a variation of the last equivalence
which makes use of specific properties of the ring $\RI$.

\subsection{A variant of the local criterion of flatness for resolutions}

Since $\RI$ is the quotient $\Rx/\ida$ of the Noetherian ring $\Rx$ by an ideal $\ida$ contained in the Jacobson radical of $\Rx$, the following variant of the local criterion of flatness  holds true (see \cite{Matsumura}*{Theorem~22.3}).
\begin{thm}\label{thm/local-flatness}
	Let $\sfM$ be
	a finitely generated $\Rx$-module such that
	$\sfM \otimes \RI$ is flat as an $\RI$-module and 
	$\Tor_1^{\Rx}(\sfM,\RI)= 0$.
	Then $\sfM$ is a
	flat $\Rx$-module. \qed
\end{thm}

Focusing on the modules in the last theorem allows to 
deduce a variant of Corollary~\ref{cor/cohom1c}  without Tor-rigidity assumptions. We recall that the quotient ring $\RI$ has also an $\Sx$-algebra
structure over a commutative ring $\Sx$. 

\begin{prp}\label{prp/cohom4}
	Set	 $\sfM \colonequals \Ho^0(\CK)$ and $\sfL \colonequals \Ho^0(\CK \otimes \RI)$.
	Assume that $\RI$ is flat as an $\Sx$-module. The following statements hold.
	\begin{enumerate}
		\item \label{eq/prp/cohom4a}
	If $\Ho^{\pm}(\CK) = 0$ and $\sfM$ is $\Rx$-flat,
	then $\Ho^{\pm}(\CK \otimes {\RI}) = 0$
	and $\sfL$ is $\Sx$-flat.
	\item \label{eq/prp/cohom4b} The converse of the previous statement is true if $\RI = \Sx$.	
	\end{enumerate}
\end{prp}
\begin{proof}
	\begin{enumerate}
		\item Under the assumptions in~\eqref{eq/prp/cohom4a}
	Assume that $\Ho^{\pm}(\CK)=0$ and that $\sfM_{\Rx}$ flat.
	Corollary~\ref{cor/cohom1c}
	implies that
	$\Ho^{\pm}(\CK \otimes {\RI}) = 0$.
	Since $\Ho^+(\CK) = 0$,
	Lemma~\ref{lem/keyA1a}
	yields that the flat $\Sx$-module
	$\sfM \otimes {\RI}$ is isomorphic to the $\Sx$-module $\sfL$. 
	\item 
	Vice versa, assume that $\Ho^{\pm}(\CK \otimes {\RI}) = 0$, $\sfL_{\Sx}$ is flat
	 and $\Sx = {\RI}$, that is, $\Sx = \Rx/\ida$.
	Since $\Ho^{+}(\CK \otimes \RI)=0$,
	Proposition~\ref{prp/keyA1b}
	implies that
	$\Ho^+(\CK) = 0$ and
	$\sfM \otimes \RI \cong \sfL$ is a flat $\RI$-module.
	The first part of
	Lemma~\ref{lem/cohom1}~\eqref{cor/cohom1a} yields that $\Tor_1^{\Rx}(\sfM,\RI) = 0$.
	Since $\sfM \in \md \Rx$,
	Theorem~\ref{thm/local-flatness} implies that the $\Rx$-module $\sfM$ is flat. In particular,
	$\Tor_+^{\Rx}(\sfM,\RI) = 0$. Since $\Ho^{\pm}(\CK \otimes \RI) = 0$,
	Lemma~\ref{lem/cohom1}~\eqref{cor/cohom1a}
	yields that $\Ho^{\pm}(\CK) = 0$. 
	This shows the converse. \qedhere
	\end{enumerate}
\end{proof}

\subsection{Complexes with a single non-vanishing cohomology under change of rings}

The next goal is to deduce another refinement of  
Corollary~\ref{cor/cohom1c}.
We fix the following notions.

\begin{setup}\label{setup/tor-ind0}
	Let 
	$\sfM$ be any right $\Rx$-module, $\sfL$ an $(\Rx,\Sx)$-bimodule
	and $C$ a left $\Sx$-module.
\end{setup}
We recall that torsion commutes with coproducts.
\begin{lem}\label{lem/tor-commutes}
	For any family $(\sfL_i)_{i\in I}$ of right $\Sx$-modules 
	and any integer $m \in \Z$ 
	there is an isomorphism of abelian groups
	$$ \Tor_{m}^{\Sx}\big( \bigoplus_{i \in I} \sfL_i, C \big) \cong
	\bigoplus_{i \in I}  \Tor_m^{\Sx}(\sfL_i, C) $$
\end{lem}	
\begin{proof}
	Let us recall that tensor products commute with colimits, and that cohomology commutes with coproducts because
	monomorphisms of $\Sx$-modules are preserved by coproducts.
	This implies the claim.
\end{proof}

\begin{lem}
	It holds that
	$\Tor_+^{\Sx}(\sfL,C)=0$ if and only if $\Tor_+^{\Sx}(P \otimes_{\Rx} \sfL,C) = 0$ for any projective right $\Rx$-module $P$.
\end{lem}
\begin{proof}
	Setting $P \colonequals \Rx$ 
	yields the `only if'-implication.
	
	To show the converse, 
	assume that
	there is a set $I$ and a section $\iota\colon P \hookrightarrow \Rx^I \colonequals \bigoplus_{i \in I} \Rx$
	of right $\Rx$-modules.
	Together with
	Lemma~\ref{lem/tor-commutes}
	it follows that
	there is a section and an isomorphism of
	abelian groups  
	$$
	\Tor_n^{\Sx}(P \underset{\Rx}{\otimes} \sfL,C) \hookrightarrow
	\Tor_n^{\Sx}\big( (\Rx^I) \underset{\Rx}{\otimes} \sfL ,C\big) \cong 
	\bigoplus_{i \in I} \Tor_n^{\Sx}(\sfL,C)
	$$
	for any $n > 0$.
	Therefore, 
	$\Tor_{+}^{\Sx}(\sfL,C) = 0$ implies 
	that
	$\Tor_+^{\Sx}(P \otimes \sfL,C)=0$.
\end{proof}
The last lemma allows to reformulate 
 a
version of the Grothendieck spectral sequence 
\cite{Rotman}*{Theorem~10.60} as follows.
\begin{thm}\label{thm/grothendieck-sequence}
	In Setup~\ref{setup/tor-ind0} assume that
	$\Tor_+^{\Sx}(\sfL,C)=0$.
	Then there is a convergent 
	spectral sequence
	\begin{align}\label{grothendieck-sequence}
	E_{pq}^2 \colonequals \Tor_{p}^{\Sx}(\Tor_{q}^{\Rx}(\sfM,\sfL), C) \Rightarrow E_{p+q} \colonequals \Tor^{\Rx}_{p+q}(\sfM,\sfL \otimes_{\Sx} C),
	\end{align}
\end{thm}
Next, we apply 
one of the first observations
in Subsection~\ref{subsec/spectral}
to this spectral sequence.
\begin{lem} \label{lem/tor-ind1b}
	If $\Tor_+^{\Rx}(\sfM,\sfL)=\Tor_{+}^{\Sx}(\sfL,C) = 0$,
	then 
	$\Tor_+^{\Sx}(\sfM \otimes_{\Rx} \sfL,C) = 0$ is equivalent to 
	$\Tor_+^{\Rx}(\sfM, \sfL \otimes_{\Sx} C) = 0$.
\end{lem}
\begin{proof}
	Let $E_{pq}^2$ be the spectral sequence from~\eqref{grothendieck-sequence}.
	Since $E^2_{p+} = 0$, 
	the condition $E^2_{+0} = 0$ is equivalent to $E_+ = 0$ by Lemma 
	\ref{lem/spectral1}~\eqref{lem/spectral1b}. This translates into the claim.
\end{proof}
Now, Corollary~\ref{cor/cohom1c} can be extended as follows.
We recall that the quotient ring $\RI$ could be viewed as $\Sx$-algebra.
\begin{prp}\label{prp/cohom1d}
	Let $\CK$ be a complex as in Setup~\ref{setup/complexes}
	and $C$ an $\Sx$-module such that
	$\Tor_{+}^{\Sx}(\RI,C) = 0$. 
	Set 
	$\sfM \colonequals \Ho^0(\CK)$, $\sfL \colonequals \Ho^0(\CK \otimes \RI)$ and ${C}' = \RI \otimes_{\Sx} C$.
	The following statements hold.
	\begin{enumerate}
		\item If $\Ho^{\pm}(\CK)= \Tor_+^{\Rx}(\sfM,{C}' \oplus \RI) = 0$, then  $\Ho^{\pm}(\CK \otimes \RI) = \Tor_+^{\Sx}(\sfL, C) = 0$.	
		\item The converse of the previous statement is true if $\RI$ is Tor-rigid as $\Rx$-module.
		\end{enumerate}
\end{prp}
\begin{proof}
	In both cases, 
	there is an $\Sx$-linear isomorphism $\sfL \cong \sfM \otimes_{\Rx} \RI$ by Proposition~\ref{prp/keyA1b}.
	Both implications follow from  Corollary~\ref{cor/cohom1c} and Lemma~\ref{lem/tor-ind1b}.
\end{proof}

\section{Ascent and descent of presilting and pretilting complexes}
\label{sec/ascent-descent}

We recall our main setup in a more compact formulation.
\begin{setup}\label{setup/main-copy}
	As before, $\Rx$ is a complete local Noetherian ring with maximal ideal $\mx$, $\B \colonequals \Rx/\ida$ its quotient with respect to an ideal $\ida \subseteq \mx$ of $\Rx$ and
	$\A$ a Noetherian $\Rx$-algebra
	such that $\Tor_{+}^{\Rx}(\A,\B) = 0$.
	We set $\AI \colonequals \A/\ida \A$
	and denote
	$$
	\begin{td} \FF\colon \D^{-}(\md \A) \ar{r} \& \D^{-}(\md \AI), \& \CM \ar[mapsto]{r} \& \PCM \colonequals \CM \overset{\mathbb{L}}{\otimes}_{\A} \AI \end{td} 
	$$
\end{setup}

In the first subsection, we apply 
the results of the previous section on cohomological vanishing to deduce equivalences 
\begin{align*}
\CL \geq \CL \quad \Leftrightarrow \quad \PCL \geq \PCL \qquad\text{and}\qquad
\CL \teq \CL
\quad \Leftrightarrow \quad
\PCL \teq \PCL
\end{align*}
under certain assumptions on $\B$ 
and a complex $\CL$ of $\A$.

In the second subsection, we show that the functor $\FF$
gives rise to an embedding of posets
$$
\begin{td} f_s\colon \silt \A \ar[hookrightarrow]{r} \& \silt \AI\,. \end{td}
$$

In the following, we will say that the functor $\FF$ is \emph{full at a pair $(\CL,\CM)$} of complexes from $\D^-(\md\A)$ 
if the map 
\begin{align}\label{eq/density}
\begin{td} 
\Hom_{\D(\A)}(\CL,\CM) \ar{r} \&  \Hom_{\D(\Abar)}(\PCL,\PCM), \& 
\phi \ar[mapsto]{r} \& \ol{\phi} \colonequals \FF(\phi) \end{td}
\end{align}
is surjective. 
The functor $\FF$ is \emph{full at $\CL$} 
if it is full at the pair $(\CL,\CL)$. 

\subsection{Ascent and descent of presilting and certain pretilting complexes}
\label{sec/ascent}
The next statement provides key of the change-of-rings functor $\FF$.

\begin{prp}\label{prp/keyA}
	Any complexes $\CL \in \per \A$ and $\CM \in \D^-(\md \A)$
	satisfy the following implications.
	\begin{align}\label{eq/key-implications}
	\CL \geq \CM 
	\quad &\Leftrightarrow \quad
	\PCL \geq \PCM
	\quad \Rightarrow \quad
	\FF\text{ is full at }(\CL,\CM) \\
	\label{eq/perp}
	\CL \perp \CM \quad &\Leftrightarrow \quad \PCL \perp \PCM \end{align}
	Moreover, if $\CL$ is a presilting complex of $\A$ there is an isomorphism of $\Sx$-algebras
	\begin{align*}
	\begin{td}
	\End_{\D(\A)}(\CL) \otimes \B \ar{r}{\sim} \& \End_{\D(\Abar)}(\PCL)\,.
	\end{td}
	\end{align*}
\end{prp}

\begin{proof}
	Set $\CK \colonequals \Hom^{\bt}_{\A}(\CP,\CQ)$
	for projective resolutions $\CP$ of $\CL$ and $\CQ$ of $\CM$.  
	By Lemma~\ref{lem/translation}~\eqref{lem/translations4} the complex $\CK$ 
	satisfies the assumptions of Setup~\ref{setup/complexes}.
	Therefore,
	the implications ~\eqref{eq/keyA1b}
	in
	Proposition~\ref{prp/keyA1b} translate into the implications
	\begin{align*}
	\CL \geq \CM \quad \Leftrightarrow \quad
	\PCL \geq \PCM  \quad \Rightarrow \quad
	\begin{td}
	{\Hom_{\D(\A)}(\CL,\CM)\otimes\B}
	\ar{r}{\gamma}[swap]{\sim} \& 
	\Hom_{\D(\Abar)}(\PCL,\PCM)
	\end{td}
	\end{align*}
	and
	the equivalence
	\eqref{eq/keyA1b2}  into ~\eqref{eq/perp} 	
	via Lemma~\ref{lem/translation}~\eqref{lem/translations1} and~\eqref{lem/translations2}.
	If $\gamma$ is bijective, then $\FF$ is full at $(\CL,\CM)$ 
	according to diagram 
	\eqref{eq/morphisms}.
	
	In case $\CL = \CM$ is presilting, $\gamma$ is an $\Sx$-algebra isomorphism by Lemma~\ref{lem/translation}~\eqref{algebra-morphism}.
\end{proof}

By~\eqref{eq/key-implications},
the functor $\FF$ preserves and reflects $\geq$ on perfect complexes.
This was also shown by Eisele \cite{Eisele}*{Proposition~6.1}
in case the Noetherian $\Rx$-algebra $\A$ 
is $\Rx$-free over a complete discrete valuation ring 
$\Rx$ and $\ida$ is its maximal ideal.

\begin{rmk}
	In different terms, the main cornerstone of Proposition~\ref{prp/keyA}
	is the isomorphism of functors
	$$
	\mathbb{R}\Hom_{\A}(\dash,\dash) \overset{\mathbb{L}}{\otimes} \B \cong \mathbb{R}\Hom_{\Abar}(\dash \overset{\mathbb{L}}{\otimes} \B,\dash \overset{\mathbb{L}}{\otimes} \B)\colon \per \A \times \D^-(\md \A) \longrightarrow \D^-(\md \Rx).
	$$
	which holds because of the assumption $\Tor_+^{\Rx}(\A,\B) = 0$ in~\eqref{eq/tor-ind0}.
\end{rmk}

The last proposition implies a derived version of Nakayama's Lemma.
\begin{cor}
	\label{cor/derived-Nakayama}
	The functor $\FF\colon \D^-(\md \A) \longrightarrow \D^-(\md \AI)$ reflects zero objects.
\end{cor}
\begin{proof}
	Let $\CM \in \D^-(\md \A)$ such that
	$\PCM \cong 0$ in $\D^-(\md \AI)$. 
	Since
	$\AI \perp \PCM$,
~\eqref{eq/perp} yields that $\A \perp \CM$, that is, $\CM \cong 0$ in $\D^-(\md \A)$. 
\end{proof}

The next statement concerns ascent and descent of tilting complexes.
\begin{prp}\label{prp/pretilting}
	Let $\CT \in \per \A$
	and $\sfN$ be an $\RI$-module.
Set $\sfM \colonequals \End_{\D(\A)}(\CT)$ and $\sfL \colonequals \End_{\D(\Abar)}(\PCT)$.
		The following statements hold.
	\begin{enumerate}		
		\item  \label{prp/keyA2}
		If $\CT \teq \CT$, then 
		there is an
		isomorphism of $\RI$-modules
		\begin{align*}
		\Tor_i^{\Rx}(\sfM, \B) \cong 		\Hom_{\D(\Abar)}(\PCT,\PCT[-i])
		\end{align*}
	for any integer $i \in \Z$, and thus $\Tor_+^{\Rx}(\sfM, \B)=0$ is equivalent to $\PCT \teq \PCT$.
		\item  \label{prp/keyA2b1}
		$\CT \teq \CT$
		 and 
		$\sfM$  is flat over $\Rx$ if and only if $\PCT \teq \PCT$ and $\sfL$ is flat over $\RI$.
		\item \label{prp/keyA2b2}
		If $\CT \teq \CT$ and $\Tor_+^{\Rx}(\sfM, N \oplus \RI) = 0$, then $\PCT \teq \PCT$ and $\Tor_+^{\RI}(\sfL, \sfN) = 0$.
		\item The converse of the previous statement is true if $\RI$ is Tor-rigid as $\Rx$-module.
	\end{enumerate}
\end{prp}
\begin{proof}
	The first claim follows from Lemma~\ref{lem/cohom1},
	the remaining implications
	from Propositions~\ref{prp/cohom4} and~\ref{prp/cohom1d}
	using Lemma~\ref{lem/translation} applied to $\CK \colonequals \Hom_{\A}^{\bt}(\CT,\CT)$ and $\Sx \colonequals \RI$
	assuming that $\CT \in \Hotb(\proj \A)$.
\end{proof}
If we choose $N \colonequals 0$ in~\eqref{prp/keyA2b2}, we obtain that
$\FF$ preserves pretilting complexes with certain endomorphism rings to pretilting complexes, and that $\FF$ reflects pretilting complexes assuming Tor-rigidity of $\RI$.

\begin{cor}\label{cor/tor-rigid-endo}
		Let $\CT$ be a pretilting complex of $\A$
		and set $\sfM \colonequals\End_{\D(\A)}(\CT)$.
		If $\RI$ is Tor-rigid as $\Rx$-module, it holds that 
		$$
		\qquad\qquad
		\Hom_{\D(\Abar)}(\PCT,\PCT[-1]) = 0 \quad 
		\text{if and only if} 
		\quad 
		\Tor_{+}^{\Rx}(\sfM,\RI) = 0\,. \qquad \qquad \, \qed$$
	\end{cor}

\subsection{A silting embedding result}

Next, we show three simple lemmas assuming fullness of the functor $\FF$ at certain pairs. The proofs use elementary category theory.

\begin{lem}\label{lem/uniqueness1}
	Assume that the functor $\FF$ is full at a pair $(\CL,\CM)$ of complexes from $\D^-(\md \A)$.
	Then $\CL \cong \CM$ 
	if and only if 
	$\PCL \cong \PCM$.
\end{lem}
\begin{proof}
	By assumption any isomorphism
	$ \phi\colon \PCL \overset{\sim}{\longrightarrow} \PCM$ 
	has a preimage
	$\alpha\colon \CL \longrightarrow \CM$ 
	under $\FF$. 
	Since $\FF$ reflects zero objects,
	$\FF(\cone(\alpha)) \cong \mathrm{cone}(\phi) \cong 0$ implies that $\cone(\alpha) \cong 0$.
	So $\CL \cong \CM$. The converse is true since $\FF$ is a functor.
\end{proof}

\begin{lem} \label{lem/summands1}
	Assume that the functor $\FF$ is full at a complex $\CL \in \per \A$. 
	For any direct summand $\CY$ of $\PCL$ there is a direct summand $\CX$ of $\CL$ such that $\PCX \cong \CY$.
	In particular, $\CL$ is indecomposable if and only if $\PCL$ is 
	indecomposable.
\end{lem}
\begin{proof} 
	
	Let $\CY$ be a direct summand of $\PCL$.
	Then there is a morphism
	$\iota\colon \CY \hookrightarrow \PCL$
	which has a left-inverse $\pi$.
	By the assumptions, $\FF$ induces  a surjective ring homomorphism $A \colonequals \End_{\D(\A)}(\CL) \twoheadrightarrow \ol{A} \colonequals \End_{\D(\Abar)}(\PCL)$.
	Because the ring $A$ is semiperfect, 
	there is an idempotent
	$\varepsilon \in A$ such that $\ol{\varepsilon} = \iota \cdot \pi$ by results of Nicholson \cite{Nicholson}*{Corollary 1.3, Proposition~1.5}.
	Since $\per \A$ is idempotent-complete, there is an object $\CX \in \per \A$ and a morphism
	$\alpha\colon \CX \hookrightarrow \CL$
	with a left-inverse $\beta$
	such that 
	$\alpha \cdot \beta = \varepsilon$.
	Then
	$\pi \cdot \ol{\alpha}\colon \PCX \longrightarrow \CY$ is an isomorphism with the inverse $\ol{\beta} \cdot  \iota$.
	
	Assume that $\CL$ is indecomposable. The preceding argument shows that $\PCL$ does not have proper direct summands. As $\FF$ reflects zero, $\PCL$ must be indecomposable.
\end{proof}

\begin{lem}\label{lem/summands2}
	Assume that the functor $\FF$ is full 
	on a perfect complex $\CL$ of $\A$.
	Then its restriction $\FF'\colon \add \CL \longrightarrow \add \PCL$ is full, essentially surjective, preserves indecomposability  and reflects isomorphism classes of objects.
	In particular, it holds that $|\CL| = |\PCL|$ and $\CL$ is basic if and only if $\PCL$ is basic.
\end{lem}
\begin{proof}	
	Since the Hom-functor is biadditive, the restriction $\FF'$ is full. In particular, 
	$\FF'$ is essentially surjective and preserves indecomposability by Lemma~\ref{lem/summands1}.
	Moreover, $\FF'$ reflects isomorphism classes of objects by Lemma~\ref{lem/uniqueness1}.
\end{proof}

The proof of 
the next lemma uses that $\FF$ is related to a morphism of rings.
\begin{lem}\label{lem/generators}
	The functor $\FF$ preserves  perfect generators.
\end{lem}
\begin{proof}
	Let $\CM$ be a generator of $\per \A$.
	Since $\A \in \langle \CM \rangle$ 
	and $\FF$ is a functor which commutes with arbitrary coproducts, cones and shifts up to isomorphism, it follows that
	$\AI \cong \FF(\A) \in \FF(\langle \CM \rangle) \subseteq \langle \PCM \rangle$. Thus $\per \AI = \langle \FF(\A) \rangle = \langle \PCM \rangle$.
\end{proof}
Similar to $\silt \A$, we denote by $\presilt \A$
the set of isomorphism classes of basic presilting complexes.
The central result of this subsection is the following.
\begin{prp}\label{prp/ps-inj}
	The functor $\FF$ 	
	induces well-defined injective maps
	$$
	\begin{td}
	f_{ps}\colon
	\presilt \A \ar[hookrightarrow]{r} \& \presilt \AI
	\end{td}
	\quad \text{ and }\quad 
	\begin{td}
	f_s\colon \silt \A \ar[hookrightarrow]{r} 
	\&	 \silt \AI
	\end{td}
	$$
	which preserve and reflect the relation $\geq$.
\end{prp}
\begin{proof}
	Let $\CL$ be a presilting complex of $\A$.
\eqref{eq/key-implications} implies that 
	$\FF$ is full at $\CL$ and that	$\PCL \geq \PCL$.
	Lemma~\ref{lem/summands2} yields that
	$f_{ps}$ is well-defined.
	
	Let $\CM$ be a presilting complex of $\A$ with $\PCL \cong \PCM$.
	Then $\FF$ is full at $(\CL,\CM)$ by Proposition~\ref{prp/keyA}, which implies that $\CL \cong \CM$ using Lemma~\ref{lem/uniqueness1}.
	This shows that $f_{ps}$ is injective.
	
	Its restriction $f_s$ is well-defined because of Lemma~\ref{lem/generators}. 
	The functor $\FF$ preserves and reflects $\geq$ 
	on all perfect complexes by~\eqref{eq/key-implications}.
	\qedhere
\end{proof}

To show that the map $f_{ps}$ is bijective it is sufficient to show that any complex $\CP \in \presilt \AI$ 
has a preimage $\CL\in \per\A$ under the functor $\FF$.
This problem is the topic of the next section.
However, to show that the map $f_s$ is bijective we need also to show that $\PCL \in \silt \AI$ implies that $ \CL$ generates $\per \A$. This requires an alternative characterization of silting complexes which is derived in Section~\ref{sec/characterization}.

\section{Lifting techniques}\label{sec/lifting}

This section deals with the question
whether for
a given complex $\CPb \in \D^-(\md {\rAI})$ there exists a complex $\CLb \in \D^-(\md\rA)$ such that there is an isomorphism $\PCL \cong \CPb$ in the category $\D^-(\md{\rAI})$. 
We will call this question the \emph{derived lifting problem}, and say that the complex $\CPb$ has a \emph{lift} $\CL$.

The derived lifting problem has been studied by Rickard \cite{Rickard91b} in a setup motivated by group theory and by Yoshino \cite{Yoshino} in commutative algebra. Both works use a technique which goes back to Eisenbud \cite{Eisenbud}.
In the present section, we adapt Rickard's arguments in order to derive a lifting criterion 
for quotients ${\rAI}$ with respect to an ideal $I \subseteq \rad \A$. 
Then we specialize to a common denominator of Rickard's and Yoshino's setup in which
any complex $\CP$ without second self-extensions admits a lift.
This conclusion is related to Tor-rigidity of the quotient $\Rx/\ida$ and extends to an abelian version of the lifting problem.

This section does not require any results of other parts of the paper.
Throughout this section, we use the assumptions and notation of Setup~\ref{setup/main-copy}.

\subsection{Minimal complexes and homotopy categories}

\begin{notation}
	Set ${\rAI}_n \colonequals \A/\ida^n \A$ for any integer $n > 0$, and $I \colonequals \ida \A$. 
\end{notation}
In particular, it holds that ${\rAI}_1 \cong {\rAI}$ and  ${\rAI}_n \colonequals\A/I^n$. The forthcoming lifting techniques are based on the following well-known properties of our setup.
\begin{lem}\label{lem/lifting}
	The following statements hold.
	\begin{enumerate}
		\item It holds that $I \subseteq \rad \A$
		and the ring $\rA$ is $I$-adically complete, that is, $\rA \cong \varprojlim {{\rAI}}_n$.
		\item \label{lem/complete-full}
		For any integer $n > 0$  
		the functor $\dash \otimes_{\rA} {{\rAI}}_n \colon \proj \rA \longrightarrow \proj {{\rAI}}_n$ is essentially surjective and full. 
	\end{enumerate}
\end{lem}
\begin{proof}
	\begin{enumerate}
		\item
		By \cite{Lam}*{(5.9)} it holds that $I \subseteq \A \mx = \A  \rad\Rx  \subseteq \rad \A$.
		The second claim holds because $\A$ is a Noetherian $\Rx$-algebra over the $\mx$-adically complete local ring $\Rx$, see \cite{Lam}*{(21.34)}. 
		\item The second claim can also be derived from known results. Alternatively, it follows from 
				the observations that $\proj \rA = \add \rA$ and $\proj {{\rAI}}_n = \add {{\rAI}}_n$, where $n > 0$, together with
		Lemma~\ref{lem/summands2}.
		\qedhere
	\end{enumerate}
\end{proof}
Next, we will rephrase the derived lifting problem in terms of the right-bounded categories of complexes $\Kom^-(\proj \rA)$ and $\Kom^-(\proj \rAI)$. 
This is possible because the rings $\A$ and ${\rAI}$ are semiperfect.
The next statements are well-known.
\begin{itemize}
	\item 
	Any complex $\CP \in \D^-(\md {{\rAI}})$ 
	is quasi-isomorphic to a \emph{minimal} complex $\CN \in \Kom^-(\proj \rA)$, which means that $\im d^i_{\CN} \subseteq \rad N^{i+1}$ for any integer $i \in \Z$.
	\item 
	Any two minimal complexes $\CN_1,\CN_2 \in \Kom^-(\proj {{\rAI}})$ are \emph{homotopy equivalent} if and only if they are \emph{isomorphic as complexes}.
\end{itemize}
We will say that a complex $\CN \in \Kom^-(\proj {{\rAI}})$ \emph{lifts} to a complex $\CM \in \Kom^-(\proj \rA)$ if there exists an isomorphism $\CM \otimes_{\rA} {\rAI} \cong \CN$ of complexes. It is easy to check the following.
\begin{itemize}
	\item The non-zero terms of $\CM$ and $\CN$ are 
	located at the same degrees.
	\item
	The complex $\CN$ is minimal if and only if its lift $\CM$ is minimal.
\end{itemize}
With these observations the following can be shown.
\begin{lem}	\label{lem/minimal-lift}	
	Let $\CPb \in \D^-(\md {\rAI})$ and $\CNb \in \Kom^-(\proj {\rAI})$ denote its minimal version. The following conditions are equivalent.
	\begin{enumerate}
		\item 
		The complex $\CPb$ lifts to a complex $\CLb \in \D^-(\md \rA)$.
		\item 
		The complex $\CNb$ lifts to a minimal complex $\CMb \in \Kom^-(\proj \rA)$. 
	\end{enumerate}
	In this case, there is an isomorphism $\CLb \cong \CMb$ in $\D^-(\md \rA)$ and
	$\CP$ is perfect if and only if $\CL$ is perfect.
	\qed
\end{lem}

\subsection{A lifting criterion in terms of higher conormal bimodules}
A frequently used idea to lift a given complex $\CP$ of projective ${\rAI}$-modules
is to lift from ${\rAI}_n$ to ${\rAI}_{n+1}$ by induction on $n \in \N$ and form the inverse limit of these iterated lifts.
In different terms, a lift might be obtained by completion.
For each $n > 0$ there may be an obstruction at lifting step $n$ which is related to the ${\rAI}$-bimodule $I^n/I^{n+1}$, the \emph{$n$-th conormal bimodule of the ideal $I$}.
There is the following coarse criterion to nullify each obstruction.
\begin{prp}\label{prp/twisted-lift}
	Let
	$\CP \in \D^-(\md {\rAI})$ be a complex
	such that 
	\begin{align}
	\label{eq/conormal-twists}
	\Hom_{\D(\Abar)}(\CP, \alpha_n(\CP) [2]) = 0\,, \quad \text{where }\quad \alpha_n(\CP) \colonequals  \CP \underset{{\Abar}}{\overset{\mathbb{L}}{\otimes}} I^n/I^{n+1}\,,
	\end{align} for any integer $n > 0$.
	Then $\CP$ has a lift $\CL \in \D^-(\md \A)$.
\end{prp}
The proof is a straightforward adaptation of Rickard's arguments \cite{Rickard91b}*{Proposition~3.1}. We include the proof for the sake of a self-contained presentation.
\begin{proof}
	We may assume that $\CP$ is given by a complex from $\Kom^-(\proj \A)$. 
	The arguments below are carried out \emph{at each degree $i \in \Z$} without further notice.
	According to Lemma~\ref{lem/lifting}~\eqref{lem/complete-full}  
	each projective ${\rAI}$-module $P^i$ has a lift $L^i \in \proj \A$.
	Set $L^i_n \colonequals L^i/L^i I^n$ for any $n > 0$.
	The differentials $d^i\colon P^i \longrightarrow P^{i+1}$ will be lifted to differentials $d^i_n\colon L^i_n \longrightarrow L^i_{n}$ by induction on $n > 0$.
	\begin{itemize}
		\item
		If $n = 1$, we may use the isomorphisms $L^i_1 \cong P^i$ to define the differentials $d^i_1$.
		\item
		Assume that 
		the differentials $d^i_n$ have been constructed for an integer $n > 0$.
		To simplify notation, set $M^i \colonequals L^i_{n+1}$, $N^i \colonequals L^i_n$ and $U^i \colonequals M^i I^n$.
		Because of Lemma~\ref{lem/lifting}~\eqref{lem/complete-full}
		each differential $d^i_n\colon N^i \longrightarrow N^{i+1}$ 
		lifts to a map $\delta^i\colon	 M^i \longrightarrow M^{i+1}$.
		This yields the commutative diagram on the left:
		\begin{align*}
		\begin{td} 
		M^i \ar{r}{\delta^i} \ar[twoheadrightarrow]{d} \& M^{i+1} \ar[twoheadrightarrow]{d} \ar{r}{\delta^{i+1}} \& 
		M^{i+2} \ar[twoheadrightarrow]{d} \\
		N^i \ar{r}{d^i_n} \& N^{i+1} \ar{r}{d^{i+1}_n} \& 
		N^{i+2} 
		\end{td}
		\qquad \qquad
		\begin{td} 
		M^i \ar{r}{\triangle^i} 
		\ar[densely dotted]{rd}
		\ar[twoheadrightarrow]{d}[swap]{\pi^i} \& M^{i+2} \\
		P^i \ar[dashed]{r}{\varepsilon^i} \& U^{i+2} \ar[hookrightarrow]{u}[swap]{\iota^{i+2}}
		\end{td}
		\end{align*}
		Any of the compositions
		$\triangle^i\colonequals\delta^{i+1}\delta^i$ might not vanish, 
		so the maps $\delta^i$ need to be replaced with proper differentials $d^i_{n+1}$.
		
		Because the maps $d^i_n$ are differentials, the image of each map  $\triangle^i$ is contained in $U^{i+2}$.
		Furthermore, each map $\triangle^i$ factors through a unique ${\rAI}$-linear map $\varepsilon^i$ as shown in the right diagram above, 
		since $\ker \pi^i  = M^{i} I \subseteq
		\ker {{\triangle}}^i$.
		It follows that the bottom squares of the following diagram commute.
		\begin{align*}
		\begin{td}
		M^{i-1} \ar{rr}[pos=0.5]{\delta^{i-1}} \ar{rd}[pos=0.6]{\triangle^{i-1}} \ar[twoheadrightarrow]{dd}[swap]{\pi^{i-1}} \&[0.0cm]\&
		M^{i} \ar{rr}[pos=0.5]{\delta^i}
		\ar[densely dotted]{ld}{\sigma^i}
		\ar[->>]{dd}
		\ar{rd}[pos=0.6]{\triangle^{i} } \&[0.0cm] \&
		M^{i+1}   \ar{rd}{\triangle^{i+1} } 
		\ar[->>]{dd}
		\ar[densely dotted]{ld}{\sigma^{i+1}}
		\&[0.0cm]  \\
		\& M^{i+1}
		\ar{rr}[pos=0.66]{\delta^{i+1}} \&
		\&
		M^{i+2}
		\ar{rr}[pos=0.33]{\delta^{i+2}} 
		\& \&
		M^{i+3}  \\
		{P^{i-1}} \ar{rr}[pos=0.33]{d^{i-1}} \ar[dashed]{rd}{\varepsilon^{i-1}} 
		\&  \&
		P^{i} \ar[dashed]{rd}{\varepsilon^{i}} 
		\ar{rr}[pos=0.33]{d^i_1}
		\ar[densely dotted]{ld}{s^{i}}
		\& \&  P^{i+1} \ar[dashed]{rd}{\varepsilon^{i+1}} 
		\ar[densely dotted]{ld}{s^{i+1}}
		\\
		\& U^{i+1} \ar{rr}[swap,pos=0.5]{\delta^{i+1}_{U}} 
		\ar[hookrightarrow]{uu}
		\& 
		\& 
		U^{i+2}  
		\ar{rr}[swap,pos=0.5]{\delta^{i+2}_{U} }
		\ar[hookrightarrow]{uu}
		\& \& U^{i+3}  \ar[hookrightarrow]{uu}[swap]{\iota^{i+3} }
		\end{td}
		\end{align*}
		Since $(\ol{I^n})^2 = 0$, the restrictions  $\delta^i_{U} \colonequals {\delta^i}|_{U^i}\colon U^i \longrightarrow U^{i+1}$ are {differentials}.
		So the maps $\varepsilon^i$ define a \emph{morphism
			$\varepsilon\colon \CP \longrightarrow U^{\bt}[2]$ of complexes} of ${\rAI}$-modules.
		Because of the isomorphisms
		$U^i \cong L^i \otimes_{\A} \ol{I^n}  \cong	P^i \otimes_{{\Abar}} \ol{I^n}$,
		there is an isomorphism $U^{\bt} \cong \CP \otimes_{{\Abar}} \ol{I^n}$ of complexes.
		
		The assumption~\eqref{eq/conormal-twists}  implies  that 
		the morphism
		$\varepsilon$ is homotopic to zero via some ${\rAI}$-linear maps $s^i$ displayed above.
		Set
		$\sigma^i \colonequals \iota^{i+1} s^i \eta^i$.
		This
		yields $\A_{n+1}$-linear maps 
		such that
$\im \sigma^i \subseteq U^{i+1}$, 
		$\triangle^i = \sigma^{i+1} \delta^i + \delta^{i+1} \sigma^i$
		and
		$\sigma^{i+1} \sigma^i = 0$.
		These properties imply that the maps
		$d^{i}_n$ lift to differentials
		$d^i_{n+1} \colonequals \delta^i - \sigma^i$.
	\end{itemize}
	Since $\A$ is $I$-adically complete,
	taking the inverse limit of the differentials
	$d^i_n$ 
	yields 
	a lift $\CL \in \Kom^{-}(\proj \A)$ of the complex $\CP$ in the sense that $\CL \otimes_{\A} {\rAI} \cong \CP$. In particular, $\PCL \cong \CP$ in $\D^-(\md {\rAI})$.
\end{proof}

The last proposition
allows to lift 
any complex $\CP \in \D^-(\md {\rAI})$
with $\CP \geq \alpha_n(\CP)$ for any integer $n > 0$
to a complex of $\D^-(\md \A)$.
In each of the next two subsections we will present a setup which leads to the last condition in case $\CP$ is a presilting complex.

In this setup, it will play a role that ${\rAI}$ is the quotient $\A/\ida \A$ 
with respect to a proper ideal $\ida$ of the commutative ring $\Rx$
and the ideal $I$ is given by $\A \ida$. 
The $n$-th conormal bimodule of the ideal $\ida$ 
is closely related to the $n$-th conormal bimodule of the ideal $I$ if certain torsion modules vanish.
\begin{lem}\label{lem/algebra-conormals}
In Setup~\ref{setup/main}, the following statements hold.
	\begin{enumerate}
		\item \label{tor-relations} For any $\Rx$-module $M$ the following statements are equivalent.
		\begin{enumerate}[label=(\roman*)]
			\item \label{tor-rel1} It holds that
			$\Tor_{+}^{\Rx}(M,\Rx/\ida) = 0$ and $ \Tor_+^{\Rx}(M, \ida^n/\ida^{n+1}) = 0$ for any integer $n > 0$.
			\item \label{tor-rel2} It holds that	$\Tor_{+}^{\Rx}(M, \Rx/\ida^n) = 0$ for any integer $n  > 0$.
		\end{enumerate}
		\item \label{conormal-free} Let $n > 0$ be an integer such that 
		$
		\Tor_1^{\Rx}(\A, \Rx/\ida^{n}) = 0$. 
		Then there is an isomorphism 
		of
		${\rAI}$-bimodules
		$$
		\begin{td} \psi_n\colon \A \otimes \ida^n/\ida^{n+1} \ar{r}{\sim} \& I^n/I^{n+1},\&
		a \otimes (b + \ida^{n+1}) \ar[mapsto]{r} \& ab + I^{n+1}\,.
		\end{td} 
		$$
	\end{enumerate}
\end{lem}
\begin{proof}
	\begin{enumerate}
		\item
		For any $\Rx$-module $M$
		and 
		any integer $n > 0$  the short exact sequence of $\Rx$-modules
		$$
		\begin{td} 0 \ar{r} \& \ida^n/\ida^{n+1} \ar{r}{\iota_n} \& \Rx/\ida^{n+1} \ar{r}  \& \Rx/\ida^n \ar{r} \& 0
		\end{td}
		$$
		gives rise to a long exact $\Tor$-sequence
		\begin{align*}
		\begin{td} 
		\ldots \Tor_{i+1}^{\Rx}(M,\ida^n/\ida^{n+1})  \ar{r}
		\& \Tor_{i+1}^{\Rx} (M,\Rx/\ida^{n+1}) \ar{r}
		\& \Tor_{i+1}^{\Rx} (M,\Rx/\ida^n) \phantom{\ldots} \ar{lld} \\
		\phantom{\ldots}\Tor_i^{\Rx}(M,\ida^n/\ida^{n+1})  \ar{r}
		\& \Tor_{i}^{\Rx} (M,\Rx/\ida^{n+1}) \ar{r}
		\& \Tor_{i}^{\Rx} (M,\Rx/\ida^n)\ldots
		\end{td}
		\end{align*}
		which implies that~\ref{tor-rel1} is equivalent to~\ref{tor-rel2}.
		\item Let $n > 0$.
		It is straightforward to check that the map $\psi_n$ is surjective and ${\rAI}$-bilinear. It appears in the commutative diagram of $\Rx$-modules
		$$
		\begin{td}
		\A \otimes \ida^n/\ida^{n+1} \ar{r}{\id \otimes \iota_n}
		\ar[->>]{d}{\psi_n}
		\&[0.25cm] 
		\A \otimes \Rx/\ida^{n+1}  
		\& a \otimes (b + \ida^{n+1})  \ar[mapsto]{r}  \ar[mapsto]{d} \& a \otimes (b + \ida^{n+1})  \ar[mapsto]{d}  \\
		I^n/I^{n+1} \ar[hookrightarrow]{r} 
		\& \A/I^{n+1} \ar[<-]{u}[rotate=-90,xshift=6pt, yshift=4pt]{\sim}
		\& ab + I^{n+1} \ar[mapsto]{r} \&  ab + I^{n+1}\,.
		\end{td}
		$$
		If $\Tor^{\Rx}_1(\A,\Rx/\ida^n) = 0$, the maps $\id \otimes \iota_n$, 
		and thus $\psi_n$ are both injective. 
		\qedhere
	\end{enumerate}
\end{proof}

\subsection{Normally flat rings}

In addition to Setup~\ref{setup/main}, we assume the following condition in this subsection.
\begin{dfn}\label{dfn/normally-flat}
	The local ring $\Rx$ is \emph{normally flat along the ideal $\ida$} if
	$\ida^n/\ida^{n+1}$
		 is flat as $\RI$-module for any integer $n > 0$.
\end{dfn}
This notion has appeared in context of resolution of singularities, see \cite{HIO}.
We note that a finitely generated flat $\RI$-module is already free.

\subsubsection*{Tor-rigidity of the quotient of the base ring}
Before we reconsider the derived lifting problem, we 
show that the last notion
is related to Tor-rigidity
of the quotient $\RI = \Rx/\ida$ 
over several rings.
\begin{prp}\label{prp/conormal-rigid}
	Assume that $\Rx$ is normally flat along the ideal $\ida$ in the sense of Definition~\ref{dfn/normally-flat}.
	Then $\RI$ is Tor-rigid
	as a module over the ring $\Rx$ as well as a module over the quotient
	$\Rx/\ida^m$ for any integer $m > 0$.
\end{prp}
\begin{proof}
	First, we show that $\RI$ is Tor-rigid as $\Rx$-module.
	Let $M$ be a finitely generated $\Rx$-module 
	and $p > 0$ an integer such that $\Tor_p^{\Rx}(M,\RI) = 0$. We set $q \colonequals -p-1$ for a more convenient notation.
	Let $\CFb$ be a resolution of $M$ by free modules of finite rank.
	For each integer $n > 0$ set $\CFb_n \colonequals \CFb \otimes \Rx/\ida^n$.
	We will deduce that $\Ho^{q}(\CFb_1) = 0$ from the existence of a
	surjective map
	$\Ho^{q}(\CFb) \longtwoheadrightarrow\Ho^{q}(\CFb_1)$.			 			
	\begin{enumerate}
		\item 
		For each integer $n > 0$ there is a short
		exact sequence of complexes of $\Rx$-modules
		$$
		\begin{td}
		0 \ar{r} \& \CFb \otimes \ida^n/\ida^{n+1} \ar{r} \& \CFb_{n+1} \ar{r}{\pi_{n}} \&  \CFb_n \ar{r} \& 0\,.
		\end{td}
		$$
		Since $\Rx$ is $\ida$-adically complete, there is an isomorphism  $\CFb \cong\varprojlim \CFb_n$ of complexes.
		\item 
		Because each $\RI$-module $\ida^n/\ida^{n+1}$ is free, it holds that $\CFb \otimes {\ida}^n/{\ida}^{n+1} \in \add \CFb_1$.
		Since $\Ho^{q+1}(\CFb_1) \cong \Tor_p^{\Rx}(M,\RI)= 0$ each
		$\RI$-linear morphism 
		$$
		\begin{td}
		\Ho^{q}(\pi_{n})\colon
		\Ho^{q}(\CFb_{n+1}) \ar{r} \& \Ho^{q}(\CFb_n) 
		\end{td}
		$$
		is surjective. 
		It follows that the inverse system of these morphisms satisfies the Mittag-Leffler condition.
		In particular, there is a short exact sequence of 
		$\Rx$-modules
		$$
		\begin{td}
		0 \ar{r} \& 
		\varprojlim^1 \Ho^{q+1}(\CFb_n) \ar{r} \&
		\Ho^{q}(\varprojlim \CFb_n)  \ar{r} 
		\& \varprojlim \Ho^{q} (\CFb_n) \ar{r} \& 0 
		\end{td}
		$$
		which is
		an algebraic analogue of Milnor's exact sequence \cite{Weibel}*{Variant below Theorem~3.5.8}.
		\item Since each transition morphism $\Ho^{q}(\pi_n)$ is surjective,
		the natural morphism $\varprojlim \Ho^{q}(\CFb_n) \longrightarrow
		\Ho^{q}(\CFb_1)$ of $\Rx$-modules is surjective as well.
	\end{enumerate}		
	In summary, 
	the $\Rx$-module $\RI$ is Tor-rigid because 
	there are surjective maps 
	$$
	\begin{td} 0 = \Ho^{q}(\CFb) \ar[twoheadrightarrow]{r} \& \varprojlim
	\Ho^{q}(\CFb_n) \ar[twoheadrightarrow]{r} \&
	\Ho^{q}(\CFb_1) \cong \Tor_{p+1}^{\Rx}(M,\RI)\,.
	\end{td}
	$$
	The proof above can be adapted to 
	the quotient ring $\Rx/\ida^m$ for any $m > 0$.
	More precisely, let $\CF$ be a free resolution of a module $M$  over $\Rx/\ida^m$ such that $\Tor^{\Rx/\ida^m}_p(M,\RI) =0$ with $p > 0$.
	For any  $n > 0$ 	we set $\CFb_n \colonequals \CF \otimes \Rx_n$ if $n < m$,  and $\CFb_n \colonequals \CF$ if $n \geq m$. Then the arguments above yield a surjective map $ 0=\Ho^{q}(\CFb_m) \longtwoheadrightarrow \Ho^{q}(\CFb_1)$, and thus the Tor-rigidity of $\RI$ as $\Rx/\ida^m$-module.
\end{proof}

The last statement yields a common source of the following instances of
Tor-rigidity.
\begin{cor}\label{cor/recover-tor-rigid}
Assume any of the following conditions.
	\begin{enumerate}[label={$\mathsf{(TR\arabic*)}$}, ref={$\mathsf{(TR\arabic*)}$}]
		\item \label{rec-tor-rig1} The ideal $\ida$
		is the maximal ideal $\mx$
		of the local ring $\Rx$.
		\item \label{rec-tor-rig2} The ideal $\ida$ is generated by an $\Rx$-regular sequence.
	\end{enumerate}
Then the quotient $\RI$ is Tor-rigid as $\Rx$-module and as $\Rx/\ida^n$-module for any integer $n > 0$.
\end{cor}
\begin{proof}
	Let $n > 0$.
	If $\ida = \mx$, then $\ida^n/\ida^{n+1}$ is a vector space over the field $\RI$.
	If $\ida$ is generated by an $\Rx$-regular sequence, then
	$\ida^n/\ida^{n+1}$ is also free as $\RI$-module (see \cite{Eisenbud-Book}*{Exercise 17.16}).
	So the claims follow from Proposition~\ref{prp/conormal-rigid}.
\end{proof}
\begin{rmk}
	Corollary~\ref{cor/recover-tor-rigid} and its proof can be extended to the case that the local ring $\Rx$ is not necessarily complete. This extension recovers the following statements.
	\begin{itemize}
		\item
		In case~\ref{rec-tor-rig1}, the Tor-rigidity properties of $\RI$  follow from the \emph{local criterion of flatness} as recalled in Theorem~\ref{thm/local-flatness}.
		\item
		In case~\ref{rec-tor-rig2}, Tor-rigidity of $\RI$ as $\Rx$-module
		amounts to \emph{rigidity of the Koszul complex} as mentioned in Remark~\ref{rmk/sub}.
	\end{itemize}
\end{rmk}

\subsubsection*{Derived lifting problem in a normally flat setup}
The next statement
extends known lifting results to a common denominator.
\begin{prp}\label{prp/2-rigid-lift}
	Assume that $\Rx$ is normally flat along the ideal $\ida$. Then
it follows that
		\begin{align}\label{eq/tor-ind-inf1}
			\tag{$\star\star$}
			\Tor_+^{\Rx}(\A, \Rx/\ida^n) = 0\text{ for any integer }n > 0.
		\end{align}
	Moreover, 
			any complex $\CP \in \D^-(\md {\rAI})$ satisfying $\Hom_{\D(\Abar)}(\CP,\CP[2]) = 0$ 
	has a lift $\CL \in \D^-(\md \A)$.
\end{prp}
\begin{proof}
	Since $\Tor_+^{\Rx}(\A, \RI) = 0$ by ~\eqref{eq/tor-ind0} and each $\RI$-module $\ida^n/\ida^{n+1}$ is flat,	Lemma~\ref{lem/algebra-conormals} 
	\eqref{tor-relations}
	implies~\eqref{eq/tor-ind-inf1}.
	
	Because of Lemma~\ref{lem/algebra-conormals}~\eqref{conormal-free}
	there are isomorphisms of ${\rAI}$-bimodules
	$$
	I^n/I^{n+1} \cong
	\A \otimes \ida^n/\ida^{n+1} \cong \A \otimes {\RI^{r_n}} \cong (\A \otimes \RI)^{r_n} \cong {\rAI}^{r_n}, \quad \text{where }r_n \geq 0.
	$$
	Since each  ${\rAI}$-bimodule $I^n/I^{n+1}$ is free,
	the second claim follows from Proposition~\ref{prp/twisted-lift}.
\end{proof}

	The conclusion of Proposition~\ref{prp/2-rigid-lift} was previously shown
	in setups related to conditions 
	of Remark~\ref{rmk/sub}, more precisely,
		if  $\A$ is free as $\Rx$-module and $\ida = \mx$ \cite{Rickard91b}, or
		if $\Rx = \A$ and 
		$\ida$ is generated by an $\Rx$-regular sequence
		\cite{Yoshino}.
In both cases, it holds that $\Tor_1^{\Rx}(\A, \RI) = 0$ and $\Rx$ is normally flat along the ideal $\ida$.
	There is also a similar lifting result for dg-modules over certain commutative dg-algebras due to Nasseh and Sather-Wagstaff \cite{Nasseh/Sather-Wagstaff}.

\subsubsection*{Abelian lifting problem in a normally flat setup}

Next, we 
make a short detour from the main route of this paper and 
consider an application of the previous Tor-rigidity results.

The \emph{abelian lifting problem} asks whether a given module over the quotient $\AI = \A/I$ lifts to a $\A$-module in the following sense.
\begin{dfn}\label{dfn/lift-mod}
	A ${\rAI}$-module $N$ \emph{lifts to a $\A$-module} $M$
	if there is an isomorphism $M \otimes_{\A} {\rAI} \cong N$ of ${\rAI}$-modules and $\Tor_{+}^{\A}(M,{\rAI}) = 0$.
\end{dfn}
The torsion vanishing condition admits a reformulation in terms of the base ring.
\begin{lem}\label{lem/tor-independence}
	For any integer $i \in \Z$ 
	and any $\A$-module $M$
	there is an isomorphism 
	$\Tor_{i}^{\A}(M,{\rAI}) \cong \Tor_{i}^{\Rx}(M,\RI)$ of $\RI$-modules.
\end{lem}
\begin{proof}
	Let $i \in \Z$ and  $\CFb$ be a free resolution of a $\A$-module $M$.
	Since ${\rAI} \cong \A \otimes_{\Rx} \RI$ and $\Tor_+^{\Rx}(\A,\RI) = 0$, 
	there are isomorphisms
	$$
	\qquad \qquad 
	\Tor_{i}^{\rA}(M,{\rAI}) 
	\cong \Ho^{-i}(\CFb \otimes_{\rA} {\rAI} ) 
	\cong
	\Ho^{-i}(\CFb \otimes_{\Rx} \RI) 
	\cong
	\Tor_i^{\Rx}(M, \RI) \qquad \qquad \qedhere
	$$
\end{proof}
The abelian lifting problem may also be considered in the following terms.
\begin{rmk}A  ${\rAI}$-module $N$ lifts to a  $\A$-module if and only if
	its projective resolution lifts to a projective resolution of a $\A$-module.
\end{rmk}
Combining an argument by Rickard \cite{Rickard91b}*{Corollary 3.2} with the last Tor-rigidity result yields the following.
\begin{prp}
	\label{prp/abelian-lift}
	Assume that $\Rx$ is normally flat along the ideal $\ida$.
	Then any finitely generated ${\rAI}$-module $N$ 
	with $\Ext^2_{\Abar}(N,N) = 0$
	lifts to a finitely generated $\A$-module in the sense of Definition~\ref{dfn/lift-mod}.
\end{prp}
\begin{proof}
	Let $\CPb$ be a minimal projective resolution of
	$N$. Since $\Ext_{\Abar}^2(N,N) =0$, 
	there is a 
	complex
	$\CLb \in \Kom^-(\proj \A)$
	such that $\CLb \otimes_{\A} {\rAI} \cong \CP$ in $\Kom^-(\proj {\rAI})$
	by
	Corollary~\ref{prp/2-rigid-lift} 
	and Lemma~\ref{lem/minimal-lift}.
	We claim that $N$ lifts to $M \colonequals \Ho^0(\CL)$. 
	
	Since $L^n= 0$ for any $n > 0$, there is an isomorphism $M \otimes_{\A} {\rAI} \cong N$ in $\md {\rAI}$.

	Next, we 
	view $\CLb$ as a complex of $\Rx$-modules.
	Since $\CL$ is a right-bounded complex such that 
	$\Tor_{+}^{\Rx}(\CL,\RI) = 0$ and
	$\Ho^{\pm}(\CL \otimes_{\Rx} \RI) = \Ho^{\pm}(\CP) = 0$, it follows that 
	$\Tor_+^{\Rx}(M,\RI) = 0$
	by Corollary~\ref{cor/cohom1c} using Proposition~\ref{prp/conormal-rigid}.
	This translates into $\Tor_{+}^{\A}(M,{\rAI}) =0$ by Lemma~\ref{lem/tor-independence}. So $M$ satisfies both conditions of a lift of $N$.
\end{proof}
\noindent
	Proposition~\ref{prp/abelian-lift} was known in case
		$\Rx$ is regular of Krull dimension one, $\A = \Rx G$  for a finite group $G$ and $\ida =\mx$ \cite{Green}, in the more general case that $\A_{\Rx}$ is free and $\ida = \mx$ \cite{Rickard91b}, as well as	the case that $\ida$ is generated by an $\Rx$- and $\A$-regular sequence \cite{Auslander-Ding-Solberg}.

\subsection{Lifting presilting complexes}

The goal of this subsection is to show that any presilting complex lifts under  condition 
\eqref{eq/tor-ind-inf1}
which is weaker than the main additional assumption of the previous subsection.

In this subsection we say that a ${\rAI}$-bimodule $M$ admits a \emph{regular bimodule resolution} if there is a right-bounded complex of regular ${\rAI}$-bimodules with finite ranks
$$  	
\begin{td}
\CB \colonequals (\quad \ldots \quad  {{\rAI}}^{r_{i+1}} \ar{r} \& {{\rAI}}^{r_{i}} \ar{r}   \&\quad \ldots \quad {{\rAI}}^{r_{2}} \ar{r} \& {{\rAI}}^{r_{1}} \ar{r} \& {{\rAI}}^{r_0})
\end{td}
$$
such that $\Ho^i(\CB) = 0$ for any integer $i < 0$ and there is an ${\rAI}$-bilinear isomorphism $\Ho^0(\CB) \cong M$. Such a resolution is \emph{not} a projective resolution in general.
\begin{rmk}
	Since ${\rAI}$ is a Noetherian $\RI$-algebra,
	the ${\rAI}$-bimodule ${\rAI}$ is projective
	if and only if the $\RI$-algebra
	${\rAI}$ is separable
	\cite{Cartan-Eilenberg}*{Chapter~IX, Theorem~7.10}.
\end{rmk}
To add context, we note that
the existence of a bimodule resolution 
imposes further restrictions.
\begin{lem}
	Let $Z({\rAI})$ denote the center of ${\rAI}$.
	If a ${\rAI}$-bimodule $M$ has a 
 regular bimodule resolution $\CB$, the following statements hold.
	\begin{enumerate}
		\item The ${\rAI}$-bimodule $M$ admits generators $m_1, \ldots, m_n$ which commute with any element of ${\rAI}$.
		Moreover, any element of $M$ commutes with any element of $Z({\rAI})$.
		\item There is a complex $\CF$ of free $Z({\rAI})$-modules such that
		there is an isomorphism $\CF \otimes_{Z({\Abar})} {\rAI} \cong \CB$ of complexes of ${\rAI}$-bimodules
		as well as an isomorphism $\Ho^0(\CF)\otimes_{Z({\Abar})} {\rAI} \cong M$ of ${\rAI}$-bimodules.
		\qed
	\end{enumerate} 
\end{lem}
The last observation suggests that bimodule resolutions might be obtained from resolutions of free $\RI$-modules. 
This is the case in the following context.
\begin{lem}\label{lem/regular-res}
	Let $N$ be a finitely generated left $\RI$-module such that $\Tor_{+}^{\Rx}(\A,N) = 0$. 
	Then $\A \otimes N$ has a regular ${\rAI}$-bimodule resolution.
\end{lem}
\begin{proof}
	Since $\A$ is an $\Rx$-algebra and $\RI$ is commutative, 
	for any left $\RI$-module $X$
	there is an ${\rAI}$-bimodule structure on 
	$\A \otimes X$ 
	given by $(a \otimes s) \cdot (b \otimes x) = ab \otimes sx$
	and $(b \otimes x)(a \otimes s) = ba \otimes sx$ 
	for any $a,b \in \A$, $x\in X$ and $s \in \RI$.
	In the case $X = \RI$ this definition recovers the regular ${\rAI}$-bimodule structure on $\A \otimes \RI$.
	Moreover, for any morphism $f\colon X \longrightarrow  Y$ of left $\RI$-modules the map 
	$ \id \otimes f\colon \A \otimes X \longrightarrow \A \otimes Y$ is ${\rAI}$-bilinear with respect to the above ${\rAI}$-bimodule structures.
	Let $\CFb$ be a resolution of the left $\RI$-module $N$ via free $\RI$-modules of finite rank.
	Since~\eqref{eq/tor-ind0} holds, that is, $\Tor_{+}^{\Rx}(\A,\RI ) = 0$, it follows that $\Tor_+^{\Rx}(\A, \CF) = 0$.
	This implies that 
	$\Ho^{-p}(\A \otimes \CF) \cong \Tor_{p}^{\RI}(\A,N) =  0$ for any integer $p > 0$. 
	By
	the previous considerations, the complex $\A \otimes \CFb$ is a regular bimodule resolution of the ${\rAI}$-bimodule $\A \otimes N$. 
\end{proof}

\begin{lem} \label{lem/hom-bound}
	Let $\CP \in \per {\rAI}$ and $\CQ \in \D^{-}(\md {\rAI})$.
	Then the integers
	\begin{align} \label{eq/min-max} 
	\ell(\CP) \colonequals \min \{ i \in \Z \ | \ \CP \geq  {\rAI}[i] \}
	\quad 
	\text{and} \quad
	r(\CQ)  \colonequals \max \{ i \in \Z \ | \ {\rAI} \geq \CQ[i] \}
	\end{align}
	are well-defined, and it holds that $\CP\geq \CQ[r(\CQ) - \ell(\CP)]$.
\end{lem}
\begin{proof}
	We may choose minimal complexes
	$\CL \in \Hotb(\proj {\rAI})$ and $\CM \in \Hot^-(\proj {\rAI})$  quasi-isomorphic to the complexes $\CP$ and $\CQ$, respectively.
	Then $$\ell(\CP) = \min \{ i \in \Z \mid L^i \neq 0 \}\quad \text{
		and  }\quad r(\CQ) = \max \{i \in \Z \mid
	M^i \neq 0 \}.$$
	So
	$\Hom_{\D(\Abar)}(\CP,\CQ[i]) \cong\Hom_{\Hot({\Abar})}(\CL,\CM[i]) = 0$ for any  $i > r(\CQ) - \ell(\CP)$.
\end{proof}
 
\begin{prp}\label{prp/twist-order}
	Let $\CP \in \per {\rAI}$ and $\CQ \in \D^{-}(\md {\rAI})$ such that $\CP \geq \CQ$.
	Let $M$ be a ${\rAI}$-bimodule which has a regular bimodule resolution.
	Then 
	$$
	\CP \geq \CQ \underset{{\Abar}}{\overset{\mathbb{L}}{\otimes}} M.
	$$
\end{prp}
This statement and its proof were inspired by a recent result of Nasseh, Ono and Yoshino \cite{Nasseh-Ono-Yoshino}*{Theorem~3.10}.
Roughly speaking, the latter result has the same conclusion allowing the complex $\CP$ to be an object of $\D^-(\md \AI)$ but imposing certain finiteness conditions on the $\AI$-bimodule $M$.

\begin{proof}
	Let $\CP$ and $\CQ$ be as above,
	and $\CB$ be a regular bimodule resolution of the ${\rAI}$-bimodule $M$.
	For any integer $i \geq 0$ 
	we set
	$$
	\CY_i \colonequals \CQ  
	\underset{{\Abar}}{\overset{\mathbb{L}}{\otimes}} 
	\CB_i,\quad \text{ with }
	\begin{td}
	\CB_i \colonequals (\ldots {{\rAI}}^{r_{i+2}} \ar{r} \& {{\rAI}}^{r_{i+1}} \ar{r} \& {{\rAI}}^{r_i})\,,
	\end{td}
	$$
	that is, $\CB_i$ 
	denotes the 
	brutal truncation
	of $\CB$ at degrees $\leq -i$.
	Since $\CB \cong M$ in the derived category $\D(\Bimod {\rAI})$ of ${\rAI}$-bimodules, there is an isomorphism $$\CY_0 \cong  \CQ 
	\underset{{\Abar}}{\overset{\mathbb{L}}{\otimes}}
	M\text{\quad in $\D({\rAI})$}.
	$$
	Therefore, 
	it is sufficient to show that 
	$\CP \geq \CY_0$.
	\begin{enumerate}
		\item First, we claim that there is an integer $m \geq 0$
		such that $\CP \geq \CY_m$.
		
		In the notations of~\eqref{eq/min-max},
		we set $m \colonequals \max\{r(\CQ)-\ell(\CP),0 \}$.
		A version of the 
		Künneth trick for complexes \cite{Yekutieli99}*{Lemma 2.1}
		implies the first inequality in
		$$ r(\CY_m ) \leq r(\CQ) + 
		r(\CB_m) \leq r(\CQ) -m \leq \ell(\CP).$$
		By	Lemma 
		\ref{lem/hom-bound} it follows that
		$\CP \geq \CY_m [r(\CY_m)  - \ell(\CP)]$, and thus $\CP \geq \CY_m$. 
		\item Let $i \geq 0$ and $n > 0$. We claim that there is an $\RI$-linear isomorphism $$\Hom_{\D(\Abar)}(\CP,\CY_i[n]) \cong \Hom_{\D(\Abar)}(\CP,\CY_{i+1}[n]).$$
		The short exact sequence of complexes of  ${\rAI}$-bimodules 
		$$
		\begin{td}
		0 \ar{r} \& 
		{\rAI}^{r_i} [i] \ar{r} \& \CB_i \ar{r} \& \CB_{i+1} \ar{r} \& 0\,,
		\end{td}
		$$
		 induces a distinguished triangle
		in the category $\D(\mathrm{Bimod} {\rAI})$.
		Applying the functors $\CQ  
		\underset{{\Abar}}{\overset{\mathbb{L}}{\otimes}} 
		\dash [n]$, and then $\Hom_{\D(\Abar)}(\CP,\dash)$, to the latter
		yields the distinguished triangle
		$$
		\begin{td}
		(\CQ)^{r_i} [i+n] \ar{r} \& Y_i[n] \ar{r} \& 
		Y_{i+1} [n]
		\ar{r} \& (\CQ)^{r_i} [i+n+1]
		\end{td}
		$$
		in  the category $\D({\rAI})$, and	subsequently	an exact sequence of $\RI$-modules
		$$
		\begin{td}
		0 \ar{r} \& 
		\Hom_{\D(\Abar)}(\CP, \CY_i[n])
		\ar{r}{\sim} \& \Hom_{\D(\Abar)}(\CP, \CY_{i+1} [n]) 
		\ar{r} 
		\& 0\,,
		\end{td}
		$$
		which has zero outer terms because of $\CP \geq \CQ$ and $i+n > 0$. 
	\end{enumerate}
	These two arguments show that
	$\Hom_{\D(\Abar)}(\CP, Y_0 [n])  \cong \Hom_{\D(\Abar)}(\CP, Y_m [n]) = 0$ for any integer $n > 0$. Thus, $\CP \geq Y_0$.
\end{proof}

The main consequence of this subsection 
is the following lifting result.
\begin{cor}\label{cor/silting-lifts}
	Assume that the $\Rx$-algebra $\A$ and the ideal $\ida$ of $\Rx$ satisfy
	\begin{align}\label{eq/tor-ind-inf} \tag{$\star\star$}
	\Tor_+^{\Rx}(\A, \Rx/\ida^n) = 0\text{ for any integer }n > 0.
	\end{align}
	Then any complex $\CP \in \per {\rAI}$ such that $\CP \geq \CP[1]$
	has a lift $\CL \in \per \A$.
\end{cor}
\begin{proof}
	Under the assumptions above,
	Lemma~\ref{lem/algebra-conormals}
	implies that
	$\Tor_+^{\Rx}(\A,\ida^n/\ida^{n+1}) = 0$
	and that there is an isomorphism
	$\A \otimes \ida^n/\ida^{n+1} \cong
	\ol{I^n}$ of ${\rAI}$-bimodules for any $n > 0$.
	By Lemma~\ref{lem/regular-res} 
	each ${\rAI}$-bimodule
	$\ol{I^n}$ 
	has a regular bimodule resolution.
	
	Since $\CP \in \per \A$ and $\CP \geq \CP[1]$, Proposition~\ref{prp/twist-order} implies that
	$\CP \geq \alpha_n(\CP)[1]$ for any  $n > 0$, where $\alpha_n(\CP)$ denotes the left-derived tensor product of $\CP$ with $\ol{I^n}$.
	According to Proposition~\ref{prp/twisted-lift}
	the complex $\CP$ has a lift $\CL \in \D^-(\md \A)$.
	Lemma~\ref{lem/minimal-lift} ensures that $\CL \in \per \A$.
\end{proof}
The last lifting criterion will be applied to presilting complexes.

\begin{rmk}\label{rmk:hierarchy}
Condition~\eqref{eq/tor-ind-inf} is satisfied in both setups of Remark~\ref{rmk/sub}.
It yields an intermediate setup between that of the previous subsection and the  
Tor-independence assumption~\eqref{eq/tor-ind0}.
More precisely, 
the following implications hold according to Propositions 
\ref{prp/conormal-rigid} and~\ref{prp/2-rigid-lift}.
\begin{align*}
	\Tor_1^{\Rx}(\A, \RI) = 0
\text{ and $\Rx$ is normally flat along $\ida$} \quad
\Rightarrow \quad
\eqref{eq/tor-ind-inf} 
\quad
\Rightarrow \quad
	\eqref{eq/tor-ind0}\,.
\end{align*}
\end{rmk}

\section{Silting complexes over Noetherian algebras as weak generators}
\label{sec/characterization}

Rickard proved that a \emph{pretilting} complex $\CT$ over the Noetherian algebra $\A$ is tilting if and only if
there is a non-zero morphism from $\CT$ to any non-zero object of the category $\D^-(\md \A)$ up to shift \cite{Rickard91b}*{Lemma 1.1}.

The goal of this section is to extend Rickard's characterization to any presilting complex of the algebra $\A$.
This extension involves arguments on differential graded categories by Keller. 
\begin{setup}\label{setup/presilt-noetherian}
	Throughout this section, we fix the following notions.
	\begin{itemize}
		\item
		As before, $\A$ denotes a Noetherian $\Rx$-algebra. 
		\item Let $\CL$ be a presilting complex of $\A$. For simplicity of the presentation we assume that $\CL$ is given by a bounded complex of finitely generated projective $\A$-modules.
		\item 
		We denote by $\Eta$ the differential graded $k$-algebra
		$\Hom_{\A}^{\bt}(\CL,\CL)$, and by 
		$\D(\Eta)$ the derived category of its dg-modules.
	\end{itemize}
\end{setup}
The central tool of this section is a variant of
Keller's derived Morita theorem which yields an embedding of categories
\begin{align*}
\Fm\colon \D(\Eta)  &\begin{td}\mathstrut \ar[hookrightarrow]{r} \&  \D(\A) \,.\end{td}
\intertext{The main technical issue is to show that $\Fm$ restricts to a functor}
\Fm'\colon\Dfm(\Eta)&
\begin{td}
\mathstrut
\ar[dashed, hookrightarrow]{r} \& \D^-(\md \A) \end{td} 
\end{align*}
starting from a certain category which can be identified with $\D^-(\md \Eta)$ in the case that $\CL$ is pretilting.
Following an outline by Keller which is based on Rickard's approach, we will deduce that 
$\CL$ is silting if and only if $\Fm'$ is an equivalence if and only if $\CL$ is a weak generator of $\D^-(\md \A)$ in the sense above. 

For general facts on dg-categories we refer to Keller's works \cites{Keller, Keller98}. We will apply some dg-categorical arguments specific to our setup following Yekutieli's book \cite{Yekutieli}.
Other than in this section, we will not use the terminology of dg-categories in this paper.

\subsection{Noetherian algebras and pseudo-finite semi-free modules}

For the dg-algebra $\Eta$ we denote
\begin{itemize}
	\item by $\mathcal{A}$ the category $\md \Rx$ of finitely generated $\Rx$-modules,
	\item
	by $\Df(\Eta)$ the full subcategory of dg-modules $\CM$ in $\D(\Eta)$ such that  $\Ho^i(\CM) \in \mathcal{A}$ for any integer $i \in \Z$,
	\item
	by $\Dfm(\Eta)$  the full subcategory of dg-modules $\CM$ in $\Df(\Eta)$ such that 
	there is an integer $m \in \Z$ with $\Ho^i(\CM) = 0$ for any integer $i > m$.
\end{itemize}
Viewing the $\Rx$-algebra $\A$ as a dg-algebra in degree zero, the analogue $\Dfm(\A)$ of the category $\Dfm(\Eta)$ 
can be identified with the category
$\D^-_{\md \A} (\A)$.
The latter 
is known to be equivalent to 
the category $\D^-(\md \A)$ (see, for example, \cite{Yekutieli}*{Corollary 11.3.21}).

Since $\A$ is a Noetherian $\Rx$-algebra, the stalk complex $\A$ is contained in $\Dfm (\A)$. There is
a counterpart for the dg-algebra $\Eta$.
\begin{lem}\label{lem/pseudo-noether}
	The dg-module $\Eta$
	is an object of the category $\Dfm(\Eta)$.
\end{lem}
\begin{proof}
	For any element $r \in \Rx$ let $\varrho_r\colon \CL \longrightarrow \CL$ denote the homogeneous morphism of degree zero given by right multiplication with $r$ at each term of $\CL$.
	Let $i \in \Z$.
	The composition of the two maps
	which 
	define the $\Rx$-algebra structure on the ring $\Ho^0(\Eta)$ and 
	the right $\Ho^0(\Eta)$-module structure on the morphism space $\Ho^i(\Eta)$
	$$
	\Rx \longrightarrow \Ho^0(\Eta) \longrightarrow \End_{\Z}(\Ho^i(\Eta)), \qquad
	r \mapsto \varrho_r \mapsto 
	(\phi \mapsto \phi \cdot \varrho_r= \varrho_r[i] \cdot \phi = r \cdot \phi)
	$$
	recovers the natural right $\Rx$-module structure on the space $\Ho^i(\Eta)$. 
	By Remark~\ref{rmk/right-bounded}
	the 
	$\Rx$-module $\Ho^i(\Eta)$ is finitely generated.
	This implies the claim.
\end{proof}

\subsection{Pseudo-finite semi-free modules}

The complex $\CL$ has a differential graded $(\Eta,\A)$-bimodule structure and gives rise to a functor
$$
\begin{td} \Fm \colonequals \dash  \underset{\Eta}{\overset{\mathbb{L}}{\otimes}} \CL\colon \D(\Eta) \ar{r} \&  \D(\A)\,. \end{td}
$$
Since $\CL$ is presilting, the dg-algebra $\Eta$ is quasi-isomorphic to the  truncated dg-algebra
$$
\trEta \colonequals (\begin{td}\ldots \mathrm{E}^{-2 }\ar{r} \&  \mathrm{E}^{-1} \ar{r} \& \ker d^0_{\Eta} \end{td})\,.
$$
Next, we will show 
that
the functor 
\begin{align*}
\begin{td} \trFm \colonequals \dash  \underset{\trEta}{\overset{\mathbb{L}}{\otimes}} \CL\colon   \D(\trEta) \ar{r} \& \D(\A)
\end{td}
\end{align*}
restricts to a well-defined functor
$\trFm'\colon \Dfm(\trEta) \longrightarrow \Dfm(\A)$
and deduce a similar statement for the original dg-algebra $\Eta$.
To do so, we need  a dg-analogue of the K\"{u}nneth trick.
\begin{prp}\label{prp/dg-kuenneth}
	For any dg-module $\trM \in \D^-(\trEta)$ it holds that $\trFm(\trM) \in \D^-(\A)$.
	
	More precisely, 
	let
	$m \in \Z$ denote the maximal integer such that $\Ho^i(\trM) =  0$ for any integer $i > m$,  and $r \in \Z$ the maximal integer 
	that $L^j= 0$ for any integer $j > r$.
	Then $\Ho^k( \trFm(\trM)) = 0$ for any integer $k > m + r$. 
\end{prp}
In the following, we denote by $\Cstr$ the category of dg-modules  and homogeneous morphisms of degree zero for the dg-algebra $\trEta$.
\begin{proof}[Comment on the proof]
	Since the dg-algebra $\trEta$ is non-positive, \cite{Yekutieli}*{Corollary 11.4.27} implies that there
	is a semi-free dg-module $\CP$ quasi-isomorphic to  $\trM$ in $\Cstr$ with $P^i = 0$ for any  $i > m$.
	Let $k > m + r$. 
	Because semi-free dg-modules are $K$-flat, it follows that 
		 $(\trFm(\trM))^k \cong (\CP {\otimes_{\trEta}} \CL)^k = \bigoplus_{i+j = k} P^i \otimes_{\Rx}  L^j =0$.
\end{proof}

Any dg-module $\trM \in \Dfm(\trEta)$ admits a  \emph{pseudo-finite semi-free resolution} in the following sense.
\begin{thm}\label{thm/dg-res}
	Let $\trM \in \Dfm(\trEta)$ such that $\Ho^+(\trM) = 0$. Then there is 
	a dg-module $\CP$ 
	which is the direct limit of a direct system of dg-modules and embeddings
	$$ 
	\begin{td}
	\CP_0 \ar[hookrightarrow]{r} \& \CP_{1} \ar[hookrightarrow]{r} \& \ldots \CP_{j-1} \ar[hookrightarrow]{r} \& \CP_j\ar[hookrightarrow]{r} \& \ldots 
	\end{td} $$ 	
	and a quasi-isomorphism $\varrho\colon \CP \longrightarrow  \trM$ which 
	is the direct limit of a family of morphisms  $(\varrho_j\colon \CP_j \longrightarrow \trM)_{j \in \N_0}$ in the category $\Cstr$
	such that the dg-modules and morphisms
	$(\CP_j,\varrho_j)_{j \in \N_0}$
	satisfy the following properties.
	\begin{enumerate}[label={$\mathsf{(P\arabic*)}$}, ref={$\mathsf{(P\arabic*)}$}]
		\item \label{truncate-cohomology} For any integer $j \in \N_0$ the map $\Ho^i(\varrho_j)\colon  \Ho^i(\CP_j) \longrightarrow \Ho^i(\trM)$ is surjective 
		for any integer $i \leq - j$
		and bijective
		for any integer $i > - j$. 
		\item \label{pseudo-finite} The dg-module $\CP_0$ satisfies $P_0^n = 0$ for any integer $n > 0$ and is \emph{pseudo-finite free}, that is, 
		there 
		is an isomorphism
		$\CP_0 \cong \bigoplus_{i \leq 0}(\trEta)^{r_i} [i]$ in the category $\Cstr$, where $r_i \in \N_0$ for any integer $i \leq 0$. 
		\item \label{free-quotient} For any integer $j > 0$ there is an isomorphism $\CP_{j} / \CP_{j-1} \cong (\trEta)^{s_j}[j]$ in the category $\Cstr$ with $s_j\in \N_0$.
	\end{enumerate}
\end{thm}
\begin{proof}[Comment on the proof]
	By Lemma~\ref{lem/pseudo-noether}, the dg-algebra $\trEta$ is \emph{cohomologically right pseudo-Noetherian}, that is, $\Ho^+(\Eta) = 0$, the ring $\Ho^0(\Eta)$ is right Noetherian, and $\Ho^i(\Eta)$ is finitely generated as $\Ho^0(\Eta)$-module for any integer $i \in \Z$.
	Since $\trEta$ is also non-positive,
	the statements follow from
	\cite{Yekutieli}*{Proof of Theorem~11.4.40}.
\end{proof}

The dg-modules $(\CP_j)_{j \in \N_0}$ might be viewed as 
dg-versions of certain truncations.
\begin{rmk}
	If
	$\CL$ is pretilting, it holds that $\Eta = \End_{\D(\A)}(\CL)$ and we may choose	 
	a complex
	$\CP \in \Kom(\proj \Eta)$ 
	with $P^n = 0$ for any integer $n > 0 $,
	a 
	quasi-isomorphism $\varrho\colon \CP\longrightarrow \CN$ of complexes, 
	$\CP_j$ as the 
	\emph{brutal truncation} 
	$\sigma^{\leq -j}(\CP)$, 
	and $\varrho_j$ as the composition
	$\CP_j \hookrightarrow \CP \overset{\varrho}{\longrightarrow} \CN$
	for any integer $j \geq 0$ in order 
	to ensure  properties~\ref{truncate-cohomology},~\ref{pseudo-finite} and~\ref{free-quotient} above.
\end{rmk}

Next, we apply the functor $\trFm$ to the dg-submodules defining the resolution $\CP$.
\begin{lem}\label{lem/Df}
	It holds that
	$\trFm(\CP_j) \in \Df(\A)$ for any integer $j \in \N_0$.
\end{lem}
\begin{proof}
	We prove the claim by induction on the integer $j \geq 0$.
	\begin{itemize}
		\item $j = 0$: 
		It holds that
		$\trFm(\trEta) \cong \CL \in \Dbf(\A)$. 
		Property
		\ref{pseudo-finite} 
		of the dg-module $\CP_0$ 
		implies that
		$\trFm(\CP_0) \cong \bigoplus_{i \leq 0} (\CL)^{r_i} [-i] \in \Df(\A)$, 
		where $r_i \in \N_0$.
		\item $j-1 \to j$: Assume that the claim is true for some integer $j-1$ with $j > 0$. 
		By the description 
		of the quotient $\CP_j/\CP_{j-1}$ in 
		\ref{free-quotient} 
		there is an integer $s_j \geq 0$ and a distinguished triangle
		$$
		\begin{td}
		\trFm(\CP_{j-1}) \ar{r} \& \trFm(\CP_j) \ar{r} \& (\trFm(\trEta))^{s_j}[j] \ar{r} \& \trFm(\CP_{j-1})[1]
		\end{td}
		$$ in the category $\D(\A)$.
		Since $\Df(\A)$ is a triangulated subcategory of $\D(\A)$,
		the induction assumption implies that $\trFm(\CP_j) \in \Df(\A)$. \qedhere
	\end{itemize}
\end{proof}

We may conclude now
that the functor $\Fm$ restricts to the following subcategories.
\begin{prp}\label{prp/tensor}
	For any dg-module $\CM \in  \Dfm(\Eta)$ it holds that $\Fm(\CM) \in \Dfm(\A)$.
\end{prp}

The next arguments parallel the 
proof of \cite{Yekutieli}*{Theorem~12.5.7}.
The main idea is to truncate the dg-module $\CM$ at the `right' degree.
\begin{proof}
	Since $\Fm$ commutes with shifts, we may assume that
	$\Ho^+(\CM) = 0$ to simplify notation.
	Let $i \in \Z$
	and $\trM$ denote the restriction of $\CM$ to the truncated dg-algebra $\trEta$.
	Since the inclusion $\trEta \hooklongrightarrow \Eta$ is a quasi-isomorphism of dg-algebras,
	there is an isomorphism  $\Ho^i(\Fm(\CM)) \cong \Ho^i(\trFm(\trM))$ of $\Rx$-modules
	by \cite{Yekutieli}*{Theorem~12.7.2}.
	\begin{itemize}
		\item
		If $i > r \colonequals\max\{ i \in \Z \ | \  L^i\neq 0 \}$, then
		$\Ho^i(\Fm(\CM)) = 0$ by Proposition~\ref{prp/dg-kuenneth}.
		\item Assume that $i \leq r$. Let $(\CP_j,\varrho_j)_{j \geq 0}$ be a direct system of dg-modules and morphisms for $\trM$ as in Theorem~\ref{thm/dg-res}.
		Set $j \colonequals r+1-i$. 
		Since the dg-module $\CP_j$ has
		property~\ref{truncate-cohomology} there is a distinguished triangle
		$$
		\begin{td}
			\CC_j[-1] \ar{r} \&
		\CP_j \ar{r}{\varrho_j} \& \trM \ar{r} \& \CC_j   
		\end{td}
		$$
		in $\D(\trEta)$ with $\Ho^i(\CC_j) = 0$ for any integer $i > -j+1$. 
		Applying $\trFm$ to the latter
		and taking cohomology at degree $i$ yields an exact sequence
		of $\Rx$-modules
		$$
		\begin{td}
		\Ho^{i-1}( \trFm(\CC_j)) \ar{r} \& 
		\Ho^{i}( \trFm(\CP_j)) \ar{r} \&
		\Ho^{i}( \trFm(\trM))\ar{r} \&
		\Ho^{i}( \trFm(\CC_j))
		\end{td}
		$$
		Proposition~\ref{prp/dg-kuenneth} implies that $\Ho^\ell(\trFm(\CC_j)) = 0$ for any integer $\ell > r-j+1 = i -2 $. Together with Lemma~\ref{lem/Df} it follows that
		$\Ho^i(\trFm(\CP_j)) \cong \Ho^{i}( \Fm(\CM)) \in \mathcal{A}$.
	\end{itemize}
	This completes the proof that $\Fm(\CM) \in \Dfm(\A)$. 
\end{proof}
\begin{rmk}
	If the complex $\CL$ is pretilting, 
	Proposition~\ref{prp/tensor} recovers
	the fact 
	that the left-derived functor
	$
	\begin{td}
	\dash \underset{\Eta}{\overset{\mathbb{L}}{\otimes}} \CL\colon
	\D^-(\md \Eta) \ar{r} \& \D^-(\md \A)
	\end{td}
	$
	is well-defined.
\end{rmk}

\subsection{Keller's derived Morita Theorem} 

Next, we consider several triangulated subcategories associated to the perfect complex $\CL$.
\begin{itemize}
	\item 	We denote by
	${\langle\CL\rangle}^{\perp}$
	the strictly full subcategory of objects $\CM$ in $\D(\A)$ such that
	$\CL \perp \CM$, or, equivalently, 
	$\Hom_{\D(\A)}(\CX, \CM) = 0$ for any object  $\CX \in \langle \CL \rangle$.
	\item Let $\Tau$ be 
	a triangulated subcategory of $\D(\A)$.
	Using terminology from the Stacks project 
	\cite{stacks-project}*{\href{https://stacks.math.columbia.edu/tag/09SJ}{Tag 09SJ}},
	we call $\CL$ a \emph{weak generator of the category $\Tau$} if
	$\CL \in \Tau$ and
	$\langle \CL \rangle^{\perp} \cap \Tau = 0$, that is, 
	$\CL \perp \CM$ for any object $\CM \in \Tau$. 
	Equivalently, for any  non-zero object $\CN \in \Tau$ 
	there is an integer $i \in \Z$ and a non-zero morphism $\CL \longrightarrow \CN[i]$ in $\Tau$.
	\item 	We denote by 
	$\Loc(\CL)$ the smallest localising subcategory containing $\CL$.
	Let us recall that a subcategory $\mathcal{L}$ in $\D(\A)$ is \emph{localising} if it is strictly full, triangulated and closed under arbitrary coproducts.
	Such a subcategory is automatically closed under direct summands \cite{Angeleri-Huegel}*{Lemma 3.6}. In particular, $\langle \CL \rangle$ is a proper subcategory of the category $\Loc(\CL)$.
\end{itemize}

By 
Keller's derived Morita Theorem
the subcategories $\Loc(\CL)$ and $\langle \CL \rangle^{\perp}$ 
are related in terms of
the dg-endomorphism algebra $\Eta$ of $\CL$. We state a variant of this theorem which was formulated by Yekutieli.

\begin{thm}\label{thm/derived-Morita}
	The functor $\Fm$ is fully faithful and has a right adjoint functor $\Gm$ given by
	\begin{align}
	\label{eq/dg-pair}
	\begin{td} 
	\D(\Eta)  
	\\[0.25cm]
	\ar[xshift=-0.15cm,hookleftarrow]{u}[font=\normalsize]{
		\Fm 
		= \dash \underset{\Eta}{\overset{\mathbb{L}}{\otimes}} \CL \ 
	} 
	\ar[xshift=0.15cm]{u}[swap, font=\normalsize]{\, \Gm
		= \mathbb{R}\Hom_\Lambda(\CL, \dash)
	}
	\D(\A) \,.
	\end{td}
	\end{align}
	Moreover, 
	the essential image of $\Fm$
coincides with the subcategory $\Loc(\CL)$,
	and the kernel of $\Gm$ by the subcategory $\langle \CL \rangle^{\perp}$, that is, for any object $\CM \in \D(\A)$
	there is an isomorphism
	$  \Gm(\CM) \cong 0$ in $\D(\Eta)$ if and only if  $\CL \perp \CM$.
\end{thm}
\begin{proof}[Comment on the proof]
	Since $\CL$ is a compact object of $\D(\A)$, it is a weak generator of the subcategory
	$\Loc(\CL)$. The claims follow from  
	\cite{Yekutieli}*{Theorem~14.2.29}.
\end{proof}
The right adjoint $\Gm$ restricts to the previously considered subcategories as well.
\begin{lem}\label{lem/Hom-finite}
	For any complex $\CM \in \Dfm(\A)$
	it holds that $\Gm(\CM) \in \Dfm(\Eta)$. 
\end{lem}
\begin{proof}
	Let $\CP$ be a complex from $\Hot^-(\proj \A)$ quasi-isomorphic to $\CM$.
	Since $\CL$ is bounded, 
	there is $m \in \Z$ such that
	$\Ho^i(\Gm(\CM)) \cong \Hom_{\Hot(\A)}(\CL,\CM[i]) =0$ for any integer $i > m$. Moreover, it holds that $\Ho^j(\Gm(\A)) \in \mathcal{A}$ for any $j \in \Z$ 
	(by, for example, Lemma~\ref{lem/translation}~\eqref{lem/translations4}).
\end{proof}

We formulate the main result of this section.
\begin{prp} \label{prp/Keller}
	The following conditions are equivalent for the
	the presilting complex $\CL$ of the Noetherian $\Rx$-algebra $\A$.
	\begin{enumerate}[label=(a\arabic*)]
		\item \label{silt4} The complex $\CL$ generates the category $\D(\A)$ in the sense that $\Loc(\CL) = \D(\A)$.
		\item \label{silt3} The complex $\CL$ is a weak generator of $\Dfm(\A)$, that is,
		${\langle\CL\rangle}^{\perp} \cap \Dfm(\A) = 0$.
		\item \label{silt1} The complex $\CL$ is a perfect generator, that is, $\langle \CL \rangle = \per \A$.
	\end{enumerate}
	\begin{enumerate}[label=(b\arabic*)]
		\item \label{silt5c}
		The functor $\Fm$ defined in~\eqref{eq/dg-pair}
		yields  an equivalence $\D(\Eta) {\overset{\sim}{\longrightarrow}} \D(\A)$.
		\item \label{silt5b} The functor $\Fm$ restricts to an equivalence
		$\Dfm(\Eta) {\overset{\sim}{\longrightarrow}} \Dfm(\A)$.
		\item \label{silt5a} The functor $\Fm$ restricts to an equivalence
		$\per(\Eta) {\overset{\sim}{\longrightarrow}} \per(\A)$.
	\end{enumerate}
In particular,
the presilting complex
$\CL$ 
is silting if and only if 
it is
a weak generator of the category $\D^-(\md \A)$.
\end{prp}

The analogue
of the last statement
for tilting complexes 
was shown by Rickard 
\cite{Rickard91b}*{Lemma 1.1}. 
The following arguments reformulate a proof which was communicated to me by Bernhard Keller \cite{Keller:private-communication}.
\begin{proof}
	We keep the  previous assumptions on
	$\A$ and $\CL$.
	We will use properties of the functors
	$\Fm$ and $\Gm$ from Theorem~\ref{thm/derived-Morita} without explicit reference.
	
	 The equivalence of~\ref{silt4} and~\ref{silt5c} holds true as $\mathrm{Im}  (\Fm) = \Loc(\CL)$. 
	\ref{silt5c} implies~\ref{silt5b}
	 since $\Fm$ and $\Gm$ restrict to functors
	between the subcategories $\Dfm(\Eta)$ and $\Dfm(\A)$
	according to
	Proposition~\ref{prp/tensor} and Lemma~\ref{lem/Hom-finite}.
	If~\ref{silt5b} is satisfied, the restriction $\Gm|_{\Dfm(\A)}$ reflects zero objects, that is,~\ref{silt3} is satisfied.
	
	To show that~\ref{silt3} implies~\ref{silt4}, let  $\langle \CL \rangle^{\perp} \cap \Dfm(\A) = 0$.
	Since the $\Rx$-algebra $\A$ is Noetherian, it holds that
	$\A \in  \Dfm(\A)$, and thus
	$\Fm\Gm(\A) \in \Dfm(\A)$. 
	Since $(\Fm,\Gm)$ is an adjoint pair,
	there is a distinguished triangle
	$$\begin{td}
	\Fm\Gm (\A) \ar{r}{\varepsilon_\A} \& \A \ar{r} \& \CZ \ar{r} \& \Fm\Gm(\A)[1]
	\end{td}
	$$
	in $\Dfm(\A)$, where $\varepsilon_{\A}$ denotes the counit.
	Because $\Fm$ is fully faithful, the unit $\eta_{\Gm(\A)}$ and therefore $\Gm(\varepsilon_\A)$ are both isomorphisms. 
	So $\CZ \in \ker \Gm \cap \Dfm(\A) = 0$,
	which shows that $\varepsilon_\A$ is an isomorphism.
	Since 
	$\A \cong \Fm\Gm(\A) \in \Loc(\CL)$  and $\Loc(\A) = \D(\A)$ 
	it follows
	that $\Loc(\CL) = \D(\A)$, that is, condition~\ref{silt4} holds.
	
	So far, we have shown the equivalences~\ref{silt4} $\Leftrightarrow$~\ref{silt3} $\Leftrightarrow$~\ref{silt5c} $\Leftrightarrow$~\ref{silt5b}.
	
	\noindent
	An object in $\D(\Eta)$ is compact if and only if it belongs to the category $\per \Eta = \langle \Eta \rangle$.
	The same is true for the ring $\A$. This yields that 
	any equivalence $\D(\Eta) {\overset{\sim}{\longrightarrow}} \D(\A)$	restricts to an equivalence	
	$\per\Eta {\overset{\sim}{\longrightarrow}} \per \A$.
	This shows the implication~\ref{silt5b}$\Rightarrow$\ref{silt5a}.
	
	If~\ref{silt5a} is satisfied,~\ref{silt1} follows because
	$\per \A = \Fm(\langle \Eta \rangle ) =  \langle \Fm(\Eta) \rangle = 
	\langle \CL \rangle $.
	Finally, 
if
	\ref{silt1} holds, that is,  
	$\langle \CL \rangle = \per \A$, implies 
	that $\A \in \langle \CL \rangle$ and ${\langle\CL\rangle}^{\perp}
	\subseteq \langle\A\rangle^{\perp} = 0$, and thus 
	\ref{silt3}.
	
	So all six conditions are equivalent.
	The last statement follows from the fact that $\D^-(\md\A)$ is equivalent to the subcategory $\Dfm(\A)$ of $\D(\A)$.
\end{proof}
At last, we discuss the role of our assumptions in Proposition~\ref{prp/Keller}.
\begin{rmk}\label{rmk/refined-morita}
	We recall that $\CL$ is a presilting complex of the $\Rx$-algebra $\A$.
	\begin{enumerate}
		\item
		Proposition~\ref{prp/Keller} is still true if $\A$ is a Noetherian $\Rx$-algebra 
		without assuming 
		the base ring $\Rx$ to be local or complete.
		\item For an arbitrary $\Rx$-algebra $\A$,
		there is a variation of Proposition~\ref{prp/Keller}, in which
		the categories $\D^-_{\mathcal{A}}(\Eta)$ 
		and $\D^-_{\mathcal{A}}(\A)$ appearing in    
		conditions~\ref{silt3} and~\ref{silt5b}
	are replaced with the categories
		$\D^-(\Eta)$ 
		and $\D^-(\A)$, respectively.
	\end{enumerate}	
\end{rmk}

\section{Silting and tilting bijections for quotients} \label{sec/main}
This section collects the main results of this paper
which assume the conditions in Setup~\ref{setup/main}.

\subsection{Descent of silting and tilting complexes}

an adaptation an argument by Rickard  \cite{Rickard91b}*{Proof of Proposition~3.1} shows that the derived change-of-rings functor reflects weak generators. 
\begin{lem}\label{lem/lift-weak-generator0}
	Let $\CL \in \per \A$ such that $\PCL$ is a weak generator of $\D^-(\md\AI)$.
	Then $\CL$ is a weak generator of $\D^-(\md \A)$.
\end{lem}
\begin{proof}
	Under the assumptions above, let $\CM \in \D^-(\md \A)$ such that 
	$\CL \perp \CM$. Proposition~\ref{prp/keyA}~\eqref{eq/perp} implies that
	$\PCL \perp \PCM$,
and thus
	$\PCM \in \langle \PCL \rangle^{\perp}\cap \D^-(\md \AI) = 0$.
	Since 
	$\FF$ reflects zero objects by Corollary~\ref{cor/derived-Nakayama}, it holds that $\CM \cong 0$. 
\end{proof}

\begin{notation}\label{not/objects}
	In addition to Setup~\ref{setup/main}, we fix an $\Rx$-module $M$ and a module $\sfN$ over the quotient $\RI = \Rx/\ida$.
	We will use the following notation.
	\begin{itemize}
		\item Let $\silt \A$ and $\tilt \A$ denote the sets of isomorphism classes of basic silting and basic tilting complexes of $\A$, respectively.
		The latter has a subset 
		$\tilt^{\sfM} \A$
		defined via basic tilting complexes $\CT$ of $\A$ satisfying
		\[
		\Tor_+^{\Rx}(\End_{\D(\A)}(\CT),\sfM) = 0\,,
		\]
		and another subset $\tilt^*_{\Rx} \A$ 
		defined via basic tilting complexes $\CT$ of $\A$ such that
		$\End_{\D(\A)}(\CT)$ is free as $\Rx$-module.
		Moreover, we set
		$\tilt^{\sfN, \RI} \A \colonequals \tilt^{\sfN} \A \cap \tilt^{\RI} \A$, where $\sfN$ is viewed as an $\Rx$-module.
				\item The sets $\silt \AI$ and $\tilt \AI$ are defined in a similar way.
				We will consider the subset
				$\tilt^{\sfN} \AI$ given by basic tilting
				complexes $\CT$ of $\AI$ such that
				\[
				 \Tor_+^{\RI}(\End_{\D(\Abar)}(\CT),\sfN) = 0\,,\]
				 and the subset
				 $\tilt^*_{\RI} \AI$
				 requiring the endomorphism ring of each tilting complex to be free as $\RI$-module.
	\end{itemize}
These sets have analogues defined by presilting respectively pretilting complexes.	
\end{notation}

\begin{rmk}
	In case $\sfN$ is the residue field $\kk$ of the local ring $\Rx$, we may use 
	the local criterion of flatness (see Theorem~\ref{thm/local-flatness})
	to identify 	$\tilt^{\kk, \B} \A$ with 
	$\tilt^*_{\Rx} \A$,
	and $\tilt^{\kk} \AI$ with $\tilt^*_{\RI} \AI$.
\end{rmk}

\begin{prp}\label{prp/descent}
	Any perfect complex $\CL$ of $\A$ satisfies the  implications
	\begin{align*}
	\begin{td}
	\CL \in \tilt^*_{\Rx} \A \ar[Leftrightarrow]{d} \&[1.5cm] \CL \in \tilt^{\sfN,\B} \A \ar[Rightarrow,xshift=-5pt]{d} \&[1.5cm] 
	\CL \in \silt \A \ar[Leftrightarrow]{d} \\
	\PCL \in \tilt^*_{\RI} \AI \& \PCL \in \tilt^{\sfN} \AI \ \ar[Rightarrow,xshift=5pt,dashed]{u}[swap]{\text{ if $\Rx/\ida$ is Tor-rigid over $\Rx$}}\& 
	\PCL \in \silt \AI\,.
	\end{td}
	\end{align*}
	The same is true for pretilting and presilting versions of the sets above.
\end{prp}
\begin{proof}
	For basic presilting complexes the equivalence on the right follows from 
	Lemma~\ref{lem/summands2} and Proposition~\ref{prp/keyA}.
	
	To show the upward implication on the right, 
	assume that  $\PCL$ is silting.
	Since $\PCL$ is a perfect generator, it is a weak generator of $\D(\AI)$.
	It follows that $\CL$ is a presilting complex by the previous argument
	and
	a weak generator of $\D^-(\md \A)$ by Lemma~\ref{lem/lift-weak-generator0}.
	Thus, $\CL$ is silting by Proposition~\ref{prp/Keller}.
	
	Proposition~\ref{prp/pretilting} applied with $S = \RI$ yields
	the remaining implications for the sets of tilting complexes and their pretilting analogues.
\end{proof}

\subsection{Silting bijections}

The next statement is the main result of this paper.
We repeat the assumptions of Setup~\ref{setup/main} for the convenience of the reader.
\begin{thm}\label{thm/embeddings2}
	Let $\ida$ be a proper ideal of a complete local ring $\Rx$
	and $\A$ a Noetherian $\Rx$-algebra
	such that 
		\begin{align} \label{eq/strong-tor-ind00}
		\tag{$\star$}
		\Tor_+^{\Rx}(\A,\Rx/\ida) = 0\,.
	\end{align}
	As before, we set $\RI \colonequals \Rx/\ida$, $\AI \colonequals \A/\ida \A$ and choose an $\RI$-module $\sfN$.
	Then the following statements hold.
	\begin{enumerate}
		\item \label{thm/embeddings2a} There are embeddings 
	of the sets
	defined in Notation~\ref{not/objects}
	\begin{align}\label{eq/cube**}
	\begin{tikzcd}[ampersand replacement=\&, cells={outer sep=1pt, inner sep=1pt}, column sep=0.4cm, row sep=0.4cm]		
	\&	
	\pretilt^*_{\Rx} \A
	\ar[hookrightarrow, color=light-gray]{rr}	\ar[hookrightarrow]{dd}{f^*_{pt}}	\&[-0.5cm] \&	
	\pretilt^{\sfN,\RI} \A
	\ar[hookrightarrow]{dd}{f^{\sfN}_{pt}}
	\ar[hookrightarrow, color=light-gray]{rr}  \&[-0.5cm] \&
	\presilt \A \ar[hookrightarrow]{dd}{f_{ps}} \&  
	\&	\CL \ar[mapsto]{dd}  \&[-0.4cm]
	\\
	\tilt^*_{\Rx} \A   \ar[hookrightarrow, color=light-gray]{ru} \ar[hookrightarrow]{dd}{
		f^{*}_t} \ar[hookrightarrow, color=light-gray]{rr} \& \&	
	\tilt^{\sfN,\RI} \A   \ar[hookrightarrow, color=light-gray]{ru} \ar[hookrightarrow]{dd}{
		f^{\sfN}_t} \ar[hookrightarrow, color=light-gray]{rr} \& \&
	\silt \A \ar[hookrightarrow]{dd}{f_s} \ar[hookrightarrow, color=light-gray]{ru} \& \\[1cm]
	\&	\pretilt^*_{\RI} \AI \ar[hookrightarrow, color=light-gray]{rr} \& \& 
	{\pretilt^{\sfN} \AI } \ar[hookrightarrow, color=light-gray]{rr} \& \&
	\presilt \AI 
	\& \& \PCL \&
	\\
	\tilt^*_{\RI} \AI \ar[hookrightarrow, color=light-gray]{ru} \ar[hookrightarrow, color=light-gray]{rr} \& \&
	\tilt^{\sfN} \AI \ar[hookrightarrow, color=light-gray]{ru} \ar[hookrightarrow, color=light-gray]{rr} \& \&
	\silt \AI \ar[hookrightarrow, color=light-gray]{ru} \&
	\end{tikzcd}
	\end{align}
	where the map
	$f_s$ is an embedding of posets.
	\item \label{thm/embeddings2b}
	The outer vertical maps $f_t^*, f_{pt}^*, f_{s}$ and  $f_{ps}$ are bijective if 
	\begin{align} \label{eq/strong-tor-ind}
		\tag{$\star\star$}
		\Tor_+^{\Rx}(\A,\Rx/\ida^n) = 0\text{ for any integer }n > 0\,.
	\end{align}	
	\item \label{thm/embeddings2c}
	If  $\RI$ is Tor-rigid as $\Rx$-module
	and the map $f_{ps}$ is bijective, then the embeddings $f_{t}^{\sfN}$ and $f_{pt}^{\sfN}$ are bijective as well 
	and 
	the set
	$\tilt^{\RI} \A$
	coincides with the set of isomorphism class of basic tilting complexes $\CT$ of $\A$ satisfying
	\[\Hom_{\D(\Abar)}(\CT,\CT[-1]) = 0\,.\]
	\end{enumerate} 
\end{thm}
\begin{proof}
	Since $\Tor_{+}^{\Rx}(\A,\RI) = 0$, 
	Proposition~\ref{prp/ps-inj}
	states that $f_{ps}$ is a well-defined injective map, which preserves and reflects relation $\geq$. 

	If~\eqref{eq/strong-tor-ind} holds, any complex $\CP \in \presilt \AI$ has a lift $\CL \in \per \A$ by Corollary~\ref{cor/silting-lifts}, and thus the map $f_{ps}$ is bijective.
	
	The remaining claims follow from  Proposition~\ref{prp/descent} and Corollary~\ref{cor/tor-rigid-endo}.
\end{proof}

\begin{rmk}
	The silting bijection $f_s$ in 
	Theorem~\ref{thm/embeddings2}~\eqref{thm/embeddings2b}
	restricts to an order-preserving bijection between the subset
	given by  two-term silting complexes of $\A$ and the similar subset for $\AI$.
	This is also true 
	without 
the Tor-independence assumptions
\eqref{eq/strong-tor-ind00} or~\eqref{eq/strong-tor-ind} by
	 work of Kimura \cite{Kimura}, 
	which extends a result by Eisele, Janssens and Raedschelders \cite{EJR} for finite-dimensional algebras.
	The connection between this result 
	and Theorem~\ref{thm/embeddings2}~\eqref{thm/embeddings2b}
	will  be explained in subsequent work.
	\end{rmk}
The next corollary was stated in the introduction and describes three setups for applications of the last theorem. These include
orders over certain regular rings,  group algebras and complete intersections.
\begin{cor} \label{cor/setups}
	Let $\ida$ be a proper ideal of the complete local ring $\Rx$ and $\A$ a Noetherian $\Rx$-algebra $\A$ satisfying any of the following conditions.
\begin{enumerate}[label={\rm (S1\alph*)}, align=left]
	\item \label{setup/orders} The $\Rx$-algebra $\A$ is free as an $\Rx$-module and the ring $\Rx$ is regular.
	\item \label{setup/groups} The $\Rx$-algebra $\A$ is free as an $\Rx$-module and the ideal $\ida$ is maximal.	
\end{enumerate}
\begin{enumerate}[label={\rm (S2)}, align=left]
	\item \label{setup/CI} The ideal $\ida$ is generated by an $\Rx$- and $\A$-regular sequence.
\end{enumerate}
	Then the six vertical maps in diagram~\eqref{eq/cube**} are bijective.
\end{cor}
\begin{proof}	
	In case~\ref{setup/orders}, \emph{any} finitely generated $\Rx$-module is Tor-rigid by results of Auslander and Lichtenbaum
	\cites{Auslander, Lichtenbaum}.
	In cases~\ref{setup/groups} and~\ref{setup/CI}, Tor-rigidity of  $\Sxx$ as $\Rx$-module was recovered in Corollary~\ref{cor/recover-tor-rigid}.	
	
	In setups
	\ref{setup/orders} and~\ref{setup/groups},
	it holds that
	$\Tor_{+}^{\Rx}(\A,\Rx/\ida^n)=0$ for any integer $n > 0$.
	In case~\ref{setup/CI}, it holds that $\Tor_+^{\Rx}(\A,\Sxx) = 0$
	as observed in Remark~\ref{rmk/sub}.
	Since each $\Sxx$-module $\ida^n/\ida^{n+1}$ is free by \cite{Eisenbud-Book}*{Exercise 17.16}, Lemma ~\ref{lem/algebra-conormals}~\eqref{tor-relations} yields that 
	~\eqref{eq/strong-tor-ind} is satisfied as well. 
	
	So in any of the three cases all of the additional assumptions of Theorem~\ref{thm/embeddings2} are satisfied.
\end{proof}
Under the assumption ~\ref{setup/groups} and that the ring $\Rx$ has Krull dimension one, 
a variation of the silting bijection $f_s$
in~\eqref{eq/cube**} was established independently by Eisele \cite{Eisele}*{Corollary 6.5}.

\begin{rmk}
	Let $\Rx'$ be a commutative complete local Cohen--Macaulay ring which admits a Noether normalization.
	Assume that $\A$ is an \emph{$\Rx'$-order}, that is, an $\Rx'$-algebra such that $\A$ is a finitely generated and Cohen--Macaulay as an $\Rx'$-module.
	Viewing $\A$ as an algebra over the Noether normalization of $\Rx'$ leads to setup~\ref{setup/orders}.
	\end{rmk}

\subsection{Silting bijections via transitivity}
\label{subsec/transitivity}
In certain situations, we may deduce surjectivity of a silting embedding $f_s$ by composing it with another silting embedding $g_s$ such that the composition 
$g_s \cdot f_s$ is bijective.
This transitivity trick allows to sharpen the bijection results of Theorem~\ref{thm/embeddings2}. For a precise formulation we fix several quotient rings.
\begin{setup}\label{setup/transitive}
	Let $\ida \subseteq \idb \subseteq \mx$ be ideals of the complete local ring $\Rx$
	and $\A$ a Noetherian $\Rx$-algebra.
	We set 
	$\AIx \colonequals \A/\ida \A$ and $\BI \colonequals \A/\idb \A$.
\end{setup}

In particular, there 
are commutative diagrams of rings and categories
\begin{align*}
\begin{tikzcd}[ampersand replacement=\&, cells={outer sep=2pt, inner sep=1pt}, column sep=0.6cm, row sep=0.5cm]		
\Rx \ar{rr}
\ar[twoheadrightarrow]{dd} \ar[twoheadrightarrow,densely dotted]{rd} 
\&[0.1cm]  \&[0.5cm]
\A
\ar[twoheadrightarrow]{dd}
\ar[twoheadrightarrow]{rrd}
\&
\&
\&
\&
\D^{-}(\md \A)
\ar{dd}[swap]{\FF} 
\ar[end anchor=north west]{rd}{\HH \ \cong \ \GG \, \circ \, \FF}
\&
\&
\& 
\\[0.75cm]
\& \Tx  \ar[densely dotted]{rrr}  \& \&
\& \BI
\& 
\& \& \D^{-}(\md \BI) \& \& 
\\
\Sxx
\ar[twoheadrightarrow,densely dotted]{ru} \ar{rr} \& \& 
\AIx \ar[twoheadrightarrow]{rru}	 \& 
\& \&
\& \D^-(\md \AIx)\,.
\ar[end anchor=south west]{ru}[swap]{\GG}
\end{tikzcd}
\end{align*}
To apply the previous results, we show that
\emph{Tor-independence} is a transitive relation 
in the  sense of the first statement below.
\begin{prp}\label{prp/tor-ind}
	Let $\Gamma$ be a Noetherian $\Rx$-algebra. Set $\Gammax \colonequals \Gamma/\ida \Gamma$. 
	Then the following statements hold.
	\begin{enumerate}
		\item \label{prp/tor-ind1}
		Any two of the following conditions 			imply the third one.
		\begin{center}
			\begin{enumerate*}[label={$\mathsf{(T\arabic*)}$}, ref={$\mathsf{(T\arabic*)}$},
				series = tobecont, itemjoin = \quad
				]
				\item \label{tor-ind/C} 	$\Tor_+^{\Rx}(\Gamma,\Sxx) = 0$
				\item \label{tor-ind/V} 
				$\Tor_+^{\Sxx}(\Gammax, 
				\Tx) = 0$
				\item \label{tor-ind/Z} 
				$\Tor_+^{\Rx}(\Gamma, \Tx) = 0$
			\end{enumerate*}
		\end{center}
		\item \label{prp/tor-ind2}
		If $(\Gammax,\Tx)$ is Tor-rigid over $\Sxx$, then 
		\ref{tor-ind/C}  and~\ref{tor-ind/V} are equivalent to
		\ref{tor-ind/Z}.
	\end{enumerate}
\end{prp}
\begin{proof}
	Since $\Tor_+^{\Sxx}(\Sxx,\Tx) =0$, there is a convergent spectral sequence
	$$
	E_{pq}^2 \colonequals \Tor_{p}^{\Sxx}\big(\Tor_{q}^{\Rx}(\Gamma,\Sxx), \ \Tx \ \big) \quad\Rightarrow \quad E_{p+q} \colonequals \Tor^{\Rx}_{p+q}(\Gamma,\Tx)
	$$
	by Theorem~\ref{thm/grothendieck-sequence}.
Condition~\ref{tor-ind/C} implies 
$E^2_{p+}= 0$. The converse is also true by Nakayama's Lemma.
Moreover,
each of the
	conditions 
	~\ref{tor-ind/V} 
	 is equivalent to 
	$E^2_{+0}=0$, and 
	~\ref{tor-ind/Z}
	to 
 $E_{+} = 0$.
	The two claims follow therefore from Proposition~\ref{prp/spectral2}. 
\end{proof}

The next statement sharpens the bijection results of Theorem~\ref{thm/embeddings2} ~\eqref{thm/embeddings2b}.
\begin{thm}\label{thm/transitive}
	In addition to Setup~\ref{setup/transitive}
	assume  that
	$$\Tor_+^{\Rx}(\A,\Sxx) = 0 \quad \text{ and } \quad \Tor_+^{\Rx}(\A,\Rx/\idb^n) =0\text{ for any integer $n > 0$}.$$
	Let 
	$\sfN$ be an $\Rx/\idb$-module.
	Then there is 
	a commutative diagram 
	of embeddings and bijections of sets
	\begin{align}\label{eq/silt-bijections}
	\begin{tikzcd}[ampersand replacement=\&, cells={outer sep=2pt, inner sep=1pt}, column sep=0.4cm, row sep=0.5cm]		
	\&
	\tilt^*_{\Rx} \A \ar[hookrightarrow, color=light-gray]{rr} 
	\ar["\sim" labl, pos=0.66]{dd}[pos=0.66]{h_t^*} 
	\&[-0.4cm] \& \tilt^{\sfN, \Tx} \A \ar[hookrightarrow]{dd}[pos=0.66]{h_t^{\sfN}}
	\ar[hookrightarrow, color=light-gray]{rrrr}
	\&[-0.4cm] 
	\& 
	\&[0.65cm] 
	\&
	\silt \A
	\ar["\sim" labl, pos=0.66]{dd}[pos=0.66]{h_s}
	\\
	\tilt^*_{\Rx} \A  
	\ar["\sim" labl,swap]{dd}[swap]{f^*_t}
	\ar[hookrightarrow, color=light-gray]{rr} 
	\ar[equal]{ru}
	\&  \&
	\tilt^{\sfN,\, \Tx,\, \Sxx
	}  \A   
	\ar[hookrightarrow]{dd}{f^{\sfN,\Tx}_t} 
	\ar[hookrightarrow, color=light-gray]{rr}
	\ar[hookrightarrow]{ru}[swap]{\iota}
	\& \&
	\tilt^{\sfN, \, \Sxx} \A \ar[hookrightarrow]{dd}{f_t^{\sfN}} 
	\ar[hookrightarrow, color=light-gray]{rr}
	\& \&
	\silt \A \ar["\sim" labl]{dd}{f_s}
	\ar[equal]{ru}
	\\[0.75cm]
	\& \tilt^*_{\Tx} \BI  \ar[hookrightarrow, color=light-gray]{rr}  \& \&
	\tilt^{\sfN} \BI  \ar[hookrightarrow, color=light-gray]{rrrr}  \& \& 
	\& \& 
	\silt \BI \,.
	\\
	\tilt^*_{\Rx/\ida} \AIx 
	\ar{ru}[rotate=45, pos=0.66]{\sim}[swap]{g_t^*} \ar[hookrightarrow, color=light-gray]{rr} \& \& 
	\tilt^{\sfN, \, \Tx} \AIx 	\ar[hookrightarrow]{ru}[swap]{g_t^{\sfN}}
	\ar[hookrightarrow, color=light-gray]{rr} \& \&
	\tilt^{\sfN} \AIx
	\ar[hookrightarrow, color=light-gray]{rr}
	\&   \&
	\silt \AIx 	\ar{ru}[rotate=45,pos=0.66]{\sim}[swap]{g_s} 
	\end{tikzcd}
	\end{align} 
	Moreover, the following statements hold.
	\begin{enumerate}
		\item \label{tor-rigid-S} If $\Sxx$ is Tor-rigid as $\Rx$-module, 
		then the maps $f_t^{\sfN,\Tx}$ and
		$f_t^{\sfN}$ are bijective.
		\item \label{tor-rigid-T-S} If $\Tx$ is Tor-rigid as $\Sxx$-module,
		then the maps $\iota$ 
		and $g_t^{\sfN}$ are bijective.
		\item \label{tor-rigid-T-R} \label{tilt-bijection1} If $\Tx$ is Tor-rigid as $\Rx$-module, then
		the map $h_t^{\sfN}$ is bijective.
	\end{enumerate}
	Similar statements are true for presilting and pretilting versions of the maps above.
\end{thm}
\begin{proof}	
	Since
	$\Tor_+^{\Rx}(\A,\Sxx \, \oplus\, \Tx) = 0$, 
	it follows that
	$\Tor_+^{\Rx/\ida}(\AIx, \Tx) =0$ by 
	Proposition~\ref{prp/tor-ind} 
	\eqref{prp/tor-ind1}.
	By
	Theorem~\ref{thm/embeddings2}~\eqref{thm/embeddings2a} and~\eqref{thm/embeddings2b} all maps in diagram~\eqref{eq/silt-bijections} are well-defined and injective, and the maps $h_t^*$ and $h_s$ are even bijective.
	The diagram commutes because there is an isomorphism $\HH \cong \GG\circ \FF$ of functors.
	Therefore, the embeddings $f_t^*$, $g_t^*$, $f_s$ and $g_s$ are bijective as well.
	 This shows the claims on diagram~\eqref{eq/silt-bijections}.
	
	Next, we deduce that $\iota = \id$ assuming
	that $\Tx$ is Tor-rigid as $\Rx/\ida$-module. 
	Let $\CT \in \tilt^{\Tx} \A$. Its endomorphism ring $\Gamma$ is a Noetherian $\Rx$-algebra
	with $\Tor_{+}^{\Rx}(\Gamma,\Tx) = 0$.
	Proposition~\ref{prp/tor-ind}~\eqref{prp/tor-ind2} 
	implies that
	$\Tor_+^{\Rx}(\Gamma, \Sxx) = 0$, and thus $\tilt^{\Tx} \A = \tilt^{\Sxx,\, \Tx} \A$
	which is equivalent to $\iota = \id$.
	
	The other claims in~\eqref{tor-rigid-S},~\eqref{tor-rigid-T-S} and~\eqref{tor-rigid-T-R}
	follow from Theorem~\ref{thm/embeddings2}~\eqref{thm/embeddings2c}.
\end{proof}

\subsection{Applications of silting bijections}
\label{subsec/app-silt-bij}

This subsection is concerned with the consequences of the last theorem.
In the next statements, we focus on $R$-free Noetherian algebras.
In this situation, `silting theory is invariant under change of quotients' in the following sense.

\begin{cor}\label{cor/R-free}
	Assume that the 
	Noetherian $\Rx$-algebra $\A$ is free as $\Rx$-module.
	Set ${\AI}^{\ida} \colonequals \A/\ida \A$, ${\AI}^{\mx}\colonequals\A/\mx\A$ and $k \colonequals \Rx/\mx$, where $\ida$ is any proper ideal and $\mx$ the maximal ideal of the complete local ring $\Rx$.
	Let $\sfN$ be any ${\AI}^{\mx}$-module.
	
	Then there is a commutative diagram of embeddings and bijections of sets
	\begin{align}\label{eq/silt-bijections-R-free}
	\begin{tikzcd}[ampersand replacement=\&, cells={outer sep=2pt, inner sep=1pt}, column sep=0.5cm, row sep=0.5cm]		
	\tilt^*_{\Rx} \A  
	\ar{rd}[pos=0.66]{h_t^*}[rotate=-45, pos=0.75,swap]{\sim}
	\ar["\sim" labl]{dd}{f^*_t} \ar[hookrightarrow, color=light-gray]{rr} 
	\&  \&[-0.2cm]
	\tilt^{\sfN, \Sxx} \A \ar[hookrightarrow]{dd}[swap]{f_t^{\sfN}} 
	\ar[hookrightarrow, color=light-gray]{rr}
	\& \&[0.75cm]
	\silt \A \ar["\sim" labl,swap]{dd}[pos=0.33]{}[swap]{f_s}
	\ar{rd}[pos=0.66]{h_s}[rotate=-45, pos=0.75,swap]{\sim}
	\\[0.75cm]
	\& \tilt {\AI}^{\mx}  \ar[hookrightarrow, color=light-gray]{rrrr}  \& \&
	\& \& \silt {\AI}^{\mx} 
	\\
	\tilt^*_{\Rx/\ida} {\AI}^{\ida} 
	\ar{ru}[rotate=45, pos=0.66]{\sim}[swap]{g_t^*} \ar[hookrightarrow, color=light-gray]{rr} \& \& 
	\tilt^{\sfN} {\AI}^{\ida}
	\ar[hookrightarrow, color=light-gray]{rr}
	\&   \&
	\silt {\AI}^{\ida} 	\ar{ru}[rotate=45,pos=0.66]{\sim}[swap]{g_s} 
	\end{tikzcd}
	\end{align} 
	The map $f_t^{\sfN}$ is bijective if the ring $\Rx$ is regular or if any of the following conditions is satisfied.
	\begin{enumerate}[label=(\arabic*)]
		\item \label{sym1} The $\Rx$-algebra $\A$ is symmetric, that is, 
		there is an isomorphism $\Hom_{\Rx}(\A,\Rx) \cong \A$ of $\A$-bimodules.
		\item \label{sym2} The finite-dimensional $\kk$-algebra ${\AI}^{\mx}$ is symmetric,
		that is, 
		there is an isomorphism $\Hom_{\kk}({\AI}^{\mx},\kk) \cong {\AI}^{\mx}$ of ${\AI}^{\mx}$-bimodules.
	\end{enumerate}
	\begin{enumerate}[label=(3\alph*)]
		\item \label{stiltB} Any silting complex of ${\AI}^{\mx}$ is tilting.
		\item \label{stiltA}
		Any silting complex of $\AI$ is tilting
		and has $\Sxx$-free endomorphism ring.
		\item \label{stiltL}
		Any silting complex of $\A$ is tilting
		and has $\Rx$-free endomorphism ring.
	\end{enumerate}
	Moreover,  the implications~\ref{sym1} $\Rightarrow$~\ref{sym2} $\Rightarrow$~\ref{stiltB} $\Leftrightarrow$~\ref{stiltA}
	$\Leftrightarrow$~\ref{stiltL} hold true.
\end{cor}
\begin{proof}
	Theorems~\ref{thm/embeddings2} and ~\ref{thm/transitive} with $\idb = \mx$
	yield diagram~\eqref{eq/silt-bijections-R-free}.
	
	If the local ring $\Rx$ is regular, 
	the map 
	$f_t^{\sfN}$ is bijective as already observed in Corollary~\ref{cor/setups}~\ref{setup/orders} using results by Auslander and Lichtenbaum.
	
	The implication~\ref{sym1}$\Rightarrow$\ref{sym2} 
	was shown in \cite{gnedin2019calabiyau}*{Proposition~3.2}.
	If the finite-dimensional $\kk$-algebra ${\AI}^{\mx}$ is symmetric, the category $\per {\AI}^{\mx}$ is $0$-Calabi--Yau
	\cite{Happel}*{Chapter I, Section 4.6},
	which shows the implication~\ref{sym2}$\Rightarrow$\ref{stiltB}.
	The remaining claims follow from the commutativity of the diagram~\eqref{eq/silt-bijections-R-free}.
\end{proof}

We return to the example of the first section.
\begin{ex}\label{ex/preproj2}
	Let $\A$ be the preprojective algebra, ${\AI}^{x_1}$ the ribbon graph order and
	${\AI}^{(m_a, m_b)}$ the Brauer graph algebra introduced
	in Example~\ref{ex/preproj1}.
	We recall that the $R$-algebra on $\A$ was defined choosing two integers $m_a, m_b > 0$.
	
	Since the Brauer graph algebra ${\AI}^{(m_a,m_b)}$ is symmetric according to \cite{Schroll}*{Theorem~2.6},
	Corollary~\ref{cor/R-free} recovers the fact that $\tilt_R^* \A = \silt \A$
	and
	yields bijections
	$$
	\begin{tikzcd}[ampersand replacement=\&, cells={outer sep=2pt, inner sep=2pt}]
		\tilt {\AI}^{x_1}  \&
		\tilt \A \ar{l}[swap]{\sim} \ar{r}{\sim}  
		\& \tilt {\AI}^{\ida} \ar{r}{\sim} \&   \tilt {\AI}^{(m_a,m_b)}
	\end{tikzcd}
	$$
	for any proper ideal $\ida$ of the ring $\Rx$
	and any integers $m_a,m_b > 0$.
	
	Work of Burban and Drozd \cite{Burban-Drozd:2004} implies that
	the ribbon graph order ${\AI}^{x_1}$ is derived-tame and provides an explicit combinatorial description of its indecomposable perfect complexes.
	On the other hand, results by Drozd \cite{Drozd:1990} and
	Bekkert, Drozd and Futorny \cite{Bekkert-Drozd-Futorny}
	yield that the preprojective algebra $\A$ 
	and
	each Brauer graph algebra of the form
	${\AI}^{(m_a, m_b)}$ 
	are derived-wild.
	
	Therefore, the tilting theory of a family of rings can be reduced to a single feasible case.
	The same conclusion can be made for the completed preprojective
	algebra of affine type $\widetilde{\mathbb{A}}_n$ for any $n > 0$. 
\end{ex}

\begin{rmk}
	It is possible to prove that the map $f_t^*$ in diagram~\eqref{eq/silt-bijections-R-free}
	is bijective using only 
	Rickard's results \cite{Rickard91a}*{Theorem~2.1} and \cite{Rickard91b}*{Theorem~3.1} together with a transitivity trick. The main arguments are contained in \cite{Gnedin-Iyengar-Krause}*{Proof of Theorem~5.6}.
\end{rmk}
In general, the class of $R$-free Noetherian algebras is not closed under derived equivalences.
The next statement reduces the question whether the tilted algebra of an $R$-free algebra is also $R$-free to the computation of a finite-dimensional morphism space.
\begin{prp}\label{prp/endo-free}
	Assume 
that
	 the Noetherian $\Rx$-algebra $\A$ is free as an $\Rx$-module.
	Then the endomorphism ring of a pretilting complex $\CT$ of $\A$ is free as an $\Rx$-module if and only if 
	\[
	\Hom_{\D(\Abar^{\mx})}(\PCT,\PCT[-1]) = 0, 
	\quad \text{where } \PCT \colonequals \CT \overset{\mathbb{L}}{\underset{\A}{\otimes}} {\AI}^{\mx}\,.
	\]
\end{prp}
\begin{proof}
This 
	follows from Proposition~\ref{prp/pretilting} 
	\eqref{prp/keyA2}
	and the local criterion of flatness (Theorem~\ref{thm/local-flatness}).  
\end{proof}


Next, we 
deduce `multiplicity independence' of silting objects in 
 setup~\ref{setup/CI}, that is, the context of complete intersections.
\begin{cor}\label{cor/regular}
	Let $\idb$ be a proper ideal of the complete local ring $\Rx$ which is generated by a sequence $\mathbf{x} \colonequals (x_1,\ldots, x_{\ell})$ of $\Rx$-regular and $\A$-regular elements.
	Let
	$\ida$ be the ideal of the ring $\Rx$ generated by the sequence
	$\mathbf{x}^m \colonequals (x_1^{m_1}, \ldots, x_{\ell}^{m_{\ell}})$
	for a sequence $m\colonequals(m_1, \ldots, m_{\ell})$ of positive integers.
	As before, set ${\AI}^{\ida} \colonequals \A/\ida \A$ and ${\AI}^{\idb} \colonequals \A/\idb \A$.
	
	Then there is a commutative diagram of embeddings and bijections of sets
	\begin{align}\label{eq/silt-bijections-regular}
	\begin{tikzcd}[ampersand replacement=\&, cells={outer sep=2pt, inner sep=1pt}, column sep=0.4cm, row sep=0.5cm]		
	\&
	\tilt^*_{\Rx} \A \ar[hookrightarrow, color=light-gray]{rr} 
	\ar["\sim" labl, pos=0.66]{dd}[pos=0.66]{h_t^*} 
	\&[-0.4cm] \& 
	\tilt^{\Rx/\idb} \A \ar["\sim" labl, pos=0.66]{dd}[pos=0.66]{h_t}
	\ar[hookrightarrow, color=light-gray]{rrrr}
	\&[-0.4cm] \& \&[0.65cm] \&
	\silt \A
	\ar["\sim" labl, pos=0.66]{dd}[pos=0.66]{h_s}
	\\
	\tilt^*_{\Rx} \A  
	\ar["\sim" labl,swap]{dd}[swap]{f^*_t}
	\ar[hookrightarrow, color=light-gray]{rr} 
	\ar[equal]{ru}
	\&  \&
	\tilt^{\Rx/\idb, \, \Rx/\ida
	}  \A   
	\ar["\sim" labl,swap]{dd}[swap]{f^{\Rx/\idb}_t} 
	\ar[hookrightarrow, color=light-gray]{rr}
	\ar[hookrightarrow]{ru}[swap]{\iota}
	\& \&
	\tilt^{\Rx/\ida} \A \ar["\sim" labl,swap]{dd}[swap]{f_t} 
	\ar[hookrightarrow, color=light-gray]{rr}
	\& \&
	\silt \A \ar["\sim" labl]{dd}{f_s}
	\ar[equal]{ru}
	\\[0.75cm]
	\& \tilt^*_{\Rx/\idb} {\AI}^{\idb}  \ar[hookrightarrow, color=light-gray]{rr}  \& \&
	\tilt {\AI}^{\idb}  \ar[hookrightarrow, color=light-gray]{rrrr}  \& \& \& \& \silt {\AI}^{\idb}\,.
	\\
	\tilt^*_{\Rx/\ida} {\AI}^{\ida}
	\ar{ru}[rotate=45, pos=0.66]{\sim}[swap]{g_t^*} \ar[hookrightarrow, color=light-gray]{rr} \& \& 
	\tilt^{\Rx/\idb} {\AI}^{\ida}	\ar[hookrightarrow]{ru}[swap]{g_t}
	\ar[hookrightarrow, color=light-gray]{rr} \& \&
	\tilt {\AI}^{\ida}
	\ar[hookrightarrow, color=light-gray]{rr}
	\&   \&
	\silt {\AI}^{\ida}	\ar{ru}[rotate=45,pos=0.66]{\sim}[swap]{g_s} 
	\end{tikzcd}
	\end{align} 
	
\end{cor}
\begin{proof}
	Since $\mathbf{x}$ is $\Rx$-regular, the 
	ring $\Rx$ is normally flat along the ideal $\idb$ in the sense of Definition~\ref{dfn/normally-flat} and $\Rx/\idb$ is Tor-rigid as $\Rx$-module (as shown, for example, in Corollary~\ref{cor/recover-tor-rigid}).
	Because $\mathbf{x}$ is also $\A$-regular it holds that $\Tor_+^{\Rx}(\A,\Rx/\idb) = 0$ as recalled in Remark~\ref{rmk/sub}.
	It can be shown by induction on the length $\ell$ that the sequence $\mathbf{x}^{m}$ is $\Rx$- and $\A$-regular as well.
	It follows that
	$\Rx/\ida$ is Tor-rigid as $\Rx$-module
	and $\Tor_+^{\Rx}(\A,\Rx/\ida) = 0$.
Now,	Theorem~\ref{thm/transitive} implies the claims.
\end{proof}
By the silting bijection $g_s$ 
any two quotient algebras 
${\AI}^{\ida}$ and ${\AI}^{\idb}$ defined by different sequences of multiplicities have the `same' silting theory.
For Brauer graph algebras, an incarnation of this phenomenon was shown for two-term silting complexes by Aihara, Adachi and Chan 
\cite{Adachi/Aihara/Chan}*{Proposition~4.7},
and deduced by Eisele \cite{Eisele}*{Theorem~6.6} under
 certain restrictions on the Brauer graph.

\begin{rmk}
	\begin{enumerate}
		\item
		In the setup of Corollary~\ref{cor/regular}
		the ideal $\idb/\ida$ of the quotient ring $\Rx/\ida$ does not have to be generated by an $\Rx/\ida$-regular sequence, and the ring $\Rx$ does not have to be normally flat along the ideal $\idb$ in general.
		In particular, surjectivity of the embedding $g_s$ is not a direct consequence of Theorem~\ref{thm/embeddings2}~\eqref{thm/embeddings2b}.
		\item
		The conclusion of Corollary~\ref{cor/regular}
		is also true for any ideal $\ida$  which is generated by an $\Rx$- and $\A$-regular sequence and contained in the ideal $\idb$.
		In particular, if  $\ida$ is generated by a subsequence $(x_1,\ldots, x_i)$ of $\mathbf{x}$ with $i \leq \ell$, the maps 
		$\iota$ and $g_t$ in diagram~\eqref{eq/silt-bijections-regular} are bijective as well.
	\end{enumerate}
\end{rmk}

%

The last applications of this subsection suggest that 
the assumption in Theorem~\ref{thm/embeddings2}
which leads to silting bijection results
is natural from the viewpoint of completion.
\begin{cor}\label{cor/silt-completion}
For any $n > 0$ set $\Rx_n = \Rx/\ida^n$ and $\A_n \colonequals \A/\ida^n \A$.
	Assume that the proper ideal $\ida$ of the complete local ring $\Rx$ and the Noetherian $\Rx$-algebra $\A$ satisfy the condition
	\begin{align} \label{eq/strong-tor-ind4}
	\tag{$\star\star$}
	\Tor_+^{\Rx}(\A,\Rx_n) = 0\text{ for any integer }n > 0\,.
	\end{align}
	Then the following statements hold.
	\begin{enumerate}
		\item There is a commutative diagram of bijections and sets
	\begin{align*}
	\begin{tikzcd}[ampersand replacement=\&, cells={outer sep=2pt, inner sep=1pt}, column sep=0.5cm, row sep=0.5cm]		
	\silt \A  
	\ar{ddd}[rotate=-90, yshift=5pt, xshift=-5pt]{\sim}[swap]{f^{(n+1)}_s} 
	\ar{rdd}{f^{(n)}_s}[swap,rotate=-30,pos=0.66]{\sim}
	\&  \& \& \CL \ar[mapsto]{ddd} 
	\ar[mapsto]{rdd}
	\&[0.25cm] 
	\\
	\\
	\& \silt \A_n  \& \& \& \CL_n \colonequals 
	\CL \underset{\A}{\overset{\mathbb{L}}{\otimes}} \A_n
	\,. \\
	\ar{ru}[rotate=30,pos=0.66]{\sim}[swap]{g^{(n)}_s}
	\silt \A_{n+1} 
	\& \& \& \CL_{n+1} \ar[mapsto]{ru}
	\end{tikzcd}
	\end{align*}
\item \label{cor/silt-completion2}
	A perfect complex $\CL$ of $\A$
	is silting if and only if  
	there is an isomorphism 
	$
	\CL \cong
	\varprojlim \CL_n
	$ in $\per \A$
	for an inverse system of silting complexes $\CL_n$ of $\A_n$ 
	such that $\CL_{n+1}$ is a lift of $\CL_n$ for any integer $n > 0$.
	\end{enumerate}
\end{cor}
\begin{proof}
	\begin{enumerate}
		\item
	According to Theorem~\ref{thm/embeddings2}~\eqref{thm/embeddings2a}
	the map
	$f^{(n)}_s$ is a well-defined embedding
	for any $n > 0$.
	Since
	$\Tor_{+}^{\Rx}(\A,\Rx_{n+1} \,\oplus\, \Rx_n) =0$,
	it holds that $\Tor_+^{\Rx_{n+1}}(\A_{n+1}, \Rx_n)=0$ by 
	Proposition~\ref{prp/tor-ind}~\eqref{prp/tor-ind1}.
	By Theorem~\ref{thm/embeddings2}~\eqref{thm/embeddings2a} and~\eqref{thm/embeddings2b}, each map $g^{(n)}_s$
	is a well-defined embedding 
	such that $f^{(n)}_s = g^{(n)}_s \cdot f^{(n+1)}_s$ and
	the map $f^{(1)}_s$ is bijective. 
	So the maps $f^{(n)}_s$ and $g^{(n)}_s$ are bijective for any $n > 0$.
	\item
	To show the `only if'-implication of the second claim, let $\CL$ be a silting complex from
$\Hotb(\proj \A)$.
	Since $f^{(n)}_s$ is well-defined, it holds that $\CL_n$ is a silting complex of $\A_n$. Because $\ida \subseteq \mx$  and the ring $\Rx$ is complete local, 
	the complex $\CL$ is $I$-adically complete with respect to the ideal $I \colonequals\A \ida$.
	
	Vice versa, let $(\CL_n)_{n \in \N}$ be a sequence of silting complexes of $\A_n$ which are iterated lifts of each other. 
	According to Lemma~\ref{lem/minimal-lift},
	we may assume that this sequence is given by iterated lifts of minimal complexes $\CL_n$ from $\Hotb(\proj \A_n)$.
	Then the complex $\varprojlim \CL_n$ is a lift of the silting complex $\CL_1$, and thus a silting complex by the right upward implication of Proposition~\ref{prp/descent}.
	This shows the converse. \qedhere
	\end{enumerate}
\end{proof}

There is an analogue of the last result for tilting complexes under a stronger condition.
\begin{cor}\label{cor/tilt}
	Assume that $\Tor_1^{\Rx}(\A,\Rx/\ida) = 0$ and that $\Rx$ is normally flat along the ideal $\ida$.
	Then the following statements hold.
	\begin{enumerate}
		\item
	There is a commutative diagram of bijections of sets
	\begin{align}\label{tilting-completion}
	\begin{tikzcd}[ampersand replacement=\&, cells={outer sep=2pt, inner sep=1pt}, column sep=0.5cm, row sep=0.5cm]
	\tilt^{\Rx/\ida} \A  
	\ar{ddd}[rotate=-90, yshift=5pt, xshift=-5pt]{\sim}[swap]{f_t^{(n+1)}} 
	\ar{rdd}{f^{(n)}_t}[swap,rotate=-30,pos=0.66]{\sim}
	\&  \& \& \CT \ar[mapsto]{ddd} 
	\ar[mapsto]{rdd}
	\&[0.25cm] 
	\\
	\\
	\& \tilt^{\Rx/\ida} \A_n 
	\& \& \& 
	\CT_n \,.
	\\
	\ar{ru}[rotate=30,pos=0.66]{\sim}[swap]{
		g^{(n)}_t
	}
	\tilt^{\Rx/\ida} \A_{n+1} 
	\& \& \& \CT_{n+1} \ar[mapsto]{ru}
	\end{tikzcd}
	\end{align}
\item
	A perfect complex $\CT$ of $\A$
	is contained in $\tilt^{\Tx} \A$ if and only if 
	there is an isomorphism 
	$
	\CT \cong
	\varprojlim \CT_n
	$ in $\per \A$
	for an inverse system of complexes $\CT_n$ from $\tilt^{\Rx/\ida} \A_n$ 
	such that $\CT_{n+1}$ is a lift of $\CT_n$ for any integer $n > 0$.
	\end{enumerate}
\end{cor}
\begin{proof}
	The assumptions above imply condition~\eqref{eq/strong-tor-ind4} according to Remark~\ref{rmk:hierarchy}.
	\begin{enumerate}
		\item
	Applying Theorem~\ref{thm/transitive} to the pair of ideals  $(\ida^{n+1},\ida^n)$ 
	and the $\Rx/\ida^n$-module $\sfN_n \colonequals \bigoplus_{i=1}^n \Rx/\ida^i$
	for each $n > 0$ 
	yields a commutative diagram of embeddings
	\begin{align}
	\label{diag/completions}
	\begin{td}
	\tilt^{\sfN_{n+1}} \A \ar[hookrightarrow]{r}{\iota_n} \ar[hookrightarrow]{d}{f^{(n+1)}_t}
	\& 	\tilt^{\sfN_n} \A \ar[hookrightarrow]{r} 
	\ar[hookrightarrow]{d}{f^{(n)}_t}
	\& \ldots \& \tilt^{\sfN_2} \A \ar[hookrightarrow]{d}{f^{(2)}_t}
	\ar[hookrightarrow]{r}{\iota_1} \& \tilt^{\sfN_1} \A \ar[hookrightarrow]{d}{f^{(1)}_t} 
	\\
	\tilt^{\sfN_{n}} \A_{n+1} \ar[hookrightarrow]{r}[yshift=1pt]{g^{(n,n-1)}_t} 
	\& 	\tilt^{\sfN_{n-1}} \A_n \ar[hookrightarrow]{r} \& \ldots \& \tilt^{\sfN_1} \A_2 \ar[hookrightarrow]{r}{g^{(1)}_t} \& \tilt \A_1\,. 
	\end{td}
	\end{align}
	Lemma~\ref{lem/algebra-conormals}~\eqref{tor-relations}
	implies that
	$\tilt^{\Rx/\ida} \A   = \tilt^{\sfN_n} \A$ for any $n > 0$, 
	and thus all inclusions of the top row above are equalities.
	Moreover, 
	the $\Rx$-module
	$\Rx/\ida$ is Tor-rigid 
	by 
	Proposition~\ref{prp/conormal-rigid}.
	So the map $f_t^{(1)}$ is bijective by Theorem~\ref{thm/embeddings2}~\eqref{thm/embeddings2c}, and thus all maps in the diagram above are bijective.
	For each $n > 0$ the composition 
of the maps in the bottom row 
	is also the composition of the maps
	$$
	\begin{td}
	\tilt^{\sfN_{n}} \A_{n+1}  \ar[hookrightarrow]{r} \& \tilt^{\Rx/\ida} \A_{n+1} \ar{r}{g_t^{(n)}} \& \tilt \A_1\,, 
 \end{td}
	$$		
where the last map is well-defined 
since $\Tor_+^{\Rx}(\A_, \Rx_{n+1} \oplus \Rx/\ida) = 0$.
	Therefore, each set $\tilt^{\sfN_{n}} \A_{n+1}$
	is equal to $\tilt^{\Rx/\ida} \A_{n+1}$ and
	diagram~\eqref{diag/completions} specializes
	 to diagram~\eqref{tilting-completion}.
	\item
	The last claim follows from the arguments in the proof of Corollary~\ref{cor/silt-completion}~\eqref{cor/silt-completion2}. \qedhere
	\end{enumerate}
\end{proof}

\section{Silting embeddings and descent for tensor products of algebras}
\label{sec/embeddings}

In this section, we turn to a more general situation than Setup~\ref{setup/main}
and consider $\Rx$-algebras $\A$ 
over a possibly non-local and non-complete commutative Noetherian ring $\Rx$.
The main results 
are summarized in Subsection 
\ref{subsec:app2}.

\begin{setup}\label{setup/flat}		
In the next two subsections, we fix the following assumptions.
	\begin{itemize}
		\item Let $\Rx$ be a commutative Noetherian ring. Again, we will abbreviate $\otimes_{\Rx}$ with $\otimes$.
		\item Let $\A$ and $(\Gamma_i)_{i \in I}$ be $\Rx$-algebras 
		such that 
		$\Tor_+^{\Rx}(\A,\Gamma_i) = 0$ for each index $i \in I$.
		For each index $i \in I$ we abbreviate
		$$
		\AoB \colonequals \A \otimes \Gamma_i\,.
		$$
		\item Let $\mathcal{A}$ be a \emph{thick} abelian subcategory of $\Md \Rx$ 
		in the sense that for any exact sequence of $\Rx$-modules
		$$ \begin{td} A_1 \ar{r} \& A_2 \ar{r} \& M \ar{r}\& A_4 \ar{r}\& A_5 \end{td}$$
		  with $A_1, A_2,A_4,A_5 \in \mathcal{A}$ it follows that $M \in \Acat$.
		Assume that the $\Rx$-module
		$\A$ is an object of $\mathcal{A}$ and that
		\begin{align}\label{eq/reflect-zero}
			\text{the functor }\begin{td}
	 \mathcal{A} \ar{r} \& {\displaystyle \prod_{i \in I}} \Md \Gamma_i\,,\end{td} \begin{td} M \ar[mapsto]{r} \& {\displaystyle \prod_{i \in I}} M \otimes \Gamma_i
	 \end{td}\text{ reflects zero objects.}
 \end{align}
		\item Let $(\Sx_i)_{i \in I}$ be a family of commutative rings
		such that
		$\Gamma_i$ is an $\Sx_i$-algebra 
		for each index $i \in I$.
	\end{itemize}
For each index $i \in I$ 
we obtain a commutative diagram of rings
	and a functor 
	$$
	\begin{td}
	\&
	\Rx \ar{r} \ar{d} \& \A \ar{d}  \&[-1cm] \& \D^-_{\Acat}(\A) \ar{d}[swap,font=\normalsize]{\FF_i} \& \CM \ar[mapsto]{d}  \&[-1cm] \\
	\Sx_i \ar{r} \& \Gamma_i \ar{r} \& \AoB  \&  \&  \D^-(\AoB) \& 
	\CM_i 
\& \colonequals 
	\CM 
\underset{\A}{\overset{\mathbb{L}}{\otimes}} \, \AoB
	 \,.
	\end{td}
	$$
We redefine the functor $\FF$ by 
	\begin{align}
		\label{eq:F-index}
	\begin{td}
	\FF \colon
	\D^-_{\Acat}(\A) \ar{r} \& {\displaystyle \prod_{i \in I}} \D^-(\AoB), \& \CM \ar[mapsto]{r} \& 
	{\displaystyle \prod_{i \in I}} \CM_i \,.
	\end{td} 
	\end{align}
\end{setup}
The category $\per \A$ is a full subcategory of the triangulated category $\D^-_{\Acat}(\A)$, since the latter contains $\A$.
\begin{rmk}
	In the final applications, the family $(\Gamma_i)_{i \in I}$ of algebras will coincide with the family $(S_i)_{i \in I}$
	of commutative rings, which will be given
	as follows.
	\begin{enumerate}[label={$\mathsf{(\Gamma\arabic*)}$}, ref={$\mathsf{(\Gamma\arabic*)}$}] 
				\item \label{Gamma3} By the local rings $(\Rx_{\px})_{\px \in \Spec \Rx}$, their completions  $(\widehat{R}_{\px})_{\px \in \Spec \Rx}$, or variations thereof where the prime spectrum $\Rx$ is replaced by its maximal spectrum. 
		\item \label{Gamma2} By only one commutative faithfully flat extension $\Sx$ of the ring $\Rx$.
				\item \label{Gamma1} By only one quotient ring $\Rx/\ida$ for an ideal $\ida \subseteq \rad \Rx$ with $\Tor_+^{\Rx}(\A,\Rx/\ida) = 0$.  
		\end{enumerate}
	In the first two cases
	we may choose $\mathcal{A} \colonequals \Md \Rx$, whereas in the last case $\A$ will be assumed to be a Noetherian $\Rx$-algebra and $\mathcal{A}$ will be set to $\md \Rx$.
	\end{rmk}

\subsection{Ascent and descent of presilting complexes revisited}

The next statement yields analogues of the key implications
in Proposition~\ref{prp/keyA} for Setup~\ref{setup/flat}.
\begin{prp}\label{prp/key-imp2}
	For any complexes $\CL \in \per \A$ and $\CM \in \D^-_{\Acat}(\A)$ 
	the following statements hold.
	\begin{enumerate}
		\item \label{prp/key-imp2a} The $\Rx$-module $\Hom_{\D(\A)}(\CL,\CM[i])$ is contained in $\mathcal{A}$ for any integer $i \in \Z$.
		\item \label{prp/key-imp2b} The complexes $\CL$ and $\CM$ satisfy the equivalences
			\begin{align}\label{eq/key-implications2}
			\CL \geq \CM 
			 &\quad \Leftrightarrow  \quad
			\CL_i \geq \CM_i \text{ for any }i \in I,
\\
			\label{eq/perp2}
			\CL \perp \CM  & \quad \Leftrightarrow \quad  \CL_i \perp \CM_i \text{ for any }i \in I. \end{align}
		Moreover, if $\CL \geq \CM$ holds,  
		for each index $i \in I$
		there is a commutative diagram 
		\begin{align}
			\label{eq/hom}
				\begin{td}
			\Hom_{\D(\A)}(\CL,\CM) \ar{rd}[swap]{\eta_i} \ar{rr}{\FF_i} \&  \&
			\Hom_{\D(\AoB)}(\CL_i, \CM_i) \\
			\&
			\Hom_{\D(\A)}(\CL,\CM) \otimes \Gamma_i \ar{ru}[rotate=30]{\sim}[swap]{\gamma_i} \&
		\end{td}
	\end{align}
where $\eta_i$ denotes the unit and $\gamma_i$ 
is an isomorphism of $\Sx_i$-modules.
		\item \label{prp/key-imp2c}
		If $\CL = \CM$ is a presilting complex of $\A$, 
 the map $\gamma_i$ 
		is an isomorphism of $\Sx_i$-algebras for each index $i \in I$.
	\end{enumerate}
\end{prp}
\begin{proof}
	Let $\CL, \CM$ be as above and $i \in \Z$.
	Since $\Ho^i(\CM) \in \Acat$ and $\Acat$ is additive, it holds that $\Hom_{\D(\A)}(P,\CM[i]) \in \Acat$ for any stalk complex $P \in \add \A$.
	Using that
	$\Acat$ is a thick abelian subcategory and 
	$\langle \A \rangle = \per \A$, the first claim follows. 
	
	Let $\CP$ and $\CQ$ be projective resolutions of $\CL$ and $\CM$, respectively.
	Then $\CK\colonequals\Hom^{\bt}(\CP,\CQ) \in \Kom^-_{\Acat}(\Add \A_{\Rx})$.
	Since $\Tor_+^{\Rx}(\A,\Gamma_i) = 0$, 
	Lemma~\ref{lem/tor-commutes} 
	implies that $\Tor_{+}^{\Rx}(\CK,\Gamma_i) = 0$ for each $i \in I$. Moreover, if there is an integer $j \in \Z$ with $\Ho^j(\CK) \otimes 
	\Gamma_i = 0 $ for each $i \in I$ it follows that $\Ho^j(\CK) = 0$ 
	using assumption~\eqref{eq/reflect-zero}.
	This allows to apply variations of Proposition~\ref{prp/keyA1b}
	and Lemma~\ref{lem/translation} to deduce the remaining claims.
\end{proof}
The behavior of silting complexes
under change of rings in case the family $(\Gamma_i)_{i \in I}$ is given by the local rings $(\Rx_{\px})_{\px \in \Spec \Rx}$ was studied previously by Iyama and Kimura. In particular, they have shown 
that the relation $\geq$ is `local' on perfect complexes \cite{Iyama-Kimura}*{Proposition~2.7}.
\begin{cor}
	\label{cor/derived-Nakayama2}
	The following statements hold.
	\begin{enumerate}
		\item \label{cor/derived-Nakayama2a} The functor $\FF$ defined in~\eqref{eq:F-index} reflects zero objects.
		\item \label{cor/derived-Nakayama2b} Any perfect complex $\CL$ of $\A$ such that $\CL_i$ is a weak generator of $\D^-(\AoB)$
		for each index $i \in I$
		is itself a weak generator of
		$\D^-_{\Acat}(\A)$.
	\end{enumerate}
\end{cor}
\begin{proof}
	The first claim follows
	by an adaptation of the proof of Corollary~\ref{cor/derived-Nakayama}
	using~\eqref{eq/perp2} and 
	\eqref{eq/reflect-zero}.
	The second claim follows from the first along the arguments in the proof of Lemma~\ref{lem/lift-weak-generator0}.
\end{proof}

\subsection{Ascent and descent of silting subcategories}

In this subsection we derive three propositions which can be combined to obtain silting embeddings and descent results in certain setups.
Since the perfect derived categories which we consider may not have the Krull-Remak-Schmidt property, we need to review some basic notions.

For any two additive subcategories $\Lcat$ and $\Mcat$ of $\D^-(\A)$ we  write $\Lcat \geq \Mcat$ if 
$\CL \geq \CM$ for any $\CL \in \Lcat$ and $\CM\in \Mcat$.
In particular,
$\add \CL \geq \add \CM$
is equivalent to
$\CL \geq \CM$.
Therefore,~\eqref{eq/key-implications2} 
can be reformulated as follows.
\begin{lem} \label{lem/relation}
	For any additive subcategories $\Lcat$ in $\per \A$ and  $\Mcat$ in $\D^-_{\Acat}(\A)$ it holds $\Lcat \geq \Mcat$ if and only if $f(\Lcat) \geq f(\Mcat)$. \qed
\end{lem}

An additive subcategory $\Lcat$ of $\per \A$ is \emph{presilting} if
any of its objects is presilting.
Such a subcategory is \emph{silting} if $\langle \Lcat \rangle = \per \A$.
Any silting subcategory is the additive hull $\add \CL$ of a silting complex $\CL$ of $\A$ by a result of Aihara and Iyama \cite{Aihara-Iyama}*{Proposition~2.20}. 
The set $\silt_{\Ccat} \A$ of silting subcategories of $\A$ admits a partial order with respect to $\geq$ by Theorem~\ref{thm/AI}.
Similar considerations apply to each ring $\AoB$.

The functor $\FF$ from~\eqref{eq:F-index} induces a map $f$
from the set of
additive subcategories in $\per \A$ 
to the set of additive subcategories in $\prod_{i \in I} \per(\AoB)$ 
via 
\begin{align} \label{eq/subcat-map}
\Lcat \longmapsto 
f(\Lcat)\colonequals \prod_{i\in I}
\add \FF_i(\Lcat) = 
\prod_{i \in I} \add \big(\Lcat 
\underset{\A}{\overset{\mathbb{L}}{\otimes}}  \AoB\big)
\,.
\end{align}
In general,
each subcategory $\FF_i(\Lcat)$ may not be closed under direct summands.
In case $\Lcat$ is given by the additive hull $\add \CL$ of a complex $\CL$ 
we may identify 
$f(\Lcat)$ with $\prod_{i \in I} \add \CL_i$. 

The next statement will provide a criterion for the injectivity of the restriction of the map $f$ to sets of presilting subcategories.
We will say that
the functor $\FF$ is \emph{full on a pair $(\Lcat, \Mcat)$ of subcategories} of $\per \A$ if $\FF$ is full on any pair $(\CL,\CM)$ 
of complexes $\CL \in \Lcat$ and $\CM \in \Mcat$ in the sense of~\eqref{eq/density}.
\begin{lem}\label{lem/uniqueness2}	
	Assume that the functor $\FF$ is full on both pairs $(\Lcat,\Mcat)$ and $(\Mcat,\Lcat)$ of additive subcategories $\Lcat, \Mcat$ of $\per \A$.
	Then $\Lcat = \Mcat$ if and only if 
	$f(\Lcat) = f(\Mcat)$.
\end{lem}
\begin{proof}
	Assume that $f(\Lcat)= f(\Mcat)$. 
	We show that any object $\CL$ from $\Lcat$ is contained in $\Mcat$.
	Because $\Mcat$ is additive, there is an object $\CM \in \Mcat$ and a morphism $
	\FF(\alpha)
	\colon \FF(\CL) \longrightarrow 
	\FF(\CM)
	$
	which has a left inverse $\FF(\beta)$.
	Since $\FF$ reflects isomorphisms, $\beta \alpha$ is an automorphism of $\CL$. So $\alpha$ is a section and $\CL \in \Mcat$. 
	
	Similarly, $\FF(\CM) \in f(\Lcat)$ implies that $\CM \in \Lcat$.
	This shows that $\Lcat = \Mcat$. 
\end{proof}

In contrast to the situation with presilting subcategories,
the map $f$ restricts to a silting embedding by purely formal reasons.
\begin{prp}\label{prp/silt-embedding}
	The map $f$ from~\eqref{eq/subcat-map} restricts to a well-defined embedding of posets 
	\begin{align*}
		\begin{td}
			f_{s}\colon
			\silt_{\Ccat} \A \ar[hookrightarrow]{r} \& {\displaystyle \prod_{i \in I}} \silt_{\Ccat} (\A_i)\,.
		\end{td}
	\end{align*}
\end{prp}
\begin{proof}	
	Let $\Lcat, \Mcat \in \silt_{\Ccat} \A$ 
	such that $f(\Lcat) = f(\Mcat)$.
	As $\A \in \langle \Lcat \rangle$, it follows that $\AoB \in \langle \FF_i(\Lcat) \rangle$ for any $i \in I$.
	Together with Lemma 
	\ref{lem/relation}
	it follows that the restriction $f_{s}$ is well-defined and preserves and reflects $\geq$.
	
	Since $f(\Lcat) \geq f(\Mcat) \geq f(\Lcat)$, it holds that $\Lcat \geq \Mcat \geq \Lcat$.
	The relation $\geq$ is anti-symmetric on $\silt_\Ccat \A$ by Theorem~\ref{thm/AI}. Thus, $f_s$ is injective.
\end{proof}

\begin{dfn}\label{def/descent}
We will say that the embedding $f_s$ is \emph{descent} if  for any additive subcategory $\mathcal{L}$ in $\per \A$ such that $f(\mathcal{L})$ is silting it follows that $\mathcal{L}$ is silting.
The map $f_s$ will be called \emph{weakly descent}
if it satisfies the property above for any additive hull $\add \CL$ of a perfect complex $\CL$ of $\A$.
Similar terminology will be used for other restrictions of the map $f$ from~\eqref{eq/subcat-map}.
\end{dfn}
In different terms, the map $f_s$ is weakly descent if and only if
 $\FF$ reflects the silting property of perfect complexes.
Next, we consider criteria for descent of the silting embedding.
\begin{prp}\label{prp/silt-descent}
	Assume that
	any $\Rx$-module $M$ with $M \otimes \Gamma_i= 0$ for any index $i \in I$ is zero. Then the map $f_s$ is weakly descent.
	
	If, moreover, the family $(\Gamma_i)_{i \in I}$ is given by a single $R$-algebra $\Gamma$, the map $f_s$ is descent.
\end{prp}
\begin{proof}
	Assume that $f(\Lcat)$ is silting and the first condition above.
	Under the additional condition, we view $I$ as an index set with one element.
	In both cases, there exists a complex $\CL \in \Lcat$ such that $\FF(\CL)$ is silting and $f(\Lcat) = \prod_{i \in I} \add \CL_i$. 
	Since $\CL_i$ is a weak generator of $\D(\AoB)$ for each index $i \in I$, 
	Corollary~\ref{cor/derived-Nakayama2}~\eqref{cor/derived-Nakayama2b} implies that
	$\CL$ is a weak generator of
	$\D^-_{\mathcal{A}}(\A)$
	with 
	$\mathcal{A} \colonequals \Md \Rx$.
	Then $\CL$ is silting
	according to 
	Remark~\ref{rmk/refined-morita}. This shows that $\Lcat = \add \CL \in \silt_{\Ccat} \A$.
\end{proof}

The main results of this section
concern the following analogues of the sets of isomorphism classes defined in Notation~\ref{not/objects}
for sets of subcategories.
\begin{notation}\label{not/subcategories}
	In addition to Setup~\ref{setup/flat}
	for each index $i \in I$ we choose an $\Sx_i$-module	
	$\sfN_i$ 
	such that
	$\Tor_{+}^{\Sx_i}(\Gamma_i,\sfN_i) = 0$. 
	We will use the following notation. 
	\begin{itemize}
		\item
		Let $\presilt_{\Ccat} \A$ denote the set of presilting subcategories in $\per \A$. As before, $\silt_{\Ccat} \A$ denotes its subset given by silting subcategories. 
		Similar notations apply to sets of pretilting and tilting subcategories as well as to each ring $\AoB$.
		\item
		Let $\pretilt^{\sfN_i,\Gamma_i}_{\Ccat} \A$
		denote the set of additive subcategories $\Pcat$ in $\per \A$ 
		such that
		any object $\CT$ from $\Pcat$ is pretilting and its endomorphism ring satisfies
		\[
		\Tor_+^{\Rx}(\End_{\D(\A)}(\CT), (\Gamma_i\underset{\Sx_i}{\otimes} \sfN_i) \, \oplus \, \Gamma_i) = 0\,.
		\]
		In case $\Gamma_i$ is a flat $\Rx$-module, we will abbreviate 
		$\pretilt^{\sfN_i,\Gamma_i}_{\Ccat} \A$ with
		$\pretilt^{\sfN_i}_{\Ccat} \A$.
		\item
		Let 
		$\pretilt^*_{\Ccat, \Rx} \A$
		denote the set of pretilting subcategories $\Pcat$ in $\per \A$ such that
		$\End_{\D(\A)}(\CT)$ is flat as $\Rx$-module for each $\CT \in \Pcat$.
		\item 
		Let	$\pretilt^{\sfN_i}_{\Ccat} (\AoB)$
		denote the set of pretilting subcategories $\Pcat$ in $\per (\AoB)$ such that
		any object $\CT$ from $\Pcat$ satisfies 
		\[
		\Tor_+^{\Sx_i}(\End_{\D(\AoB)}(\CT), \sfN_i) = 0\,.
		\]
		\item
		Let $\pretilt_{\Ccat, \Sx_i}^* (\AoB)$ denote the set of tilting subcategories in $\per (\AoB)$ 
		such that the endomorphism ring of each object in $\Pcat$ is flat as $\Sx_i$-module.
	\end{itemize}
	All sets of pretilting subcategories have analogues defined by tilting subcategories.	
\end{notation}
Next, we collect sufficient conditions for ascent and descent of presilting and pretilting subcategories.
\begin{prp}\label{prp/pre-descent}
	Any additive subcategory $\mathcal{L}$ of $\per \A$ satisfies the implications
	\begin{align*}
		\begin{td}
			\Lcat \in \pretilt^*_{\Ccat,\Rx} \A \ar[dashed,Rightarrow,xshift=-5pt]{d}[swap]{
				\text{ 
					\ref{setup/flat2} 
				}
			}
			\&
			\Lcat \in {\displaystyle \bigcap_{i \in I} }\pretilt_{\Ccat}^{\sfN_i,\Gamma_i} \A \ar[Rightarrow,xshift=-5pt, yshift=5pt]{d} \&
			\Lcat \in \presilt_{\Ccat} \A 
			\ar[Leftrightarrow,xshift=0pt]{d}
			\\
			f(\Lcat) \in 
			{\displaystyle \prod_{i \in I}}
			\pretilt^*_{\Ccat,\Sx_i}(\AoB) 
			\ar[Rightarrow,xshift=5pt,dashed]{u}[swap]{\text{
					\ref{globalize-flat}
			}}
			\& f(\Lcat) \in {\displaystyle \prod_{i \in I}} \pretilt_{\Ccat}^{\sfN_i}(\AoB) 
			\ar[Rightarrow,yshift=5pt, xshift=5pt,dashed]{u}[swap]{
				\text{\ref{setup/tor-rigid} }
			}\& 
			f(\Lcat) \in {\displaystyle \prod_{i \in I}} \presilt_{\Ccat} (\AoB)
		\end{td}
	\end{align*}
	where the additional conditions are given as follows.
	\begin{enumerate}[label={$\mathsf{(C\arabic*)}$}, ref={$\mathsf{(C\arabic*)}$}] 
		\item \label{setup/flat2} 
For any flat $\Rx$-module $M \in \mathcal{A}$ it follows that $M \otimes \Gamma_i$ is flat as $\Sx_i$-module for each index  $i \in I$.
\item \label{globalize-flat}
		Any $\Rx$-module $M \in \mathcal{A}$ such that 
	$\Tor_1^{\Rx}(M,\Gamma_i) =0$ and
		$M \otimes \Gamma_i$ is flat as $S_i$-module for each index $i \in I$ is itself flat. 
\item \label{setup/tor-rigid}
		For any $\Rx$-module $M \in \mathcal{A}$ with
		$\Tor_1^{\Rx}(M,\Gamma_i) = 0$ for any index $i \in I$ it follows that
		$\Tor_{+}^{\Rx}(M,\Gamma_i) = 0$ for any index $i \in I$.
	\end{enumerate}
\end{prp}
\begin{proof}
	The equivalence on the right is a special case of
	Lemma~\ref{lem/relation}.
	 For any complex $\CL$ in $\Lcat$ 
	 the $\Rx$-module $\End_{\D(\A)}(\CL)$ is contained in $\mathcal{A}$ by Proposition~\ref{prp/key-imp2}~\eqref{prp/key-imp2a}.
	 With this observation, 
	 the remaining implications follow from slight variations of Propositions 
	~\ref{prp/cohom4} and~\ref{prp/cohom1d} using the translations in Lemma~\ref{lem/translation}.
	\end{proof}

\subsection{Applications to commutative base-change}\label{subsec:app2}

In the final applications below we focus on the case that the family $(\Gamma_i)_{i\in I}$ is given by suitable commutative rings.
We will use 
the terminology of Definition
\ref{def/descent} and 
Notation~\ref{not/subcategories}
 without further reference, but repeat all necessary assumptions in the next statements for the convenience of the reader.

The next statement yields embedding results closely related to Theorem~\ref{thm/embeddings2}~\eqref{thm/embeddings2a}. 
\begin{cor}\label{cor/local-embedding}
	Assume~\ref{Gamma1}, that is, $\A$ is a Noetherian $R$-algebra 
	and
	$\ida$ is a subideal of $\rad \Rx$ with $\Tor^{\Rx}_+(\A,\Rx/\ida) = 0$.
	Let $\sfN$ be an $\Rx/\ida$-module. Set $\AI \colonequals \A/\ida\A$.
	
	Then
	there are
	well-defined injective  maps of sets
	\begin{align*}
		\begin{tikzcd}[ampersand replacement=\&, cells={outer sep=1pt, inner sep=1pt}, column sep=0.5cm, row sep=0.2cm]		
			\&	
			\pretilt^*_{\Ccat,\Rx} \A
			\ar[hookrightarrow, color=light-gray]{rr}	\ar[hookrightarrow]{dd}{f^*_{pt}}	\&[-0.5cm] \&	
			\pretilt_{\Ccat}^{\sfN, \Rx/\ida} \A
			\ar[hookrightarrow]{dd}{f^{\sfN}_{pt}}
			\ar[hookrightarrow, color=light-gray]{rr}  \&[-0.5cm] \&
			\presilt_{\Ccat} \A \ar[hookrightarrow]{dd}{f_{ps}} \&  
			\&	
			\\
			\tilt^*_{\Ccat,\Rx} \A   \ar[hookrightarrow, color=light-gray]{ru} \ar[hookrightarrow]{dd}{
				f^{*}_t} \ar[hookrightarrow, color=light-gray]{rr} \& \&	
			\tilt_{\Ccat}^{\sfN, \Rx/\ida} \A   
			\ar[hookrightarrow, color=light-gray]{ru} \ar[hookrightarrow]{dd}{
				f^{\sfN}_t} \ar[hookrightarrow, color=light-gray]{rr} \& \&
			\silt_{\Ccat} \A \ar[hookrightarrow]{dd}{f_s} \ar[hookrightarrow, color=light-gray]{ru} \& \\[1cm]
			\&	
			{ \pretilt^*_{\Ccat,\Rx/\ida} \AI } \ar[hookrightarrow, color=light-gray]{rr} \& \&
			{\pretilt_{\Ccat}^{\sfN} \AI} \ar[hookrightarrow, color=light-gray]{rr} \& \&
			\presilt_{\Ccat} \AI 
			\& \& 
			\\
			\tilt^*_{\Ccat,\Rx/\ida} \AI \ar[hookrightarrow, color=light-gray]{ru} \ar[hookrightarrow, color=light-gray]{rr} \& \&
			\tilt_{\Ccat}^{\sfN} \AI \ar[hookrightarrow, color=light-gray]{ru} \ar[hookrightarrow, color=light-gray]{rr} \& \&
			\silt_{\Ccat} \AI \ar[hookrightarrow, color=light-gray]{ru} \&
		\end{tikzcd}
	\end{align*}
	where the maps $f_{t}^*, f_{pt}^*, f_s, $ and $f_{ps}$ are descent. If $\Rx/\ida$ is Tor-rigid as $\Rx$-module, 
	the maps $f_t^{\sfN}$ and $f_{pt}^{\sfN}$ are descent as well. 
\end{cor}
\begin{proof}
	The six maps are well-defined by
	Propositions~\ref{prp/silt-embedding} and  
	\ref{prp/pre-descent}. 
	To show that the map $f_s$ is descent, we may repeat the proof of Proposition~\ref{prp/silt-descent} with $\mathcal{A} \colonequals \md \Rx$.
	The remaining claims on descent properties follow from
	the upward implications in Proposition~\ref{prp/pre-descent}, where~\ref{globalize-flat} is satisfied by the local criterion of flatness (Theorem~\ref{thm/local-flatness}).
	
	To show that $f_{ps}$ is injective, let  $\Lcat, \Mcat \in \presilt_{\Ccat} \A$ with $f(\Lcat)=f(\Mcat)$. Then $\Lcat \geq \Mcat \geq \Lcat$
	and $\FF$ is full on $(\Lcat,\Mcat)$ as well as $(\Mcat,\Lcat)$ by 
	diagram~\eqref{eq/hom}.
	Lemma~\ref{lem/uniqueness2} yields that $\Lcat  = \Mcat$.
\end{proof}

\begin{cor}\label{cor/faithful-flat}
	Let $\Rx \longrightarrow \Sx$ be a faithfully flat morphism of commutative Noetherian rings, $\A$ an $\Rx$-algebra and $\sfN$ be an $\Sx$-module.
	Then there are well-defined maps between the sets 
\begin{align}\label{eq/half-embed}
	\begin{tikzcd}[ampersand replacement=\&, cells={outer sep=1pt, inner sep=1pt}, column sep=0.08cm, row sep=0.2cm]		
		\&	
		\pretilt^*_{\Ccat,\Rx} \A
		\ar[hookrightarrow, color=light-gray]{rr}	\ar{dd}{f_{pt}^*}	\&[-0.5cm] \&	
		\pretilt_{\Ccat}^{\sfN} \A
		\ar{dd}{f^{\sfN}_{pt}}
		\ar[hookrightarrow, color=light-gray]{rr}  \&[-0.5cm] \&
		\presilt_{\Ccat} \A \ar{dd}{f_{ps}} \&  
		\&	
		\\
		\tilt^*_{\Ccat,\Rx} \A   \ar[hookrightarrow, color=light-gray]{ru} \ar[hookrightarrow]{dd}{
			f^{*}_t} \ar[hookrightarrow, color=light-gray]{rr} \& \&	
		\tilt_{\Ccat}^{\sfN} \A   
		\ar[hookrightarrow, color=light-gray]{ru} \ar[hookrightarrow]{dd}{
			f^{\sfN}_t} \ar[hookrightarrow, color=light-gray]{rr} \& \&
		\silt_{\Ccat} \A \ar[hookrightarrow]{dd}{f_s} \ar[hookrightarrow, color=light-gray]{ru} \& \\[1cm]
		\&	
		{ \pretilt^*_{\Ccat,\Sx} (\A \otimes \Sx) } \ar[hookrightarrow, color=light-gray]{rr} \& \&
		{\pretilt_{\Ccat}^{\sfN} (\A \otimes \Sx)} \ar[hookrightarrow, color=light-gray]{rr} \& \&
		\presilt_{\Ccat} (\A \otimes \Sx) 
		\& \& 
		\\
		\tilt^*_{\Ccat,\Sx} (\A \otimes \Sx) \ar[hookrightarrow, color=light-gray]{ru} \ar[hookrightarrow, color=light-gray]{rr} \& \&
		\tilt_{\Ccat}^{\sfN} (\A \otimes \Sx) \ar[hookrightarrow, color=light-gray]{ru} \ar[hookrightarrow, color=light-gray]{rr} \& \&
		\silt_{\Ccat} (\A \otimes \Sx) \ar[hookrightarrow, color=light-gray]{ru} \&
	\end{tikzcd}
	\end{align}
	where the maps $f_t^*$, $f_t^{\sfN}$ and $f_s$ are injective and all six maps are descent.
\end{cor}
\begin{proof}
	 We recall that the functor $\mathstrut_- \otimes \Sx\colon \Md \Rx \longrightarrow \Md \Sx$ preserves and reflects flat modules.
	 The claims follow from Propositions~\ref{prp/silt-embedding}, ~\ref{prp/silt-descent} and~\ref{prp/pre-descent}.
	\end{proof}

Corollary~\ref{cor/local-embedding} can be used to show the following variation of the last statement.
\begin{cor}
	Let $\phi\colon \Rx \longrightarrow \Sx$ be a flat morphism of commutative Noetherian rings and $\A$ a Noetherian $\Rx$-algebra such that 
	$\phi$ induces a ring isomorphism
	$\Rx/\ida \cong \Sx/\ida \Sx$
	for a subideal $\ida$ of $\rad \Rx$ with $\Tor_{+}^{\Rx}(\A, \Rx/\ida) = 0$.
	Then
	 all six maps in~\eqref{eq/half-embed} are well-defined, injective and descent. \qed
	\end{cor}

For our next statement we fix the following notation.
Let $\px$ be a prime ideal of $\Rx$ and $n$ a positive integer.
We denote by $\Rx_{\px}$ the localization of $\Rx$ at $\px$, by $\widehat{\Rx}_{\px}$ the completion of $\Rx_{\px}$ at its maximal ideal, and by $\Rx_{n}$ the Artinian ring $\Rx_{\px}/\px^n \Rx_{\px}$.
In particular, $\Rx_1$ is the residue field $k(\px)$.
Given an $\Rx$-algebra $\A$ and an integer $n > 0$, this defines an $\Rx_{\px}$-algebra
 $\A_{\px} \colonequals \A \otimes \Rx_{\px}$,
 an $\widehat{\Rx}_{\px}$-algebra
${\widehat{\A}}_{\px} \colonequals \A \otimes {\widehat{\Rx}_{\px}}$, and an $\Rx_n$-algebra ${\A}_{\px,n} \colonequals \A_{\px} \otimes_{\Rx_{\px}} \Rx_n$.

As a final application, we obtain the following global-to-local picture on silting and tilting theory of an algebra.
\begin{thm}\label{thm/global-to-local}
	In the notations above, 
	for any $\Rx$-algebra $\A$ over a commutative Noetherian ring $\Rx$ 
	and any integer $n > 0$
	there are well-defined embeddings of sets
\begin{align}
	\label{eq/global-local}
	\begin{tikzcd}[ampersand replacement=\&]
		\tilt^*_{\Ccat,\Rx} \A \ar[hookrightarrow]{d}{f_t^*} 
		 \ar[hookrightarrow, color=light-gray]{r}
		\&  		\tilt_{\Ccat} \A \ar[hookrightarrow]{d}{f_t} 
		\ar[hookrightarrow, color=light-gray]{r}	\&
		 \silt_{\Ccat} \A \ar[hookrightarrow]{d}{f_s}
\\
		\underset{\px \in \Spec \Rx}{\prod} \tilt^*_{\Ccat,\Rx_{\px}} \A_{\px} \ar[hookrightarrow]{d}{g_t^*} 
		\ar[hookrightarrow, color=light-gray]{r}
		\& \underset{\px \in \Spec \Rx}{\prod} \tilt_{\Ccat} {\A}_{\px} 
		\ar[hookrightarrow]{d}{g_t}	 
		\ar[hookrightarrow, color=light-gray]{r}
		\&	\underset{\px \in \Spec \Rx}{\prod} \silt_{\Ccat} \A_{\px} \ar[hookrightarrow]{d}{g_s} 
		\\
		\underset{\px \in \Spec \Rx}{\prod}	\tilt^*_{\Ccat, \widehat{\Rx}_{\px}} {\widehat{\A}_{\px}} 
				\ar["\sim" labl,swap,dashed]{d}[pos=0.33]{}[swap]{h_t^*}
		\ar[hookrightarrow, color=light-gray]{r}  \& 
		\underset{\px \in \Spec \Rx}{\prod} \tilt_{\Ccat} {\widehat{\A}_{\px}} \ar[hookrightarrow, color=light-gray]{r}
		\& 
		\underset{\px \in \Spec \Rx}{\prod} \silt_{\Ccat} {\widehat{\A}_{\px}} 
		\ar["\sim" labl,swap,dashed]{d}[pos=0.33]{}[swap]{h_s}
		 \\
		\underset{\px \in \Spec \Rx}{\prod} \tilt^*_{\Rx_n} \A_{\px,n} \ar[hookrightarrow, color=light-gray]{rr} \& \&  \underset{\px \in \Spec \Rx}{\prod} \silt \A_{\px,n}
	\end{tikzcd}
\end{align}
where the maps $f_{t}^*$, $f_t$ and $f_{s}$ are weakly descent, 
  the maps $g_t^*$, $g_t$ and $g_s$ are descent,
	and the bijections $h_t^*$ and $h_s$ are well-defined if the $\Rx$-algebra $\A$ is projective and finitely generated as $\Rx$-module.
These statements remain true if the prime spectrum of $\Rx$ is replaced by its maximal spectrum.
\end{thm}

\begin{proof}
	We recall from commutative algebra that for each prime ideal $\px$ of $\Rx$
	the natural inclusion $\Rx \hookrightarrow \Rx_{\px}$ is flat and that the projection 
	$\Rx_{\px} \twoheadrightarrow \Rx_{n}$ factors through a faithfully flat morphism $\Rx_{\px} \hookrightarrow \widehat{\Rx}_{\px}$ of Noetherian rings.	
		 Since vanishing and flatness of an $\Rx$-module are local properties, the claims on the maps $f_s$, $f_t$ and $f_t^*$ follow from Propositions
		\ref{prp/silt-descent}.
The statements on the maps $g_s$, $g_t$ and $g_t^*$ are an application of Corollary~\ref{cor/faithful-flat}.

In case the $\Rx$-algebra is projective and finitely generated as $\Rx$-module, each category $\per \widehat{\A}_{\px}$ has the Krull-Remak-Schmidt property and we may identify 
 its additive subcategories with 
isomorphism classes of its basic complexes.
In this case, the maps $h_t^*$ and $h_s$ are well-defined and bijective by Corollary~\ref{cor/R-free}.
\end{proof}

The last theorem complements recent work by Iyama and Kimura on silting complexes \cite{Iyama-Kimura}. 
Among other results, they proved that for a Noetherian $\Rx$-algebra $\A$ a complex $\CL$ in $\Db(\md \A)$ is silting if and only if its localization $\CL_{\px}$ is a silting complex of $\A_{\px}$ for each prime ideal $\px$ of the ring $\Rx$ \cite{Iyama-Kimura}*{Theorem~2.18}.
In case $\A$ is also projective over $\Rx$,  there is an analogue of this result for tilting complexes of $\A$ with $\Rx$-flat endomorphism rings 
and tilting complexes over the finite-dimensional $k(\px)$-algebras $(A_{\px,1})_{\px \in \Spec \Rx}$ \cite{Gnedin-Iyengar-Krause}*{Proposition~4.6}. 

At last, we note that some of the previous statements can be extended with respect to non-commutative base-change and have consequences concerning derived equivalences.
\begin{rmk}
	The maps $f_t^*$, $f_t$ and $f_s$ are also well-defined and injective and the maps $f_t$ and $f_s$ are weakly descent in case the localizations $(\A_{\px})_{\px \in \Spec \Rx}$ in ~\eqref{eq/global-local} are replaced by any family $(\A_i)_{i \in I}$ 
	of rings $\A_i \colonequals \A \otimes \Gamma_i$
	such that each ring $\Gamma_i$ is an $\Rx$-flat algebra
	as well as an $\Sx_i$-flat algebra over a commutative ring $\Sx_i$ 
and
	 any $\Rx$-module $M$ with $M \otimes \Gamma_i = 0$ for each index $i \in I$ is zero.
	
	If the $\Rx$-algebra $\A$ is derived equivalent to another ring $A$, the latter is an $\Rx$-algebra as well
	and the rings
	$\A_i$ and $A_i$ are derived equivalent for each index $i \in I$ by \cite{Rickard91a}*{Theorem~2.1} or
	Proposition~\ref{prp/key-imp2}~\eqref{prp/key-imp2c}. 
\end{rmk}
We refer to \cite{Gnedin-Iyengar-Krause}*{Section~5} for further
comparison of global and local aspects of the derived Morita theory of an algebra.

\end{document}